\documentclass{article}
\usepackage{lmodern}
\usepackage[leqno]{amsmath}
\usepackage{amsthm,amssymb,enumitem,graphicx,yfonts,mathrsfs,textcomp,bbold}
\usepackage{xcolor}

\def\spine{1.1in}
\usepackage[heightrounded, twoside, top=1in, bottom=1.3in, inner=\spine, outer=\spine]{geometry} 
\usepackage[hyperfootnotes=false]{hyperref}
\usepackage[final,expansion=true,protrusion=true,spacing=true,kerning=true]{microtype} 
\microtypecontext{spacing=nonfrench}

\def\uline#1{\underline{#1\mkern-6mu}\mkern6mu}
\def\oline#1{\mkern10mu\overline{\mkern-10mu#1}}

\def\epsilon{\varepsilon}

\def\ext{{\xi}}
\def\ker{{g}}
\def\m{\mathcal M} 

\def\restr#1{\lower4pt\hbox{$\Big|$}\lower9pt\hbox{${\scriptstyle #1}$}\kern-1pt} 
\def\pphi{\varphi}
\def\q{\bs{{q}}}
\def\Q{{\bf Q}}

\def\cvt{E_{\text{\tiny Q}}}
\def\geq{\geqslant}
\def\leq{\leqslant}

\def\t{\textswab{t}}
\def\w{\textswab{w}}
\def\l{\textswab{L}}
\def\f{\textswab{F}}
\def\h{{\eta}}
\def\e{\textswab{e}}
\def\E{{\hat{\textswab{e}}}}
\def\s{\textswab{S}}
\def\ee{{\tilde{\e}}}

\def\weakto{{\scalebox{1.4}[1.0]{$ \rightharpoonup $}}}
\DeclareMathOperator*{\glim}{ \Gamma-lim}
\DeclareMathOperator*{\sgn}{sgn}
\DeclareMathOperator*{\dist}{dist}

\newcommand{\R}{\mathbb{R}}
 
\newcommand{\bs}[1]{\boldsymbol{#1}} 
\newcommand{\diam}[1]{\text{\rm diam}(#1)}
\newcommand{\diag}{\text{\rm diag\,}}

\newcommand{\supp}{\text{\rm supp\,}}
\newcommand{\Hd}{\mathcal{H}_d}
\renewcommand{\c}{C_{s,d}}
\renewcommand{\d}{\,{d}}

\makeatletter
\renewenvironment{proof}[1][\proofname]{\par
  \pushQED{\qed}%
  \normalfont \topsep6\p@\@plus6\p@\relax
  \trivlist
  \item[\hskip\labelsep
        \itshape
    \textbf{\textit{#1}}\@addpunct{.}]\ignorespaces
}{%
  \popQED\endtrivlist\@endpefalse
}
\providecommand{\proofname}{Proof}

\let\oldfootnote\footnote
\def\footnote{\@ifstar\footnote@star\footnote@nostar}
\def\footnote@star#1{{\let\thefootnote\relax\footnotetext{#1}}}
\def\footnote@nostar{\oldfootnote}
\makeatother

\newtheorem{theorem}{Theorem}
\newtheorem*{theorem*}{Theorem}
\newtheorem{lemma}[theorem]{Lemma}
\newtheorem{proposition}[theorem]{Proposition}
\newtheorem{corollary}[theorem]{Corollary}
\theoremstyle{definition} 
\newtheorem{definition}[theorem]{Definition} 
\newtheorem{remark}[theorem]{Remark}

\numberwithin{theorem}{section}
\numberwithin{equation}{section}

\newlist{gammalist}{enumerate}{1} 
\def\gammalisttag{$  ^{\Gamma} $}
\setlist[gammalist, 1]{label = \textbf{\arabic*\gammalisttag.}, ref = \textbf{\arabic*\gammalisttag}, start=1} 

\newlist{elist}{enumerate}{1} 
\def\elisttag{$  ^{E} $}
\setlist[elist, 1]{label = \textbf{\arabic*\elisttag.}, ref = \textbf{\arabic*\elisttag}, start=1} 

\title{Asymptotic properties of short-range interaction functionals}
\author{Douglas Hardin \and Edward B. Saff \and Oleksandr Vlasiuk}
\date{}

\begin{document}
\maketitle

\begin{abstract}
    We describe a framework for extending the asymptotic behavior of a short-range interaction from the unit cube to general compact subsets of $ \mathbb R^d $. This framework allows us to give a unified treatment of asymptotics of hypersingular Riesz energies and optimal quantizers. We further obtain new results about the scale-invariant nearest neighbor interactions, such as the $ k $-nearest neighbor truncated Riesz energy. Our generalized approach has applications to methods for generating distributions with prescribed density: strongly-repulsive Riesz energies, centroidal Voronoi tessellations, and a popular meshing algorithm due to Persson and Strang.

\end{abstract}

\footnote*{{\it Date:}\ \today}
\footnote*{2000 {\it Mathematics Subject Classification.}\ {Primary, 31C20, 28A78; Secondary, 52A40}.}
\footnote*{{\it Key words and phrases.}\ {Riesz energy, optimal quantizers, centroidal Voronoi tessellations, equilibrium configurations, covering radius, separation distance, meshing algorithms}.}

\tableofcontents

                                            
                                            
                                            



\section{Introduction} 
\subsection{Discrete energies and their applications}

Behavior of discrete subsets of the Euclidean space $ \mathbb R^d $ is of interest in diverse areas such as information theory, numerical analysis, and discrete geometry, and arises in questions related to optimal quantizers, cubature formulas, meshing algorithms, and packing and covering problems. In many of these areas, a configuration best suited for a given purpose can be found as an extremizer of a certain functional. This characterization allows one to invoke optimization techniques, and is often the preferred approach.

Of the functionals employed in the aforementioned areas, the most amenable to analytical and numerical treatment are those emphasizing interactions between the nearby elements of a discrete set. Thus, for functionals in which short-range interactions dominate, the shape of support can be computed explicitly, which has made them a popular choice for meshing and point distribution, and other numerical modelling applications. In contrast, finding the support of an equilibrium measure of a long-range interaction, for example the harmonic measure, is a notably difficult problem of classical potential theory and PDEs. 

In the present paper we abstract the essential features of short-range interactions, which allow to determine the distribution of minimizers and the asymptotics of minimal values. 
We study interaction functionals $ \e $, depending on a multiset of $ N $ vectors $ \omega_N := \{  x_1, \ldots,  x_N \} \subset \mathbb R^d  $ (we allow repeating vectors), and possibly on a compact set $ A\subset \Omega $, where $ \Omega \subset \mathbb R^d $ is open and fixed. 
We are interested in the minimal value taken by 
\[
    \e(\omega_N, A)
\]
for a set $ A $ and $ \omega_N\subset A $, as well as the configuration $ \omega_N^*(A) $ on which it is attained:
\[
    \e(\omega_N^*(A),\ A) = \min \{ \e(\omega_N, A) : \omega_N \subset A \}.
\]
In particular, $ \e $ is assumed to be lower semicontinuous and $ A $ compact, so that the above minimum exist.

The precise meaning of being a short-range interaction is given in Section~\ref{sec:short_range} in terms of asymptotics of $ \e(\omega_N,A) $ when $ N\to\infty $. 
As an example, one could take $ \e $ to be the potential energy of the configuration $ \omega_N $ by endowing each pair $ x_i, x_j $, $ i\neq j $, with an interaction energy $ g(x_i, x_j)  $, so that $ \e $ is the total energy of interactions between the particles,
$$ \e(\omega_N\!) = \sum_{i\neq j} g(x_i, x_j). $$ 
For short-range interactions that we discuss here, either I) remote particles do not interact, or II) remote particles interact much weaker than those located nearby. Thus, an especially straightforward way to introduce a short-range interaction is to perform the above summation only over the pairs $ x_i, x_j $, for which $ x_j $ is among the nearest neighbors to $ x_i $ in $ \omega_N $. In Sections~\ref{sec:nearest_neighbor} and \ref{sec:CVT} we discuss two examples of interactions that follow this principle: Riesz $ k $-nearest neighbor functionals and quantization errors. These functionals illustrate category I); an instance of a category II) interaction is the hypersingular Riesz energy, given by
$$ \e(\omega_N\!) = \sum_{i\neq j} \|x_i - x_j\|^{-s}, \qquad \omega_N\subset \mathbb R^d,\ s > d. $$ 
Its asymptotic properties have been considered in a series of studies starting with the paper~\cite{hardinMinimal2005} of the first two authors.

\begin{figure}[t]
    \centering
    \includegraphics[width=0.4\linewidth]{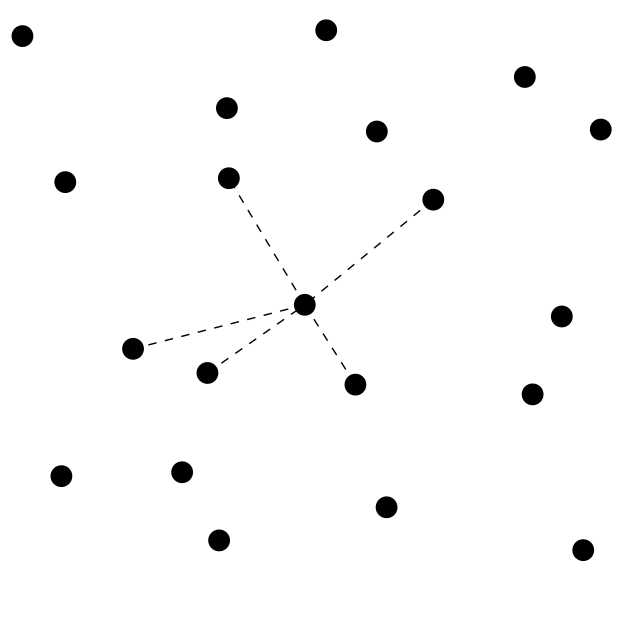}
    \hspace{1cm}
    \includegraphics[width=0.4\linewidth]{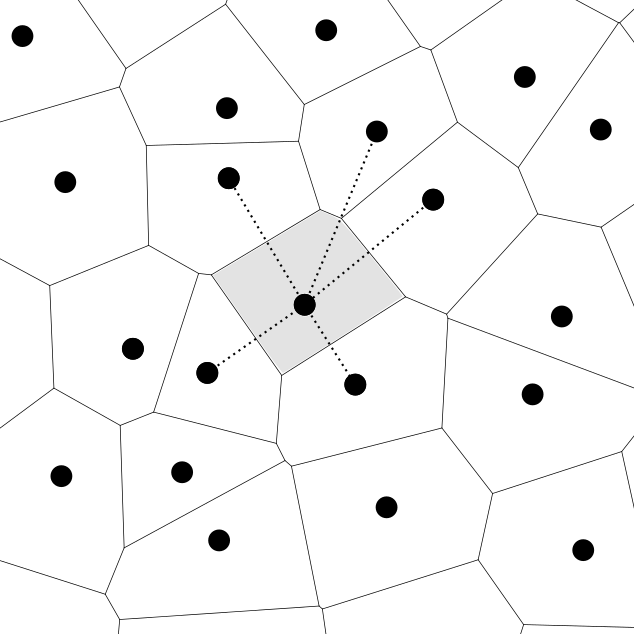}
    \caption{Left: truncated Riesz energy only includes terms for $ k $ nearest neighbors of every $ x\in \omega_N $. Right: quadratic quantization error depends on the second moment of Voronoi cells of such $ x $, the shape of which depends on the position of points with adjacent cells only. Both interactions are short-range.}%
    \label{fig:short_range}
\end{figure}

In our unified context, we deduce the asymptotics and limiting weak$ ^* $ distribution of the minimizers. These results apply in particular to the two aforementioned examples of short-range interactions. This may appear somewhat surprising because the Riesz energy and optimal Voronoi tessellations (the special case of optimal quantizers when the second moments are used) have been until now viewed as unrelated approaches to distributing discrete configurations according to a given density.  We show, in fact, that they can be treated in a similar fashion. 

It should be noted that the fundamental step of our approach, Theorem~\ref{thm:cube}, consists in a scaling argument, which has already been used in the context of optimal quantizers \cite{zadorAsymptotic1982, gruberOptimum2004}, and the context of discrete energy \cite{hardinDiscretizing2004,hardinMinimal2005}. The goal of our paper is to explore the limits of applicability of such scaling approach. To that end we study the short-range property of the relevant interactions in the most general context. We show that the existence of asymptotics on a single cube, which follows by the (sub-)scaling argument, is combined to give the asymptotics on $ d $-rectifiable sets $ A $. Apart from the results for the nearest neighbor Riesz interactions, which are completely new, we generalize the known results of Gruber \cite{gruberOptimum2004} about optimal quantizers by weakening the smoothness assumptions to $ d $-rectifiability.

We will draw parallels with some commonly used meshing algorithms, showing that they also illustrate the short-range interaction paradigm. It is therefore not surprising that the many practical algorithms which yield a discrete sequence with a known distribution are of the short-range type. Indeed, one of the primary motivations for the present study is to achieve better understanding of the point distribution algorithms. It is well-known that Riesz energies have been used for generating well-behaved point distributions, see for example the monograph of Borodachov and the first two authors on this topic \cite{borodachovDiscrete2019}. The same can be said about optimal quantizers, which, in the case of quadratic moments being used, have long being applied for discretization through the use of the famous Lloyd's algorithm; see the review by Du, Faber, and Gunzburger \cite{duCentroidal1999}.

In addition to the sum of pairwise interactions discussed above, $ \e $ can be made dependent on the positions of points in $ \omega_N $:
\begin{equation}
    \label{eq:energy_functional}
    \e(\omega_N\!) = \sum_{i\neq j}  \ker ( x_i,  x_j) + \tau(N) \sum_i \ext( x_i).
\end{equation}
Both $ \ker $ and $ \ext $ must be lower semicontinuous for the minimum to exist. The scaling factor $ \tau(N) $ depends on the growth of $ \ker $ towards the diagonal, and is chosen so that the two sums in the above expression have the same growth rate. 

Finally, it is instructive to mention that functional~\eqref{eq:energy_functional} can be viewed as a discrete analog of the classical variational energy:
\begin{equation}
    \label{eq:integral_operator}
    I(\mu; \ker, \ext):= \iint\limits_{A\times A} \ker(x,  y) d\mu( x)d\mu( y) + \int\limits_A \ext( x) d\mu( x), \quad  \mu \in \mathcal P(A),
\end{equation} 
where $ \mathcal P(A) $ stands for the set of probability measures on $ A $. When $ \ker $ is the harmonic kernel, the properties of such energies and properties of their minimizers have been the focus of studies of a number of authors, in many cases motivated by the connections to approximation theory and theory of analytic functions; examples are the classical works of Frostman, Riesz, Deny \cite{frostman1935potentiel, denyPotentiels1950}, and the later investigations of the locally compact spaces by Fuglede and Ohtsuka \cite{fugledeTheory1960, ohtsukaPotentials1961}. More recently, there has been a revival of interest to the quadratic functionals on probability measures of the above form, in particular due to the connections to optimal transport, information theory and PDEs \cite{chafaiConcentration2016, carrilloGlobal2014a,carrilloRegularity2016, carrilloGeometry2017,limIsodiametry2019,finsterSupport2013}. In several studies the setting of the problem of minimization of \eqref{eq:integral_operator} has been extended from optimization over a single supporting set $ A $ to several, possibly infinitely many, ``condenser plates", each containing a prescribed fraction of the probability measure \cite{harbrechtRiesz2012,zoriiEquilibrium2013,zoriiNecessary2014, fugledeAlternative2018}.  
In the aforementioned literature, kernel $ \ker $ is assumed to be integrable with respect to the Hausdorff measure on $ A $. 

A separate line of research deals with the behavior of the discrete energy~\eqref{eq:energy_functional} when the kernel $ \ker $ is not integrable on $ A $ \cite{hardinMinimal2005, borodachovAsymptotics2008,Hardin2012,borodachovLow2014}. Since there is no corresponding operator of the form~\eqref{eq:integral_operator} in this case, the relevant tools to study the asymptotics of \eqref{eq:energy_functional} are no longer variational, but rather coming from geometric measure theory. 
In Section~\ref{sec:gamma}, we show that a continuous analog of a short-range interaction can be introduced, using the notion of $ \Gamma $-convergence.

\subsection{Notation and basic definitions}
By $ \|\cdot \| $ we usually denote the Euclidean norm in $ \mathbb R^d $, but the following arguments apply to any fixed norm.
Weak$ ^* $ convergence of a sequence of measures $ \mu_n$, $ n\geq 1 $, to $ \mu $ is denoted by $ \mu_n \weakto \mu $, $ n\to \infty $. A ``cube'' always refers to a cube with sides parallel to the coordinate axes.
The unit cube in $ \mathbb R^d $, centered at the origin, is denoted by $ \q_d $. The class of finite unions of cubes with disjoint interiors, contained in $ \Omega $, is called $ \mathscr Q $.
The kernel of our interaction, as well as the confining potential are assumed to be defined on an open $ \Omega \subset \mathbb R^d $, with $ \Omega $ fixed. The $ d $-dimensional Lebesgue measure is denoted by $ \lambda_d $.

It is assumed that $ d' $ is integer with $ d' \geq d $.
A set $ A \subset \mathbb R^{d'} $ is called $ d $-rectifiable if for some compact $ A_0 \subset \mathbb R^d $ and a Lipschitz map $ \mathcal F $ there holds $ A = \mathcal F(A_0)  $. The $ d $-dimensional Hausdorff measure is denoted by $ \mathcal H_d $; is it normalized so as to coincide with $ \mathcal \lambda_d $ on isometric embeddings from $ \mathbb R^d $ to $ \mathbb R^{d'} $. Notation $ \mathcal M_d $ stands for the $ d $-dimensional Minkowski content in $ \mathbb R^{d'} $. A {$d$-regular set} $ A $ satisfies $ cr^d \leq  \mathcal H_d (A\cap B(x,r)) \leq Cr^d $ for every $ x\in A $, and $ 0\leq r \leq \diam A $.

\subsection{Overview of results}
\subsubsection{General short-range functionals}
We shall always assume that functionals $ \e(\omega_N,A) $ are lower semicontinuous in the configuration $ \omega_N\subset \Omega $ for a fixed compact $ A\subset \Omega $ (we use the product topology on $ \Omega^N \subset (\mathbb R^d)^N $). They also must satisfy certain monotonicity properties in the set $ A $, namely, for  compact $ A,B \subset \Omega $ there must hold
\[
        \liminf_{N\to\infty} \sgn\varsigma \cdot \frac{\e(\omega_N,A)}{\t(N)}   \geq
        \limsup_{N\to\infty} \sgn\varsigma \cdot\frac{\e(\omega_N,B)}{\t(N)}      \qquad\text{for}\quad  A\subset B,
\]
where $ \varsigma \neq 0 $ is a constant associated to $ \e $, prescribing the direction of monotonicity.
 We refer to such $ \e $  as set-monotonic.

We are interested in the asymptotics of $ \e(\omega_N^*(A)) $ when $ N\to\infty $, specifically, 
\[
    \l_\e(A) := \lim_{N\to\infty} \frac{\e(\omega_N^*(A))}{\t(N)}, \qquad A\subset \mathbb R^d, \text{ compact. }
\]
Note that all the limits mentioned in the present discussion will be shown to exist.
The main result of this paper for general $ \e $ is that, assuming $ \e $ is short-ranged and the asymptotics of $ \e(\omega_N^*(A)) $ are stable under small perturbation of $ A $, existence of the asymptotics of $ \e $ on cubes in $ \mathbb R^d $ implies their existence on all compact sets $ A\subset \mathbb R^d $. The assumptions on $ \e $ are made precise in Section~\ref{sec:short_range}, see Theorems~\ref{thm:general_uniform} and~\ref{thm:weak_star_limit}; here we give an outline of the obtained statement. 
\begin{theorem}
    \label{thm:main1}
    Suppose $ \e $ is a lower semicontinuous set-monotonic functional, and $ \t(N)  $ is a monotonic function, such that $ \lim_{N\to\infty} \t(tN)/ \t(N) = \w(t) $ for a strictly convex $ \w $, and any $ t>0 $ from a set of positive Lebesgue measure. If\\
    (a) $ 0 < \l_\e(x+a\q_d) < +\infty $  exists for any $ x+a\q_d \subset \Omega $ and depends on $ a $, but not x;\\
    (b) $\e$ is short-range;\\
    (c) $ \l_\e $ has a continuity property on cubes,

    \noindent then\\
    (A) $\w = t^{1+\sigma}$ where $ \sgn \sigma \cdot \t(N) $ is increasing,  and $ \sigma \in (-\infty, -1)\cup (0,\infty) $;\\
    (B) $ \l_\e(A) = \l_\e(\q_d) \cdot (\lambda_d(A))^{-\sigma} $ for any compact $ A \subset \Omega $;\\
    (C) the counting measures of minimizers $ \omega_N^*(A) $ converge weak$ ^* $ to the probability measure proportional to $ \lambda_d $ on $ A $, for any compact  $ A\subset \Omega $ with $ \lambda_d(A) > 0 $.
\end{theorem}
Note that the value of $ \l_\e(A) $ is determined only by $ \lambda_d (A) $, in particular it is translation-invariant. It is often of interest to consider $ \e $ for which this value depends on the position of $ A $. Such effect can be achieved either by adding a weight $ \h $ or a external field $ \xi $ to $ \e $, so that for the resulting $ \e_{\h,\xi} $, one has a mean value-type property that
\[
    \lim_{N\to\infty} \frac{\e_{\h,\xi}(\omega_N^*(x+a\q_d))}{\t(N)} =  \h(y)\frac{\l_\e(\q_d)}{a^{d\cdot \sigma}} + \xi(z), \qquad \text{for some } y,z\in x+a\q_d.
\]
We discuss this generalization in Section~\ref{sec:translation_dependent}, in particular obtaining the expression of the limiting density for the minimizers of the weighted interaction. In the following section we give explicit examples of how the weight and confining potential can be introduced, with $ \e $ given by the nearest neighbor Riesz energies.

Another direction in which Theorem~\ref{thm:main1} can be extended is exploring the asymptotics of $ \e(\omega_N^*(A)) $ when the set $ A $ of Hausdorff dimension $ d $ is embedded in $ \mathbb R^{d'} $ with $ d' > d $. With some straightforward modifications, results similar to Theorem~\ref{thm:main1} can be obtained. We discuss this direction in Section~\ref{sec:embedded}.

\subsubsection{Riesz nearest neighbor energies}
In the present paper we extend the approach used previously for the non-integrable {\it hypersingular} Riesz kernels, and use the general Riesz kernel
$$ \ker_s( x,  y) := \| x -  y \|^{-s}, \qquad s > 0, $$
with the nearest neighbor truncation, which ensures that the resulting energy functional only takes into account short-range interactions when $ N\to \infty $. 
In Section~\ref{sec:nearest_neighbor}, we define the $ k $-nearest neighbor truncated energy as follows:
\begin{equation*}
    E^k_s(\omega_N; \kappa, \ext) = \sum_{i=1}^N \sum_{j\in I_{i,k} } \kappa( x_i,  x_j) \|{x}_i - {x}_j \|^{-s} +  N^{s/d} \sum_{i=1}^N \ext( x_i).
\end{equation*} 
Here $ I_{i,k} $ stands for the set of indices of $ k $ nearest neighbors of $ x_i $ in $ \omega_N $, ordered by nondecreasing distance to $ x_i $. The motivation for this functional is similar to that of using the hypersingular Riesz kernels with $ s >  d = \dim_H A $. Minimizers of such energies have been previously used to discretize continuous distributions with good practical results \cite{vlasiukFast2018}. We establish the following analog of the Poppy-seed bagel theorem \cite[Thms. 2.1--2.2]{hardinMinimal2005} as a consequence of the general results in the previous section.  
\begin{theorem}
    \label{thm:k_asympt}
    Suppose $ A \subset \mathbb R^{d'} $ is a compact $ d $-rectifiable set.
    For a fixed $ k\geq 1 $ and $ s > 0 $, any sequence of minimizers of the $ k $-truncated Riesz $ s $-energy with kernel $ \kappa $ and external field $ \ext $ has the energy asymptotics given by
    \[
        \begin{aligned}
            \lim_{N\to\infty} \min_{\omega_N\subset A} \frac{E^k_s(\omega_N; \kappa, \ext)}{N^{1+s/d}}  
            &= \c^k\int_A \kappa(x,x) \pphi( x)^{1+s/d}\d\Hd(x) + \int_A \ext( x) \pphi( x) \d\Hd(x)\\
            &=: \mathfrak S(s,k,\kappa,\ext),
    \end{aligned}
    \]
    where
    \begin{equation*}
        \pphi( x) = \left(\frac{L_1-\ext( x)}{\c^k(1+s/d)\kappa( x,  x)}\right)^{d/s}_+,
    \end{equation*} 
    and the constant $ L_1 $ is chosen so that $ \pphi \d\Hd $ is a probability measure.
    Furthermore, if $ \mathcal H_d(A) > 0 $, for any sequence of configurations $ \{ \omega_N \} $ achieving the energy asymptotics of  $ \mathfrak S(s,k,\kappa,\ext) $, the weak$ ^* $-limit of the corresponding counting measures is $ \pphi \d\Hd $.
\end{theorem}
This result has the rather remarkable consequence that in order to recover a given distribution, it suffices to minimize an energy functional with the Riesz kernel, that only takes into account interactions with the nearest neighbor! Of course, this requires the knowledge of the $ C_{s,d}^k $ constant, but its value can easily be approximated numerically, and is stable with respect to the computation error. We thus justify the application of gradient flow to nearest neighbor truncation of the Riesz energy as a means to obtain a prescribed distribution \cite{vlasiukFast2018}. An interesting question is how the choice of $ k $ influences the speed of convergence of the gradient flow; it is reasonable to expect that very small values of $ k $ will negatively impact convergence, whereas for large $ k $ the problem becomes almost as computationally expensive as the complete $ N $-body interaction, and so a certain intermediate $ k $ must be optimal. We currently don't know a way to determine such an optimal $ k $.

\subsubsection{Optimal quantizers} 
The so-called optimal quantizers of a fixed measure $ \mu $ can be viewed as the best approximation to it among the measures, supported on at most $ N $ points.
To define the appropriate notion of proximity, recall the $ L_p $-Wasserstein metric. For Borel probability measures $ \mu_1, \mu_2 $, the $ L_p $-Wasserstein distance between them is the optimal transport distance with cost $ \|x-y\|^p $, that is,
\[
    \rho_p(\mu_1, \mu_2) := \inf_{\pi\in\Pi} \left(\int_{\mathbb R^d} \|x-y\|^p \, d\pi(x,y) \right),
\]
where $ \Pi $ is the set of probability measures $ \pi $ on $ \mathbb R^d\times \mathbb R^d $, such that their marginals are $ \mu_1 $ and $ \mu_2 $, respectively. An optimal quantizer of $ \mu $ is a discrete measure $ \nu $, supported on at most $ N $ points, minimizing $ \rho_p(\mu,\nu) $.

While discussing the properties of such $ \nu $ is Section~\ref{sec:CVT}, we will however use an equivalent characterization.
It can be shown \cite[Lemma 3.4]{grafFoundations2000} that the solution to the problem of approximating $ \mu $ with discrete measures supported on a collection of points $ \omega_N  $ coincides with the solution to the problem of minimizing the quantization error of $ \mu $, given by
\[
    \cvt(\omega_N) = \int_{\mathbb R^d} \min_i \|y-x_i\|^p \d\mu(y),\qquad \omega_N = \{ x_1,\ldots, x_N \},
\]
over all $ N $-point configurations in $ \mathbb R^d $. 
We consider a slightly different formulation of the optimal quantization problem in the case of $ \mu = \mathbb 1_A \cdot \lambda_d $, assuming additionally that optimization is performed over $ \omega_N \subset A $. Asymptotically the two formulations coincide. We also discuss weighted quantization with $ \mu = \h\, \mathbb 1_A \cdot \lambda_d $ for a weight $ \h $.

Our methods recover the known results on optimal quantizers, such as the asymptotics of the minimal values of the quantization error, and weak$ ^* $ distribution of optimal quantizers. Furthermore, we improve the results of Gruber~\cite{gruberOptimum2004} in the embedded case, generalizing them to $ d $-rectifiable sets from smooth manifolds in $ \mathbb R^{d'} $. Our result below also includes the external field term, which previously has not been considered for optimal quantizers.
Remarkably, the problem of optimal quantization is quite similar to the problem of optimization of the hypersingular Riesz energy for $ s>d $, or to that of optimization of nearest neighbor-truncated Riesz energies. These problems arise for different ranges of the scaling power $ \sigma $ in Theorem~\ref{thm:main1}, positive for Riesz energies and negative for optimal quantizers, but can be handled by the same methods. The underlying properties in both ranges however are those of short-range functionals. It appears to be the first time this has been observed; the two types of problems previously were considered separately.

The statement of the main theorem of Section~\ref{sec:CVT} is as follows.
\begin{theorem}
    Suppose $ p >0 $, $ A \subset \mathbb R^{d'} $ is a $ d $-rectifiable $ d $-regular set, and $ \cvt $ is the quantization error with continuous weight $ \h\geq h_0 > 0 $, and continuous external field $ \xi\geq0 $:
    \[
        \cvt(\omega_N; \h,\xi) = \int_{A} \h(y) \min_{i} \|y-x_i\|^p \,d\mathcal H_d(y) + N^{p/d-1} \sum_{i=1}^N \xi(x_i),\qquad A\subset \mathbb R^{d'},\ d'\geq d.
    \]
    Then 
    \[
        \lim_{N\to\infty} {N^{p/d}}{{\e_{\h,\xi}}(\omega_N^*( A))} =
        c_{p,d} \int_{A} \h(x)\,\pphi(x)^{-p/d}\, d\mathcal H_d(x) + \int_{ A} \xi(x) \, \pphi(x)\, d\mathcal H_d(x),
    \]
    where $ \pphi $ is the density of a probability measure $ \mu $ supported on $  A $, and is given by    
    \[
        \pphi(x) = \frac{d\mu}{d\mathcal H_d}(x) = \left(\frac pd \cdot\frac{c_{p,d}\h(x)}{\xi(x)-L_1}\right)^{\frac{d}{p+d}}_+
    \]
    for a normalizing constant $ L_1 $. Furthermore, if $ \mathcal H_d(A) > 0 $, the weak$ ^* $-limit of the counting measures for any sequence of configurations attaining the optimal asymptotics is $ \pphi \d\Hd $
\end{theorem}

\subsubsection{A meshing algorithm as a short-range interaction}
As an illustration of the wide applicability of our approach, Section~\ref{sec:meshing} discusses a way to characterize the distribution of the equilibrium configurations, obtained using a meshing algorithm of Persson and Strang \cite{Persson2004a}. This algorithm consists in generating a configuration as a stable point of the Euler algorithm for a certain ODE, and then obtaining a mesh by computing the Delaunay triangulation of this configuration. The relevant ODE arises from modeling the edges of the Delaunay triangulation graph as elastic springs, obeying a modified Hooke's law, see Figure~\ref{fig:persson_strang}, where a weighted distribution is demonstrated. Somewhat counterintuitively, as we discuss in Section~\ref{sec:meshing}, the equilibrium configuration corresponds to a maximum, not minimum, of a certain functional. To define this functional, let $ T_i $ be the indices of points connected to $ x_i $ by the edges of the Delaunay graph $ (V,E) = (\omega_N, E) $:
\[
    T_i = \{ j: (x_i, x_j) \in E \},
\]
and denote the average squared edge length of this graph by
\[
    m_2 (\omega_N) =  \frac1{2\#E} {\sum_{i=1}^N\sum_{j\in T_i} \|x_j-x_i\|^2}.
\]
Clearly, $ 2\# E = \sum_{i=1}^N \#T_i $.
\begin{figure}[t]
    \centering
    \includegraphics[width=0.3\linewidth]{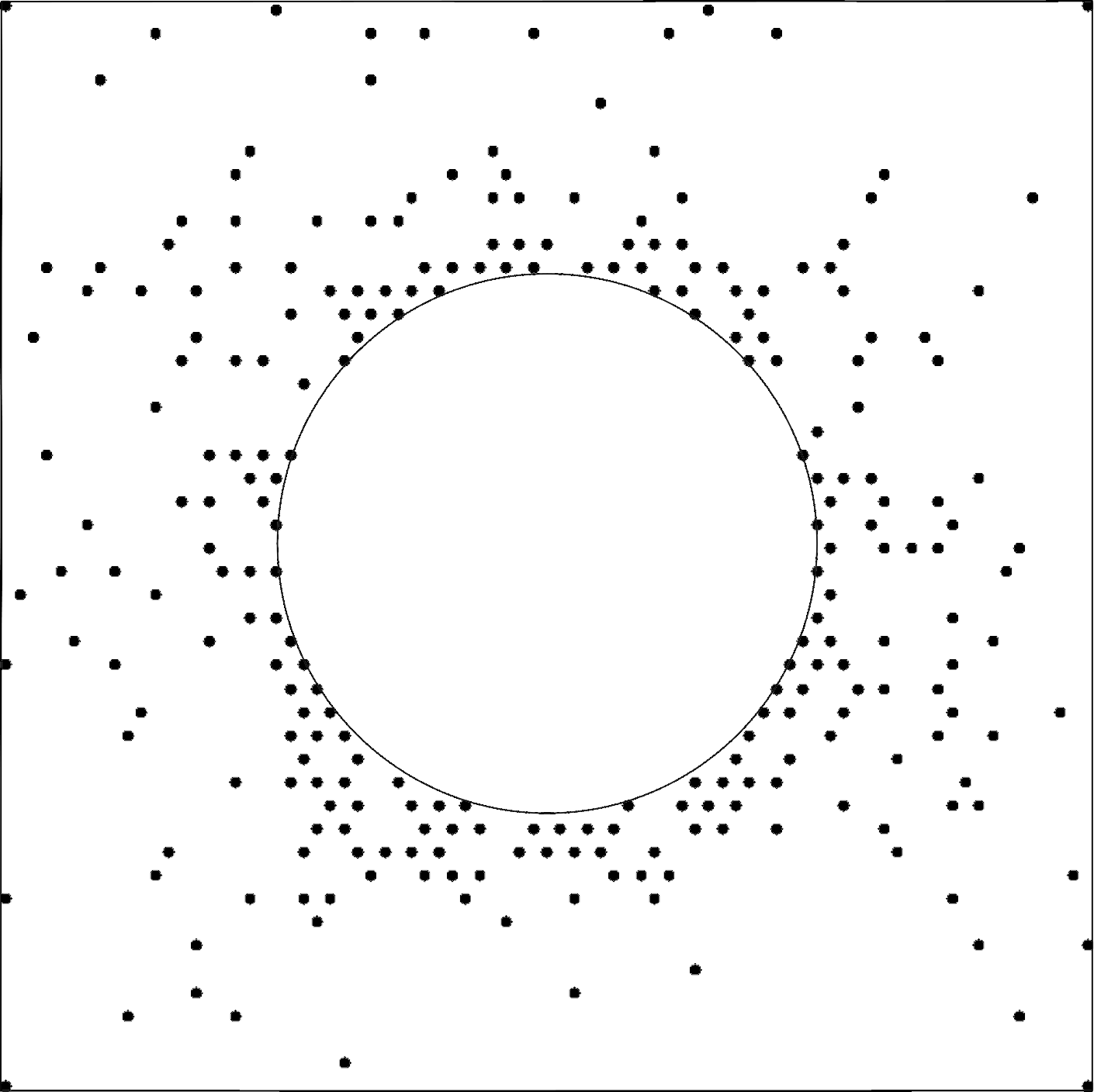}
    ~
    \includegraphics[width=0.3\linewidth]{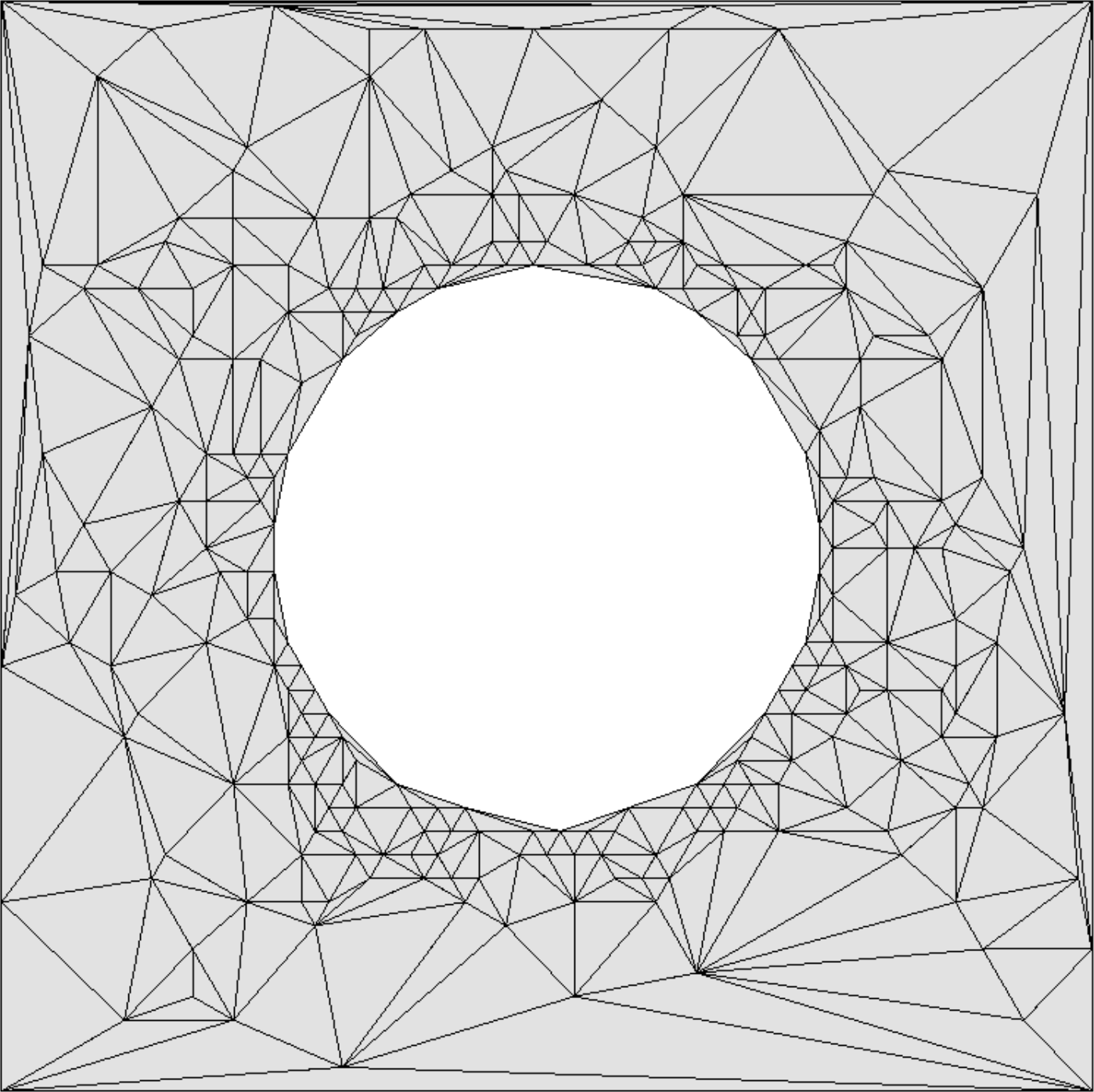}
    ~
    \includegraphics[width=0.3\linewidth]{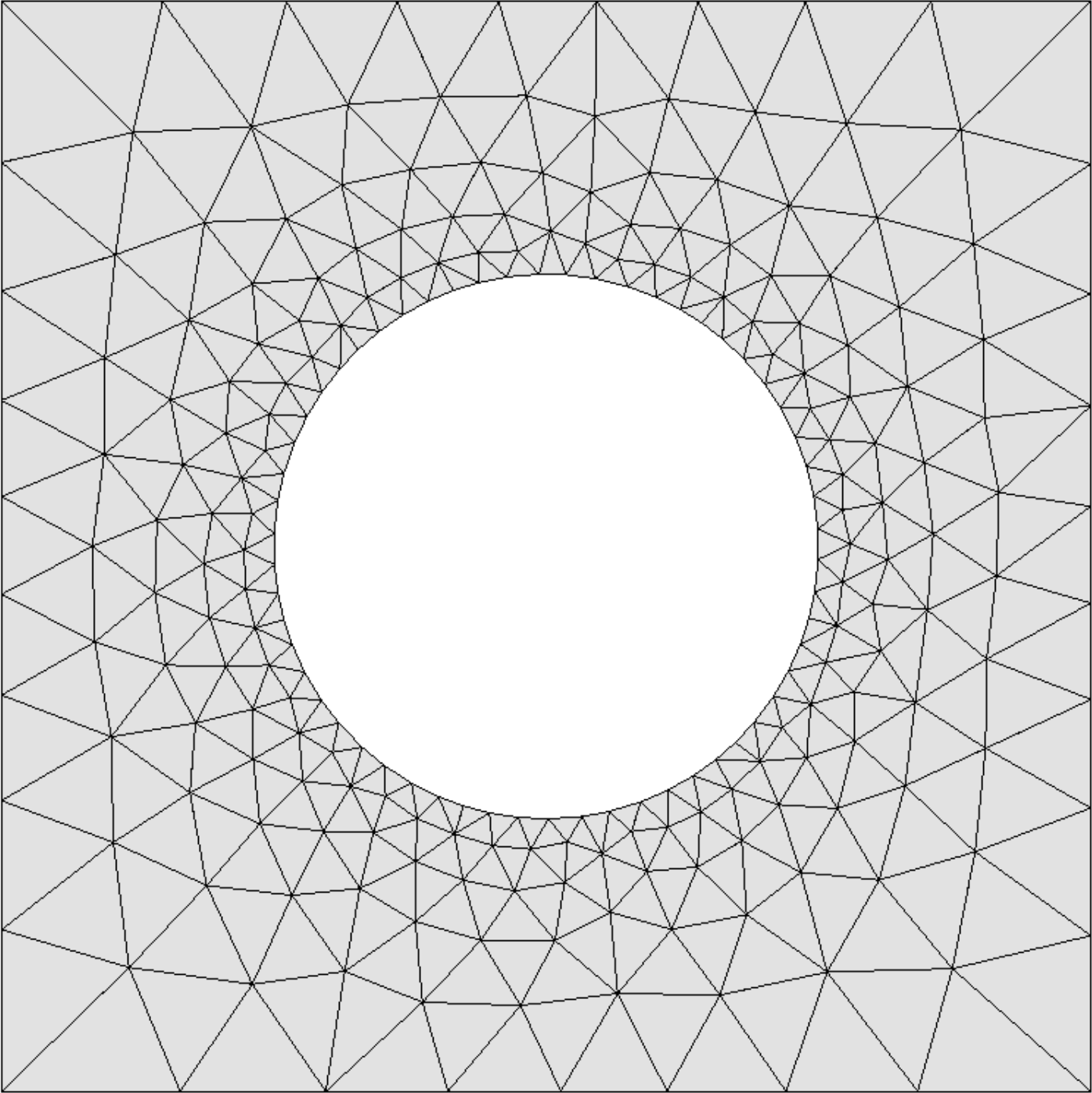}
    \caption{Major steps of the Persson-Strang meshing algorithm, left to right: (i)~distribute points, (ii)~triangulate, (iii)~force equilibrium. Figures produced with the Matlab implementation from \cite{Persson2004a}.}%
    \label{fig:persson_strang}
\end{figure}
With the above notation, we will consider the optimization problem
\begin{equation*}
    \hat\e(\omega_N) = \sum_{i=1}^N\sum_{j\in T_i} \frac12 \left( (1+P) \cdot m_2(\omega_N) - \|x_j-x_i\|^2 \right)_+ \longrightarrow \max,
\end{equation*}
where $ P > 0 $ is a fixed constant. We write $ \E $, as opposed to $ \e $, to indicate that this functional will be maximized.

It is not hard to see that the minimal value of $ \E = 0 $ is attained for the configuration consisting of a single point: $ x_1 = \ldots = x_N $. Another reason why this optimization problem requires dealing with maximizers, not minimizers, can be seen from the scaling rate of the above functional (quadratic), which causes the corresponding rate function to be concave: 
\[
    \t(N) = N^{1-2/d}.
\]
Here $ \sigma = -2/d $, which corresponds to the range $(-1,0)$, excluded from Theorem~\ref{thm:main1}. In Section~\ref{sec:generalizations}, we discuss how Theorem~\ref{thm:main1} can be naturally generalized to this range of $ \sigma $ by replacing the minimizers with maximizers. 

Before we proceed to state the results about maximizers of $ \E $, note that as given above, $ \E $ is not well-defined for collections $ \omega_N $ containing a subset of $ d+2 $ points that lie on a sphere in $ \mathbb R^d $, since Delaunay triangulation of such a subset is not well-defined. This issue can be resolved by the usual approach of small perturbations of the configuration. Indeed, configurations $ \omega_N' $ with spherical subsets form a subset of $ A^N $ of $ (\lambda_d)^N $-measure zero, and so can obtained as limits of $ \omega_N $ that do not contain such subsets. Defining 
\[
    \E(\omega_N') := \limsup_{\omega_N \to \omega_N'} \E(\omega_N),
\]
we ensure that $ \E $ is upper semicontinuous on $ A^N $, and thus maximizers are well-defined. Equivalently, one can think of $ \E $ as being the upper semicontinuous envelope of the values it takes on collections in general position.

The results of Section~\ref{sec:meshing} yield the following statement for the asymptotics of $ \E $. In it, we will assume that for $ j\in T_i $, $ x_j $ is among the nearest neighbors to $ i $, that is, $ T_i\subset I_{i,k} $ for some large enough $ k = k(d) $. This assumption corresponds to choosing a suitable starting configuration in the Persson-Strang procedure, by drawing it from the uniform distribution on $ A $ \cite{Persson2004a}. Indeed, for uniformly distributed configurations the expected number of edges in the Delaunay graph is $ k(d)n $ \cite{dwyerExpected1993},  which suggests that $ T_i\subset I_{i,k} $ holds on average. The question of whether this inclusion holds for Delaunay triangulation of an arbitrary $ \omega_N \subset A $ remains open. For practical purposes, it suffices to prune the edges in the Delaunay graph that violate this inclusion; the following theorem then shows that maximizing the functional associated with the modified graph leads to uniformly distributed configurations.
\begin{theorem}
    Suppose $ d\geq 3 $ and the compact set $ A\subset \mathbb R^d $ satisfies $ \lambda_d(\partial A) =0 $. If $ T_i \subset I_{i,k} $ for $1\leq i \leq N$ and a fixed $ k=k(d) $, there holds
    \[
        \lim_{N\to \infty} \frac{\sup_{\omega_N\subset A} \hat\e(\omega_N)}{N^{1-2/d}} = C_\e(P)\, \lambda_d(A)^{2/d}.
    \]
    If additionally $ \lambda_d(A)>0 $, the maximizers of the upper semicontinuous envelope of $ \hat\e $ converge to the uniform measure on $ A $.
\end{theorem}
Non-uniform distributions can be considered as well, by applying the discussion in Section~\ref{sec:translation_dependent}.
Note that the case $ d = 2 $ cannot currently be treated in the same way; this appears to be only a deficiency of our approach, as in practice the algorithm performs equally well for this value of $ d $.  

\subsubsection{\texorpdfstring{$\Gamma$}{Gamma}-convergence of the short-range functionals}

In Section~\ref{sec:gamma} we discuss $ \Gamma $-convergence of short-range energies.  The results of this section can be viewed as an extension of the results about the large deviation principles for the hypersingular Riesz energy from a joint work of Leblé and Serfaty with the first two authors \cite{hardinLarge2018}. The present extension considers more general short-range interactions $ \e $, but only for the case of zero temperature.

Let us start by recalling some classical results about convergence of a discrete energy functional to its continuous counterpart. 
Suppose $ \mathcal P_N(A) $ is the class of counting measures of $ N $-point subsets of $ A \subset \mathbb R^d $:
\[
    \mathcal P_N(A) = \left\{ \frac1N \sum_i \delta_{x_i} : x_i \in \omega_N \subset A \right\}.
\]
Instead of defining $ \e(\omega_N, A) $ on the collections of points in $ A $, we can instead define it on the counting measures of such collections, and extend to the entire $ \mathcal P(A) $ by setting it equal to $ +\infty $ elsewhere.

It is well known \cite[Prop. 11.1]{sandierVortices2007}, \cite{serfatyCoulomb2015} that when $ \ker \geq 0 $ is a radial, monotonic kernel integrable with respect to $ \lambda_d $, the functionals $ e_N $ that consist of sums of pairwise interactions through $ \ker $ will $ \Gamma $-converge to their continuous double integral counterpart. More precisely, setting
\[
    \e_N(\mu, A) = 
    \begin{cases}
        \frac1{N^2}\sum_{i\neq j} \ker(x_i, x_j), & \mu \in \mathcal P_N(A),\ \supp(\mu) =\omega_N;\\
        +\infty,                                   & \text{otherwise.}
    \end{cases}
\]
and
\[
    \e(\mu, A) =  \iint_A \ker(x,y) \, d\mu(x) d\mu(y),
\]
we have 
\[
    \glim_{N\to \infty} \e_N(\mu, A) =  \e(\mu, A).
\]

In the case of a non-integrable kernel, such as $ \ker(x,y) = \|x-y\|^{-s} $, $ s>d $, both the normalization and the limiting continuous functional are different from the above. For a general short-range interaction $ \e $, the normalization is given by the rate function $ \t $. We have the following result about $ \Gamma$-convergence of such $ \e $.
\begin{theorem}
    \label{thm:gamma}
    Suppose $ \e $ is a short-range interaction as in Theorem~\ref{thm:main1}, $ \sigma > 0 $, and let
    \[
        \e_N(\mu, A) = 
        \begin{cases}
            \e(\omega_N, A), & \mu \in \mathcal P_N(A),\ \supp(\mu) =\omega_N;\\
            +\infty,                                   & \text{otherwise,}
        \end{cases}
    \]
    and 
    \begin{equation}
        \label{eq:gamma_limit} 
        \s(\mu, A) := 
        \begin{cases}
            C_\e \int_A  \left(\frac{d\mu}{d\lambda_d}\right)^{1+\sigma}\,d\lambda_d(x), & \mu \ll \lambda_d, \\ 
            +\infty, & \text{otherwise},
        \end{cases}
    \end{equation}
    where $ C_\e $ does not depend on $ A $, but only on $ \e $ itself.
    Then
    \[
        \glim_{N\to\infty}\frac{\e(\cdot,\, A)}{\t(N)} = \s(\cdot,\,A)
    \]
    on $ \mathcal P( \Omega) $ equipped  with the weak$ ^* $ topology.
\end{theorem}
When $ \sigma < -1 $, the above theorem holds with minor modifications. Thus, the functional $ \s $ is given by $ C_\e \int_A  \left(\frac{d\mu}{d\lambda_d}\right)^{1+\sigma}\,d\lambda_d(x) $ for all $ \mu $, not necessarily absolutely continuous with respect to $ \lambda_d $. This corresponds to the fact that when $ \sigma > 0 $, the asymptotics of minimizers on a set of $ \lambda_d $-measure 0 are infinite, an effect that does not apply to $ \sigma < -1 $.

Comparison of the proof of Theorem~\ref{thm:gamma} above and that of the corresponding result for 2-point interactions with integrable kernel reveals the difference in the asymptotic structures of energies for the long-range and short-range energies: constructing a recovery sequence for the long-range case requires a delicate argument, whereas a short-range recovery sequence is produced by simply taking the minimizers of an appropriate functional with external field. This distinction illustrates certain rigidity that the short-range minimizers possess, which makes them especially appealing for various applications.

\section{Translation invariant short-range functionals}
\label{sec:short_range}

In this section we discuss the properties of general asymptotically short-range functionals $ \e $, acting on discrete subsets of an open bounded set $ \Omega \subset \mathbb R^d $, which is assumed fixed throughout the section. The functionals we will deal with need not be defined outside $ \Omega $. Whenever we discuss compact subsets of $ \mathbb R^d $, they are assumed to be contained in $ \Omega $. 

We use $ \lambda_d $ for the full-dimensional Lebesgue measure in $ \mathbb R^d $.
Expression 
$$ \q_d := [-1/2,1/2]^d $$
stands for the unit cube centered at the origin; accordingly, $ x + a\q_d $ is a cube with side length $ a $ centered at a point $ x \in \mathbb R^d $. Note that until explicitly indicated otherwise, all the cubes considered in $ \mathbb R^d $ have sides parallel to the coordinate axes. A class compact subsets of $ \Omega $ of special interest to us consists of unions of cubes with disjoint interiors:
\[
    \Q = \bigcup_{m=1}^M (x + a_m\q_d).
\]
We denote the class of such sets by $ \mathscr Q $.

\subsection{General properties of short-range interactions}
\label{sec:general} 
Let $ \e $ be defined on all finite collections of points $ \omega_N = \{ x_1,\ldots,x_N \} \subset \Omega \subset \mathbb R^d $, except perhaps those of cardinality $ N \leq N_0(\e) $ for a fixed $ N_0(\e) $:
\[
    \e:\omega_N  \to [E_0,\infty],  \qquad \omega_N\subset \Omega,\ N\geq N_0(\e).
\]
Let further $ \e $ be lower semicontinuous on $ \Omega^N $ in the product topology for every fixed $ N \geq N_0(\e) $; such an $ \e $ will be called an {\it interaction functional}. In what follows $ \Omega\subset \mathbb R^d $ is usually equipped with Euclidean $ l_2 $ distance, but we also consider topologies induced by other norms. 
Functionals $ \e $ that depend on the underlying set will also be of interest: for example, the functionals corresponding to the quantization problem (see Section~\ref{sec:CVT}) are of the form
\[
    \e(\cdot, A):\omega_N  \to [0,\infty],  \qquad \omega_N\subset \Omega,\ N\geq N_0(\e),
\]
where $ A \subset \Omega $. The function $ \e(\cdot, A) $ is assumed to be lower semicontinuous on $ \Omega^N $ for any fixed $ A $.

In the present section we consider the functionals $ \e $ with position-independent asymptotics;  Section~\ref{sec:translation_dependent} will be dedicated to modifying such functionals with a (position-dependent) weight.
We first need to introduce the notion of the {\it translation invariant} interaction functional $ \e $; for such $ \e $, with every translation $ \mathcal T:\mathbb R^d \to \mathbb R^d $ and collection $ \omega_N\subset \mathbb R^d $ there must hold
\[
    \e( \mathcal T(\omega_N), \mathcal T(A))  = \e(\omega_N, A).
\]

The interaction functionals $ \e $ to which the discussion in this section applies are:
\begin{enumerate}
    \item Riesz functionals \cite{hardinMinimal2005}:
        \[
            E(\omega_N) = \sum_{i\neq j} \|x_i - x_j\|^{-s}, \qquad \omega_N= \{ x_1,\ldots, x_N \} \subset \mathbb R^d,
        \]
        for the hypersingular case $ s> d $.
    \item Riesz hypersingular functionals for non-isotropic norms
        \[
            E^{\mathcal V}(\omega_N) = \sum_{i\neq j} \|x_i - x_j\|^{-s}_{\mathcal V}, \qquad \omega_N= \{ x_1,\ldots, x_N \} \subset \mathbb R^d,
        \]
        where the convex body $ \mathcal V $ is the unit ball of the $ \|\cdot\|_{\mathcal V} $ norm.
    \item Truncated Riesz energies (Section~\ref{sec:nearest_neighbor}):
        \[
            E^k(\omega_N) = \sum_{i=1}^N \sum_{j\in I_{i,k}} \|x_i - x_j\|^{-s}, \qquad \omega_N= \{ x_1,\ldots, x_N \} \subset \mathbb R^d,
        \]
        for $ s>0 $, where $ I_{i,k} $ is the set of indices of the $ k $ nearest neighbors to $ x_i $ in $ \omega_N $.

    \item The optimal quantization functional on a compact set $ A $, defined as
        \[
            \cvt(\omega_N, A) = \sum_{i=1}^N \int_A \mathbb{1}_{V_i}\,\|y-x_i\|^p \d\lambda_d(y),
        \]
        where $ p > 0 $, $ V_i $ is the Voronoi cell of $ x_i $, and $ \mathbb{1}_{V_i} $ is its indicator function.
\end{enumerate} 
All of the above examples are translation- and scale invariant functionals, that is,  their minimizing configurations are preserved under translation and rescaling.

By lower semicontinuity with respect to configuration, global minimizers of $ \e(\cdot\,,A) $ are defined for any $ N $ and on any compact $ A\subset \Omega $; we denote them by $ \omega_N^*(A) $:
\begin{equation}
    \label{eq:omitting_set}
    \e(\omega_N^*(A)) := \min \{ \e(\omega_N, A) : \omega_N \subset A \subset \Omega \}.
\end{equation}
In order to deal with the asymptotic properties of a functional $ \e $, we will need to associate every such $ \e $ with a triple of functions 
$$ (\f, \t, \w):[0,\infty)\to[0,\infty]^3, $$
which will be referred to as the {\it cube limit function}, {\it rate}, and {\it fractional rate} of the functional $ \e $, respectively. The necessary properties of these functions are given in Definition~\ref{def:short_range}.
In subsequent discussion we write 
\[ 
    \l_\e(A) = \lim_{N\to\infty} \frac{\e(\omega_N^*(A))}{\t(N)},
\]
for a compact $ A \subset \Omega $, and say that $ A $  belongs to the class $ \mathscr L_\e $, if the above limit exists. We will also assume that the rate function $ \t $  is positive, measurable, and satisfies
\begin{equation}
    \label{eq:rate}
    \lim_{u\to \infty}\frac{\t(t u)}{\t(u)} = \w(t),
\end{equation}
for $0\leq t\leq 1$ and a strictly convex $ \w$. We refer to the class of functions $ \t $ for which \eqref{eq:rate} holds as {\it rate functions}. Note that by definition, such $ \t $ are regularly varying functions, which in turn implies that the limit $ \w $ must be a power law; this is further discussed in Section~\ref{sec:regular_variation}.

Finally, we will need certain monotonicity properties of the interaction functionals.  
Observe that if $ \e $ depends only on the configuration $ \omega_N $, the inclusion $ A\subset B $ for any two compact sets $ A,B \subset \Omega $ implies $ \e(\omega_N^*(A)) \geq \e(\omega_N^*(B)) $, since the second minimum is taken over a larger set. 
On the other hand, in our go-to example of a functional with set-dependence (4), this inequality is reversed. Motivated by these specific short-range interactions, we introduce the notion of set monotonicity.

Let the constant $ \varsigma \neq 0 $ be fixed. The functional $ \e $ will be called {\it set-monotonic}, or simply {\it monotonic}, if there holds
\begin{equation}
    \label{eq:monotonicity}
        \liminf_{N\to\infty} \sgn\varsigma \cdot \frac{\e(\omega_N^*(A))}{\t(N)}   \geq
         \limsup_{N\to\infty} \sgn\varsigma \cdot\frac{\e(\omega_N^*(B))}{\t(N)}      \qquad\text{for}\quad  A\subset B
\end{equation}
for any compact $ A\subset B\subset \Omega $.

In the following definitions, as well as in the rest of the present paper, all the discussed functionals are lower semicontinuous and monotonic. We assume throughout the paper that  $ \t $ is a rate function satisfying~\eqref{eq:rate}. The overall idea of this paper is to study when existence of the asymptotics of a functional on cubes can be extended to a general compact set in $ \mathbb R^d $. The following definition states the existence of asymptotics on cubes.
\begin{definition}
    \label{def:cube_asymp}
    An functional $ \e $ with triple $(\f, \t, \w)$ is said to have the {\it  translation-invariant asymptotics on cubes}, if for all  $a>0$ and $ x\in\mathbb R^d $,
    \begin{equation}
        \label{eq:cube_asymp}
        \lim_{N\to\infty} \frac{\e(\omega_N^*(x+a\,\q_d))}{\t(N)} = \f(a^d) > 0
    \end{equation}
    for a function $ \f $, finite on $ (0,\infty) $.
\end{definition}

Recall that $ \mathscr Q $ stands for the class of finite unions of cubes with disjoint interiors, contained in $ \Omega $.
\begin{definition}
    \label{def:short_range}
    A functional $ \e $ is said to have the {\it short-range interaction property}:
    for union of any $ M $ cubes with disjoint interiors, $ \Q = \bigcup_1^M Q_m $, and a  compact $ A $ there holds
    \begin{equation}
        \label{eq:short_range}
        \lim_{N\to\infty} \frac{\sum_{m=1}^M \e(\omega_N\cap Q_m, A\cap Q_m)}{\e(\omega_N\cap \Q, A\cap \Q)} = 1,
    \end{equation}
    where the sequence $ \omega_N $ has one of the two forms:

    \noindent{\bf (i)}
    $ \omega_N  = \omega_N^*(A)  $;

    \noindent{\bf (ii)}
    $ \omega_N = \bigcup_m \omega_{N_m}^*(A\cap Q_m) $, with $ \sum_m N_m = N $, $ N_m\to \infty $.
    When $ \varsigma > 0 $, the cubes $ Q_m $ are additionally assumed to be positive distance apart for such sequences.
\end{definition}
\begin{remark}
    \label{rem:one_sided_short_range}
    As will transpire from the proofs in this section, it suffices to show \eqref{eq:short_range} with ``$ \geq $'' in place of equality and $ \liminf $ in place of the limit for $ \omega_N^*(A) $. Similarly, for sequences of piecewise minimizers described in (ii), it suffices to show that the $\limsup$ of the ratio in \eqref{eq:short_range} is less than or equal to 1.

    The explanation for this is that Definition~\ref{def:short_range} implies a form of $ \Gamma $-convergence for restrictions of $ \e $ to subsets of a compact set $ A $. See Section~\ref{sec:gamma} for the discussion of this convergence.
\end{remark}
The argument extending existence of asymptotics from Definition~\ref{def:cube_asymp} to general compact subsets of $ \Omega $ involves approximation in the symmetric difference metric on measurable sets, given by $\rho_\triangle(A,B) := \lambda_d(A\triangle B)$, which results in sets with close values of the asymptotics. To justify the possibility of such approximation, we need the following
\begin{definition}
    \label{def:cube_conts}
    Let the constant $ \varsigma \neq 0 $ be the same as in \eqref{eq:monotonicity}.
    The functional $ \e $ will be said to have the {\it continuity property on cubes}, if
    for every cube $ Q = (x+a\,\q_d) \subset \Omega $ and any $ \epsilon \in (0,1) $ there holds
    \begin{equation}
        \label{eq:cube_conts}
        \lim_{N\to\infty}   \sgn \varsigma \cdot  \frac{\e(\omega_N^*(Q))}{\t(N)} \geq  
         \limsup_{N\to\infty}\, ( \sgn \varsigma - \epsilon) \cdot  \frac{\e(\omega_N^*(D))}{\t(N)} 
    \end{equation}
    whenever the compact $ D \subset Q $ satisfies $ \lambda_d(D) \geq (1 - \delta( \epsilon) )\,\lambda_d(Q)$.
\end{definition}
Before we proceed, let us emphasize again that all the interaction functionals discussed below are assumed to have the properties from Definitions~\ref{def:cube_asymp}--\ref{def:short_range}~and~\ref{def:cube_conts}, as well as be lower semicontinuous and satisfy the monotonicity equation \eqref{eq:monotonicity}. For brevity, we refer to such functionals as {\it admissible interaction functionals}.

\medskip
As already mentioned, equation~\eqref{eq:rate} for the rate function implies $ \w $ must be a power law. It will be shown in Section~\ref{sec:regular_variation}, as well as that the assumptions on $ (\f,\t,\w) $ together with Definitions~\ref{def:cube_asymp}--\ref{def:cube_conts} imply $ \f $ is also a power law, and is related to $ \w $ by
\[
    \f(t) = \f(1)t^{-\sigma}, \qquad \w(t) = t^{1+\sigma},
\]
for a constant $ \sigma\in(-\infty,-1)\cup(0,+\infty) $ such that $ \sgn \sigma = \sgn \varsigma $. Equation \eqref{eq:rate} thus means that the rate $ \t $ consists of the fractional rate $ \w $ given by a power function, and a slowly varying remainder.

Note that when $ \varsigma > 0 $, $ \lim_{t\downarrow0} \f(t) = +\infty $. This case corresponds to energy functionals with singularity on the diagonal: $ \e(\omega_N) \to +\infty $ when $ \omega_N\to x_0 \in \mathbb R^d $, with the latter convergence is understood in the product topology of $ \Omega^N $. In the cases relevant to the present paper, the singularity arises when any two points of the configuration are close. It is in order to avoid $ \e $ from being singular on unions of minimizers in Definition~\ref{def:short_range} that we consider sets positive distance apart.

We have so far omitted the case $ -1 < \sigma < 0 $, corresponding to the strictly concave fractional rate function $ \w $. For such $ \w $, of course, we don't expect the minimizers of $ \e $ to be uniformly distributed. It is possible, however, to obtain an analogous theory for maximizers of upper semicontinuous functionals, if the counterparts of Definitions~\ref{def:cube_asymp}, \ref{def:short_range}, \ref{def:cube_conts} hold. We will discuss such functionals and the asymptotics of their maximizers in Section~\ref{sec:generalizations}, for now just note that the corresponding results are very similar.

Let us now give an overview of the properties in Definitions~\ref{def:cube_asymp}--\ref{def:short_range} and their consequences to better motivate them. First, using \eqref{eq:cube_asymp} and monotonicity of the interaction functionals we bound the ratio of $ \e(\omega_N^*(A)) $ to $ \t(N) $ for a general compact $ A $.
\begin{lemma}
    For a compact $ A \subset \Omega $ with $ \lambda_d(A)>0 $ and an admissible interaction functional $ \e $, one has
    \[
        0 <  \liminf_{N\to\infty} \frac{\e(\omega_N^*(A))}{\t(N)} \leq \limsup_{N\to\infty} \frac{\e(\omega_N^*(A))}{\t(N)} < \infty.
    \]
\end{lemma}
\begin{proof}
    Suppose $ \varsigma > 0 $; the case of negative $ \varsigma $ is handled in the same way.
    By Lebesgue density theorem, there exists a cube $ Q $ such that $ \lambda(A\cap Q) \geq (1-\delta(1/2)) \lambda_d(Q) $, where $ \delta(\cdot) $ is the function from Definition~\ref{def:cube_conts}; by this definition and set-monotonicity of $ \e $ \eqref{eq:monotonicity}, one obtains further
    \[
        \limsup_{N\to\infty}  \frac{\e(\omega_N^*(A))}{\t(N)}
        \leq
        \liminf_{N\to\infty}   \frac{\e(\omega_N,A\cap Q)}{\t(N)} \leq   2\lim_{N\to\infty}    \frac{\e(\omega_N^*(Q))}{\t(N)} = 2\f(\lambda_d(Q)) < \infty.
    \]
    On the other hand $ A \subset Q' $ for some large cube $ Q' $.
    From the set-monotonicity of $ \e $ and \eqref{eq:cube_asymp}, we obtain:
    \[
        \liminf_{N\to\infty} \frac{\e(\omega_N^*(A))}{\t(N)} \geq \f(\lambda_d(Q')) > 0.
    \]
\end{proof}
We have therefore $ \e(\omega_N^*(A)) = O(\t(N)) $. Since Definition~\ref{def:short_range} gives for a compact $ A \subset \Omega $,
\[ 
    {\e( \omega_N^*(A)\cap \Q_1, A\cap \Q_1)+\e(\omega_N^*(A)\cap \Q_2, A\cap \Q_2)} =   {\e\left(\omega_N^*(A)\right)} + o( \e(\omega_N^*(A)) ), \qquad N \to \infty,
\]
for any pair of subsets $ \Q_1,\Q_2 \in \mathscr Q $ with disjoint interiors, and it follows
\begin{equation}
    \label{eq:splitting_o_t}
    {\e( \omega_N^*(A)\cap \Q_1, A\cap \Q_1)+\e(\omega_N^*(A)\cap \Q_2, A\cap \Q_2)} =   {\e\left(\omega_N^*(A)\right)} + o(\t(N)), \qquad N \to \infty,
\end{equation}
for such $ \Q_1,\Q_2 $, and $ A $.

To motivate the assumption of convexity of the rate function \eqref{eq:rate}, let further the compact sets $ A_1,A_2 \in \mathscr L_\e $ be positive distance apart and $ A = A_1\cup A_2 $. For some suitable $ \Q_l $, $ A_l = A\cap \Q_l $, $ l=1,2 $. Dividing equation~\eqref{eq:splitting_o_t} through by $ \t(N) $ gives
\begin{equation}
    \label{eq:lower_split}
    \liminf_{N\to\infty} \frac{\e\left(\omega_N^*(A)\right)}{\t(N)} \geq \lim_{N\to\infty} \frac{\t(N_1)}{\t(N)} \cdot \frac{\e(\omega_{N_1}^*(A_1))}{\t(N_1)} + \lim_{N\to\infty} \frac{\t(N_2)}{\t(N)}\cdot\frac{\e(\omega_{N_2}^*( A_2))}{\t(N_2)},
\end{equation}
if the limits of $ \t(N_l)/ \t(N) $ on the right exist, where we write $ N_l = \#(\omega_N^*(A) \cap A_l) $, $ l=1,2 $. As will transpire from the proof of Lemma~\ref{lem:cube_uniform}, the minimality of $\e(\omega_N^*(A_1\cup A_2))$  implies that the inequality above can be replaced by equality; in fact for such $ A_1$,  $ A_2 $, there holds $ A =  A_1\cup A_2 \in \mathscr L_\e $ and
\begin{equation}
    \label{eq:split_union}
    \l_\e(A) =  \lim_{N\to\infty} \frac{\t(N_1)}{\t(N)} \cdot \l_\e(A_1 )  + \lim_{N\to\infty} \frac{\t(N_2)}{\t(N)} \cdot \l_\e(A_2 )  , 
\end{equation}
in particular the limits on the right do exist.

Indeed, to show the existence of the limits, note that the ratios $ N_1/N $ and $ N_2/N $ are, by definition, the counting measures $ \nu^*_N(A_1) $ and $ \nu^*_N(A_2) $, where
\[
    \nu^*_N = \frac1N \sum_{i=1}^N \delta_{x_i}, \qquad \{ x_1,\ldots,x_N \} = \omega_N^*.
\]
In view of weak$ ^* $-compactness of probability measures on $ A_1\cup A_2 $, by passing to a subsequence if necessary, one can assume that $ \nu_N^* \weakto \mu $, $ N\to\infty $, for some probability measure $ \mu $, and so
\[
    N_l = \mu(A_l)N + o(N), \qquad N\to\infty, \quad l=1,2.
\]
Then by \eqref{eq:rate}, the limits in \eqref{eq:split_union} do exist:
\[
    \lim_{N\to\infty} \frac{\t(N_l)}{\t(N)} = \w(\mu (A_l)),\qquad l=1,2.
\]
Since the limits of counting ratios must add up to $ 1 = \mu(A_1\cup A_2) $, we can rewrite the right-hand side of \eqref{eq:split_union} as 
\begin{equation*}
    \w(t) \cdot \l_\e(A_1 )  + \w(1-t) \cdot \l_\e(A_2 ), \qquad t\in [0,1].
\end{equation*}
This is a strictly convex function on $ [0,1] $, and so has a unique minimizer there. Hence the distribution of the weak$ ^* $ limit of the subsequence of minimizing configurations $ \nu^*_N $ does not depend on the subsequence: 
\begin{equation}
    \label{eq:split_min}
    \begin{aligned}
    \min \{ \w(t) \cdot \l_\e(A_1 )  + \w(1-t) \cdot \l_\e(A_2 ) 
    :\ & t\in[0,1] \} \\
    &= \w(\mu(A_1)) \cdot \l_\e(A_1 )  + \w(\mu(A_2)) \cdot \l_\e(A_2 ).
\end{aligned}
\end{equation}
The convexity of $ \w $ in \eqref{eq:rate} ensures that the minimum is uniquely defined, and so are the values of $ \mu(A_l) $, $ l=1,2 $. Furthermore, in the case $ \l_\e(A_1) = \l_\e(A_2) $, convexity of $ \w $ implies there must hold
\[
    \mu(A_1) = \mu(A_2).
\]


Finally, let us discuss the notion of continuity on cubes from Definition~\ref{def:cube_conts}. By combining it with \eqref{eq:cube_asymp} we immediately obtain the continuity of $ \f $.
\begin{corollary}[From Definition~\ref{def:cube_conts}] 
    \label{cor:f_is_conts}
    Under the assumptions of Definition~\ref{def:short_range}, the function $ \f:[0,\infty) \to [0,\infty] $ is continuous in the extended sense. 
\end{corollary} 
Our primary use for Definition~\ref{def:cube_conts} is that it implies, in conjunction with Definitions~\ref{def:cube_asymp}--\ref{def:short_range}, that if a compact set $ A $ can be approximated well from above by unions of cubes with disjoint interiors $ \Q $ in the metric $ \rho_\triangle(A,\Q) = \lambda_d(A\triangle \Q)  $, then $ A $ is itself in $ \mathscr L_\e $. Indeed, for a compact $ A $, fix an $ \epsilon > 0 $ and let $ \delta = \delta(\epsilon) $, determined as in \eqref{eq:cube_conts}. As will be shown in Theorem~\ref{thm:general_uniform}, 
if for the compact set $ A $ and every $ \epsilon >0 $ there is a compact union of $ M $ cubes $ \bigcup_m Q_m = \Q\in \mathscr L_\e $ such that $ \Q\supset A $ and $ \lambda_d(A\cap Q_m)> \lambda_d (Q_m) (1 -\delta) $, $ 1\leq m \leq M $, then $ A \in \mathscr L_\e $ and 
\[
    \sgn\varsigma\cdot \l_\e(A) = \sup \{ \sgn\varsigma\cdot\l_\e(\Q) :  A \subsetneq \Q \in \mathscr L_\e,\ \Q\text{ a union of cubes with disjoint interiors }\}.
\]

In Section~\ref{sec:asymptotics_uniform} we construct such approximations of compact subsets of $ \mathbb R^d $ with finite unions of cubes $ \Q $ that have pairwise disjoint interiors. For these approximations to be useful, we must first show that the limit $ \l_\e $ exists on the unions of cubes, so that indeed $ \Q \in \mathscr L_\e $. In Lemmas~\ref{lem:cube_uniform} and \ref{lem:cube_volume} we verify that this is the case, and in addition establish a functional equation \eqref{eq:fun_equation} relating $ \f $ to $ \w $, which will allow us to characterize both of these functions, as well as $ \t $, in the following Section~\ref{sec:regular_variation}.
\begin{lemma}
    \label{lem:cube_uniform}
    Let $ \e $ be  an admissible interaction functional with the cube limit function $ \f $ and a rate $ \t $.  
    Then for any $ \Q = \bigcup_{m=1}^M Q_m := \bigcup_{m=1}^M (x_m + a_m\q_d) $, a union of cubes with disjoint interiors, there holds
    \begin{equation}
        \label{eq:cube_uniform}
        \lim_{N\to\infty} \frac{\e(\omega_N^*(\Q))}{\t(N)} = \sum_{m=1}^M \w(\beta^*_m)\f(a_m^d)
    \end{equation}
    with 
    \[
        (\beta_1^*,\ldots\beta_M^*) = \arg\min \left\{ \sum_m \w(\beta_m) \f(a_m^d) : \sum_m \beta_m = 1, \ \beta_m > 0 \right\}.
    \]
\end{lemma}
\begin{proof}
    For simplicity, we give the proof with $ M=2 $; the general case follows by the same argument. Let $ \mathscr N $ be a subsequence along which the limits of counting measures of $ \omega_N^*(Q_m) $ exist:
    \begin{equation*}
            \lim_{\mathscr N \ni N \to \infty} N_m / N =: \beta^*_m, \qquad m=1,2,
    \end{equation*}
    where $ N_m = \#(\omega_N^*(\Q)\cap Q_m)  $.
    By Definition~\ref{def:short_range}, there holds
    \[
        \begin{aligned}
        \liminf_{\mathscr N \ni N \to \infty} \frac{\e(\omega_N^*(\Q))}{\t(N)} 
        =& \liminf_{\mathscr N \ni N \to \infty}
        \left(
            \frac{\e(\omega_N^*(\Q))}{ \sum_{m=1,2}\e\left(\omega_N^*(\Q)\cap Q_m, Q_m\right) } \cdot 
            \frac{ \sum_{m=1,2}\e\left(\omega_N^*(\Q)\cap Q_m, Q_m\right) }{\t(N)}
        \right) \\
        &\geq \liminf_{\mathscr N \ni N \to \infty} \frac{\e\left(\omega_N^*(\Q)\cap Q_1, Q_1\right)}{\t(N)} + \liminf_{\mathscr N \ni N \to \infty} \frac{\e\left(\omega_N^*(\Q)\cap Q_1, Q_1\right)}{\t(N)}\\
        &= \liminf_{\mathscr N \ni N \to \infty} \frac{\t(N_1)}{\t(N)} \cdot\frac{\e\left(\omega_N^*(\Q)\cap Q_1, Q_1\right)}{\t(N_1)}
        + \liminf_{\mathscr N \ni N \to \infty} \frac{\t(N_2)}{\t(N)} \cdot \frac{\e\left(\omega_N^*(\Q)\cap Q_1, Q_1\right)}{\t(N_2)}\\
        &\geq \w(\beta^*_1) \f(a_1^d) + \w(\beta^*_2) \f(a_2^d) ,
    \end{aligned}
    \]
    where it is used that the limits $ \t(N_m)/\t(N) $ exist by equation~\eqref{eq:rate} and $ \t $ being a rate function. 

    To obtain the converse estimate, note first for $ \varsigma < 0 $, that item (ii) in Definition~\ref{def:short_range} gives that for any collection of $ \beta_m > 0 $ adding up to 1, there exists a sequence $ \omega_N = \bigcup_{m=1,2} \omega_{N_m}^*( Q_m) $, where $ N_m/N\to \beta_m$ for $ \mathscr N \ni N \to \infty $, to which equation~\eqref{eq:short_range} applies. Then, Definition~\ref{def:short_range} gives for such $ \omega_N $,
    \[
        \lim_{\mathscr N \ni N \to \infty} \frac{\e(\omega_N,\Q)}{\t(N)} = \lim_{\mathscr N \ni N \to \infty} \frac{\e(\omega_{N_1}^*(Q_1)) + \e(\omega_{N_2}^*(Q_2))}{\t(N)} =  \w(\beta_1) \f( a_1^d) + \w(\beta_2) \f( a_2^d).
    \]
    By definition of the minimizer, $ \e(\omega_N^*(\Q)) \leq \e(\omega_N,\Q) $, which gives
    \begin{equation}
        \label{eq:ineqs1}
        \begin{aligned}
            &\w(\beta^*_1) \f(a_1^d) + \w(\beta^*_2) \f(a_2^d) \\
            &\leq  \limsup_{\mathscr N \ni N \to \infty} \frac{\e(\omega_N^*(\Q))}{\t(N)} 
            \leq \lim_{\mathscr N \ni N \to \infty} \frac{\e(\omega_N,\Q)}{\t(N)} \\
            & = \w(\beta_1) \f(a_1^d) + \w(\beta_2) \f(a_2^d).
    \end{aligned}
    \end{equation}
    For $ \varsigma > 0 $, the energy functional can have a singularity on the diagonal, so instead of placing minimizers on the entire cubes $ Q_m $, we need to place them on a proper subset of each cube. To that end, let $ \gamma \in (0,1) $ be constant and denote $ \gamma Q_m  =(x_m +  \gamma a_m\q_d) $. Clearly, $ \gamma Q_m $ are a positive distance apart for any $ \gamma<1 $. Using Definition~\ref{def:short_range} again for  $ \omega_N = \bigcup_{m=1,2} \omega_{N_m}^*(\gamma Q_m) $, we have
    \begin{equation}
        \label{eq:ineqs2}
        \begin{aligned}
        \limsup_{\mathscr N \ni N \to \infty} \frac{\e(\omega_N^*(\Q))}{\t(N)} 
        &\leq \limsup_{\mathscr N \ni N \to \infty} \frac{\e(\omega_N^*(\gamma\Q))}{\t(N)}\\
        &\leq \lim_{\mathscr N \ni N \to \infty} \frac{\e(\omega_{N_1}^*(\gamma Q_1)) + \e(\omega_{N_2}^*(\gamma Q_2))}{\t(N)}\\
        &=  \w(\beta_1) \f((\gamma a_1)^d) + \w(\beta_2) \f((\gamma a_2)^d),
        \end{aligned}
    \end{equation}
    where the inequality follows from monotonicity \eqref{eq:monotonicity}.

    The inequalities \eqref{eq:ineqs1} and \eqref{eq:ineqs2} hold for all $ \beta_1+\beta+2 =1 $ and a fixed $ \beta_1^*,\beta_2^* $. Using convexity of $ \w $ and continuity of $ \f $ we conclude that for both positive and negative $ \varsigma $,
    \[
        (\beta_1^*,\beta_2^*) = \arg\min \left\{ \w(\beta_1) \f(a_1^d) + \w(\beta_2) \f(a_2^d) : \beta_1 + \beta_2 = 1 \right\}.
    \]
    This shows that the value of the asymptotics of $ \e(\omega_N^*(\Q)) $ does not depend on the subsequence $ \mathscr N $, and is equal to $ \w(\beta^*_1) \f(a_1^d) + \w(\beta^*_2) \f(a_2^d) $. 
\end{proof}
The above proof shows that $ \mathscr Q \subset \mathscr L_\e $. Note that the same argument can now be repeated with $ Q_m\in \mathscr Q $ instead of $ Q_m $ being a single cube. One would then obtain an equation similar to \eqref{eq:cube_uniform} with the quantity $ \f(a_m^d) $ replaced by 
\[
    \l_\e(Q_m) = \lim_{N\to\infty} \frac{\e(\omega_N^*(Q_m))}{\t(N)}.
\]

\begin{corollary}
    \label{cor:cube_uniform}
    The statement of Lemma~\ref{lem:cube_uniform} holds with $ \Q = \bigcup_m Q_m $, where sets $ Q_m \in \mathscr Q $ and have disjoint interiors, and $ \l_\e(Q_m) $ in place of $ \f(a_m^d) $. 
\end{corollary}
The proof of Lemma~\ref{lem:cube_uniform} and Corollary~\ref{cor:cube_uniform} show also that for $ \Q_1, \Q_2 $ as in Definition~\ref{def:short_range}, there holds
\begin{equation*}
    \lim_{N\to\infty} \frac{\e(\omega_{N_1}^*(\Q_1)) + \e(\omega_{N_2}^*(\Q_2))}{\e(\omega_N^*(\Q_1\cup \Q_2))} = 1,
\end{equation*}
where $ N_l = \#(\omega_N^*(\Q_1\cup \Q_2)\cap\Q_l) $, $ l=1,2 $. Here we ascribe the points in $ \Q_1\cap \Q_2 $ to one of the sets, and not the other, so that the total cardinality is correct: $ N_1 + N_2 =N $. 
\begin{lemma}
    \label{lem:cube_volume}
    Let $ \e $ be  an admissible interaction functional with the cube limit function $ \f $ and a rate $ \t $.  
    Then
    \begin{equation}
        \label{eq:limits_equal}
        \lim_{N\to\infty} \frac{\e(\omega_N^*(\Q))}{\t(N)} = \f(\lambda_d(\Q)).
    \end{equation}
    for any $ \Q = \bigcup_{m=1}^M (x_m + a_m\q_d) $, a union of cubes with disjoint interiors and $ \lambda_d(\Q) < \infty $.
\end{lemma}
\begin{proof}
    Let $ Q_m = (x_m + a_m\q_d) $, $ 1\leq m \leq M $. By Lemma~\ref{lem:cube_uniform}, the limit in the left-hand side of \eqref{eq:limits_equal} exists and does not change if the cubes in $ \Q $ are rearranged by translations $ \mathcal T_m  $ into another collection of cubes $ \bigcup_m \mathcal T_m (Q_m) $ with disjoint interiors.
    Let $ Q_0 $ be a fixed cube of the same measure as the collection $ \Q $, 
    $$ Q_0 =  \sqrt[d] {\lambda_d (\Q)}\,\, \q_d. $$
    It suffices to show that the asymptotics of the minimizers of $ \e $ on $ \Q $ and $ Q_0 $ coincide.

    Fix a number $ \epsilon >0 $. We will proceed using that by Definition~\ref{def:cube_conts}, an approximation of the cube collection $ \Q $ below up to a set of measure $ \delta(\epsilon) \lambda_d(\Q) $ results in approximation of asymptotics up to $ \epsilon $. To construct such approximation, first consider a subdivision of $ Q_0 $ into $ n^d $ equal cubes congruent to 
    $$ \widetilde Q_n = Q_0/n =  \frac{\sqrt[d] {\lambda_d (\Q)}}{n}\, \q_d ;$$
    by taking $ n $ sufficiently large, one can ensure that there exist integers $ l_m $, $ 1\leq m \leq M $, for which
    \begin{equation}
        \label{eq:choice_of_partition}
        0\leq \lambda_d (Q_m) - l_m^d\lambda_d(\widetilde Q_n) < \delta\left(\epsilon\right)\min_m \{ \lambda_d(Q_m) \}, \qquad 1 \leq m \leq M.
    \end{equation}
    \begin{figure}[t]
        \centering
        \includegraphics[width=0.7\linewidth]{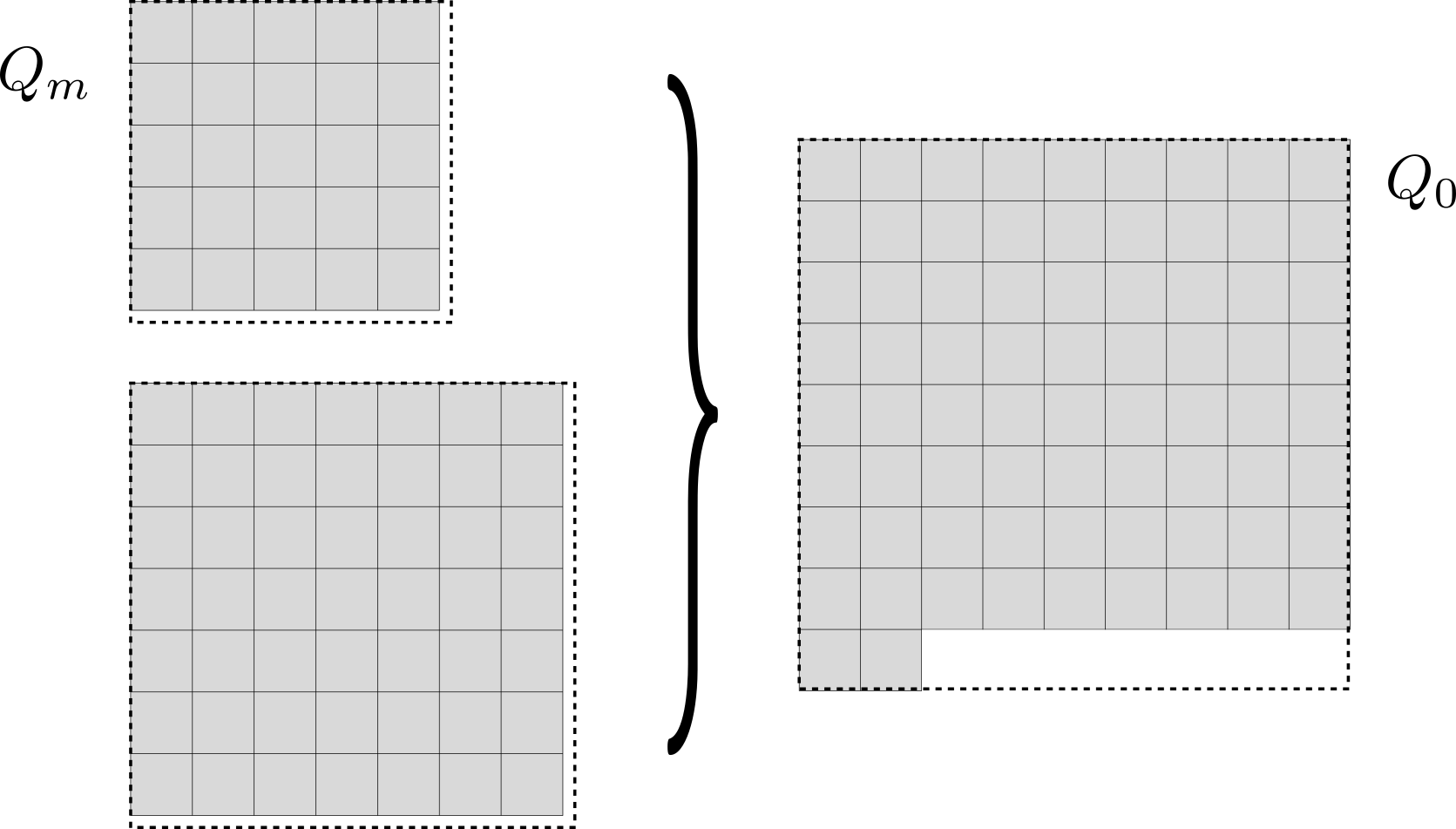}
        \caption{Shaded: collection of cubes $ \bs D $, each congruent to $ \widetilde Q_n $ (left) and its rearrangement inside $ Q_0 $, denoted by $ \bs D_0 $ (right).}%
        \label{fig:cubes}
    \end{figure}
    Let $ \widetilde n = \sum_m l_m^d $. Collection $ \Q $ can be approximated by a collection of $ \widetilde n $ cubes congruent to $ \widetilde Q_n $, which we denote by $ \bs D $, with $ l_m^d $ such cubes placed inside each of $ Q_m $, $ 1\leq m \leq M $. By Corollary~\ref{cor:cube_uniform} applied to the set $ \bs D $, 
    \[
        \lim_{N\to\infty}   \sgn \varsigma \cdot  \frac{\e(\omega_N^*(\bs D))}{\t(N)}
        =   \sgn \varsigma \cdot \sum_{m=1}^M  \w(\beta_m^{\bs D})\, \l_\e(\bs D\cap Q_m),
    \]
    where similarly to Lemma~\ref{lem:cube_volume},
    \[
        (\beta_1^{\bs D},\ldots\beta_M^{\bs D}) = \arg\min \left\{ \sum_m \w(\beta_m) \l_\e(\bs D\cap Q_m) : \sum_m \beta_m = 1, \ \beta_m > 0 \right\}.
    \]
    Applying Lemma~\ref{lem:cube_volume} to $ \Q $ gives in the same way
    \[
        \lim_{N\to\infty}   \sgn \varsigma \cdot  \frac{\e(\omega_N^*(\Q))}{\t(N)}
        =   \sgn \varsigma \cdot \sum_{m=1}^M  \w(\beta_m^{\Q})\, \f(a_m^d),
    \]
    where
    \[
        (\beta_1^{\Q},\ldots\beta_M^{\Q}) = \arg\min \left\{ \sum_m \w(\beta_m) \, \f(a_m^d) : \sum_m \beta_m = 1, \ \beta_m > 0 \right\}.
    \]
    By the choice of partition \eqref{eq:choice_of_partition}, the continuity property \eqref{eq:cube_conts}, and monotonicity~\eqref{eq:monotonicity}, it follows
    \begin{equation*}
        \sgn \varsigma \cdot  \f(a_m^d) \geq  
        (\sgn \varsigma - \epsilon) \cdot  \l_\e(\bs D\cap Q_m) \geq (\sgn \varsigma - \epsilon) \cdot \f(a_m^d).
    \end{equation*}
    Since the function
    $$ (z_1, \ldots, z_m) \mapsto \min \left\{ \sum_m \w(\beta_m) z_m : \sum_m \beta_m =1, \ \beta_m > 0 \right\} $$
    is increasing in each $ z_m $, we have for $ \sgn \varsigma = 1 $
    \[
        \begin{aligned}
            \sgn\varsigma\cdot \lim_{N\to\infty}   \frac{\e(\omega_N^*(\Q))}{\t(N)} 
            &=  \sum_{m=1}^M  \w(\beta_m^{\Q})\, \f(a_m^d) 
            \geq 
            (1 - \epsilon) \sum_{m=1}^M  \w(\beta_m^{\Q}) \l_\e(\bs D\cap Q_m) \\
            &\geq
            (1 - \epsilon) \sum_{m=1}^M  \w(\beta_m^{\bs D}) \l_\e(\bs D\cap Q_m) \\
            &= (\sgn \varsigma -\epsilon) \cdot \lim_{N\to\infty} \frac{\e(\omega_N^*(\bs D))}{\t(N)}.
        \end{aligned}
    \]
    Also, for $ \sgn \varsigma = -1 $, 
    \[
        \begin{aligned}
        \sgn\varsigma\cdot \lim_{N\to\infty} \frac{\e(\omega_N^*(\Q))}{\t(N)} 
        &=  -\sum_{m=1}^M  \w(\beta_m^{\Q})\, \f(a_m^d) 
        \geq 
        -\sum_{m=1}^M  \w(\beta_m^{\bs D})\, \f(a_m^d) \\
        &\geq 
        -(1 + \epsilon)\sum_{m=1}^M  \w(\beta_m^{\bs D})\, \l_\e(\bs D\cap Q_m) \\
        &= (\sgn \varsigma - \epsilon) \cdot \lim_{N\to\infty} \frac{\e(\omega_N^*(\bs D))}{\t(N)}.
        \end{aligned}
    \]
    Combining the last two inequalities with the monotonicity property \eqref{eq:monotonicity}, one has an estimate for the asymptotics on $ \bs D $:
    \[
        \sgn \varsigma \cdot  \lim_{N\to\infty}   \frac{\e(\omega_N^*(\Q))}{\t(N)} 
        \geq  
        (\sgn \varsigma - \epsilon) \cdot \lim_{N\to\infty} \frac{\e(\omega_N^*(\bs D))}{\t(N)}
        \geq \sgn \varsigma \cdot \lim_{N\to\infty} \frac{\e(\omega_N^*(\Q))}{\t(N)} -\epsilon \l_\e(D).
    \]
    Note that this implies, $ \l_\e(D) $ is bounded uniformly in $ \epsilon $ when $ \epsilon \downarrow0 $.
    On the other hand, as rearrangement of cubes does not change the asymptotics, placing the $ \widetilde n $ cubes comprising the collection $ \bs D $ inside $ Q_0 $ allows to approximate the asymptotics on the latter set as well. Denote the rearrangement of $ \bs D $ placed inside $ Q_0 $ by $ \bs D_0 $. Summing up the relations \eqref{eq:choice_of_partition} over $ m $ gives
    \[
        0\leq \lambda_d (Q_0) - \lambda_d (\bs D_0) < \delta\left(\epsilon\right)\lambda_d(Q_0).
    \]
    Thus, an application of \eqref{eq:cube_conts} to $ \bs D_0 \subset Q_0 $ together with \eqref{eq:monotonicity} gives for the asymptotics on $ \bs D_0 $
    \[
        \sgn \varsigma \cdot \lim_{N\to\infty}  \frac{\e(\omega_N^*(Q_0))}{\t(N)} \geq  
        (\sgn \varsigma - \epsilon) \cdot \lim_{N\to\infty}  \frac{\e(\omega_N^*(\bs D_0))}{\t(N)} \geq
        \sgn \varsigma \cdot \lim_{N\to\infty}  \frac{\e(\omega_N^*(Q_0))}{\t(N)} -\epsilon \l_\e(D_0).
    \]
    Since $ \l_\e(\bs D) = \l_\e(\bs D_0) $ by construction, the two estimates for $ \bs D $ and $ \bs D_0 $ yield
    \[ 
        \left|\lim_{N\to\infty} \frac{\e(\omega_N^*(\Q))}{\t(N)} - \lim_{N\to\infty} \frac{\e(\omega_N^*(Q_0))}{\t(N)}\right|  < 2 \epsilon \l_\e(\bs D).
    \]
    Since $ \epsilon $ is arbitrary and $ \l_\e(\bs D) $ is bounded uniformly in $ \epsilon $, this completes the proof.
\end{proof}

\subsection{Functions of regular variation}
\label{sec:regular_variation}
Let us show that the assumptions on $ \t $ and Definition~\ref{def:short_range} imply $ \f $ and $ \w $ must be power laws. To that end, we will use some ramifications of the Cauchy functional equation, in particular the theory of regularly varying functions \cite{binghamRegular1987}. First, we need the notion of slowly varying functions, introduced by Karamata.
\begin{definition}
    A positive measurable function $ v(t) $ defined on $ [a,\infty) $ and satisfying 
    \[
        \lim_{u\to\infty} \frac{v(t u)}{v(u)} = 1
    \]
    is said to be {\it slowly varying}.
\end{definition}
The statement of the next theorem is also due to Karamata, and it fully characterizes the fractional rate $ \w(t) $.
\begin{theorem}[{\cite[Thm. 1.4.1]{binghamRegular1987}}]
    \label{thm:binghamRegular1987}
    If a measurable $ \t > 0 $ satisfies \eqref{eq:rate} for $ t $ from a set of positive measure, then\\
    a) \eqref{eq:rate} holds for all $ t > 0 $;\\
    b) $ \w(t) = t^{1+\sigma} $ for some $ \sigma $ and all $ t > 0 $;\\
    c) $ \t(t) = t^{1+\sigma} v(t) $ with a slowly varying $ v(t) $.
\end{theorem}
\noindent The connection of this result to Cauchy functional equation is due to Feller~\cite{feller1971introduction}. This connection explains the assumption that \eqref{eq:rate} hold on a set of positive measure, since a multiplicative subgroup of $ \mathbb R_+ $ containing a set of positive measure coincides with the entire $ \mathbb R_+ $. It is further worth noting that one has the following representation theorem by Korevaar et al.\ for the slowly varying component $ v(t) $.
\begin{theorem}[\cite{korevaarNote1949, binghamRegular1987}]
    The function $ v(t) $ is slowly varying if and only if it can be represented as
    \[
        v(t) = c(t) \exp\left(\int_{a}^t \frac{\delta(u)}{u}\, du \right), \qquad x \geq a,
    \]
    for an $ a>0 $, a measurable $ c(t) $ tending to a positive limit and a $ \delta(t) $ tending to 0 when $ t\to \infty $.
\end{theorem}



The above results give a full characterization of the fractional rate function $ \w $. To study the asymptotic properties of $ \e $ and distribution of its minimizers, we will need to characterize $ \f $ as well. Consider a union of $ M $ equal cubes with disjoint interiors $ \Q = \bigcup_{m=1}^M Q_m $ and combine equations \eqref{eq:cube_uniform}--\eqref{eq:limits_equal} to obtain
\[
    \f(\lambda_d(\Q))=\lim_{N\to\infty} \frac{\e(\omega_N^*(\Q))}{\t(N)} = \f(\lambda_d (Q_1)) \sum_{m=1}^M \w(\beta_m)  = M \w\left(\frac1M\right) \f(\lambda_d (Q_1)).
\]
We used here that, in view of convexity of $ \w $, the minimum of 
$$ \f(\lambda_d (Q_1)) \sum_{m=1}^M \w(\beta_m)  $$
is achieved at $ \beta_1^* = \beta_2^* = \ldots = \beta_M^* = 1/M $. Since $ \lambda_d(Q_1) $ can take any positive value, we obtain the following functional equation for $ \f $:
\begin{equation}
    \label{eq:fun_equation}
    \f(Mt) = M\w(1/M) \f(t), \qquad t>0, \quad M = 1,2,3,\ldots.
\end{equation}
Such equations and their solutions have been studied by {Acz{\'e}l and Dar{\'o}czy} \cite{aczelMeasures1975}, as a generalization of the so-called {\it completely additive number-theoretic functions} $ \phi $, satisfying the equation 
\[
    \phi(MN) = \phi(M) + \phi(N), \qquad M,N =1,2,3,\ldots.
\]
Remarkably, just as theory of regular variation discussed above, the study of such functional equations is related to the Cauchy functional equation.
The first characterization theorem for $ \phi $ was given by Erd\H os (1957), and several generalizations have been obtained later \cite{aczelMeasures1975}. Here we will need the following result to characterize the cube limit function $ \f $:
\begin{theorem}[{\cite[Thm. (0.4.43)]{aczelMeasures1975}}]
    \label{thm:aczel_daroczy}
    Suppose a continuous monotonic $ f $ satisfies
    \[
        f(Mt) = A(M)f(t), \qquad t\geq 1, \quad M= 1,2,3,\ldots.
    \]
    If $ A(M) \not\equiv 1 $, then
    \[
        f(t) =  f(1)\,t^{-\sigma}, \qquad \sigma \neq 0,
    \]
    and 
    \[
        A(M) = M^{-\sigma}.
    \]
\end{theorem}
\noindent By Corollary~\ref{cor:f_is_conts}, $ \f $ is continuous; it is monotonic by~\eqref{eq:monotonicity}. Applying the above theorem, we conclude the following.
\begin{corollary}
    \label{cor:f_is_powerlaw}
    If the cube limit function $ \f $ satisfies \eqref{eq:fun_equation}, then $ \f(t) = \f(1) \,t^{-\sigma} $, $ t>0 $, $ \sigma \neq  0$.
\end{corollary} 
We have therefore obtained a complete characterization of both $ \f $ and $ \w $ for an admissible interaction functional $ \e $. The present section can be summarized in the next result.

\begin{proposition}
    \label{prop:triple}
    If an interaction functional $ \e $ associated with the triple $ (\f, \t, \w) $ satisfies Definitions~\ref{def:cube_asymp}--\ref{def:cube_conts}, then 
    \[
        \f(t) = \f(1)t^{-\sigma}, \qquad \t(t) = t^{1+\sigma} v(t), \qquad \w(t) = t^{1+\sigma}
    \]
    for a $ \sigma \in (-\infty, -1) \cup (0,\infty) $ and a slowly varying $ v $.
\end{proposition}
Note that the range $ [-1,0] $ for $ \sigma $ has to be excluded since $ t^{1+\sigma} $ is not (strictly) convex there. The above proposition enables us to use $ \sgn \sigma $ in place of $ \sgn \varsigma $ from now on.

\subsection{Asymptotics of short-range functionals}
\label{sec:asymptotics_uniform}

The main theorem of this section extends the expression for the asymptotics of $ \e(\omega_N^*(A)) $ in \eqref{eq:cube_asymp} to general compact subsets of $ \mathbb R^d $, thereby showing that all such sets belong to $ \mathscr L_\e $. 

\begin{theorem}
    \label{thm:general_uniform}
    Let $ \e $ be  an admissible interaction functional with cube limit function $ \f $ and rate $ \t $.  
    Then 
    \[
        \lim_{N\to\infty} \frac{\e(\omega_N^*(A))}{\t(N)} = \f(\lambda_d (A))
    \]
    for any compact $ A \subset \mathbb R^d $. In particular, the above limit exists.
\end{theorem}
\begin{proof}
    Fix the compact set $ A\subset \mathbb R^d $ and a number $ \epsilon > 0 $. 
    We will approximate $ A $ with a collection of  cubes with pairwise disjoint interiors $ \Q $; by Lemma~\ref{lem:cube_volume}, the asymptotics of $ \e $ on such a collection is equal to $ \f(\lambda_d(\Q)) $, which then results in an estimate for the asymptotics on $ A $.

    Let $ \omega^{-1}(\f, \epsilon) $ be a number such that $|\f(t_1) - \f(t_2)| \leq \epsilon$,  whenever  $|t_1-t_2| \leq \omega^{-1}(\f, \epsilon)$ with $ t_1,t_2\in [\lambda_d(A)/2, 2\lambda_d(A)] $.
    Let further
    $$ \xi = \min \{ \delta(\epsilon), \ \omega^{-1}(\f, \epsilon) \}  ,$$
    where $ \delta(\epsilon) $ is taken as in \eqref{eq:cube_conts} of  Definition~\ref{def:cube_conts}.

    Suppose $ \sigma > 0 $ first. 
    For the chosen $ \epsilon $, there exists a finite cover of $ A $ with dyadic cubes, $ A \subset \Q = \bigcup_m Q_m $, such that $ \lambda_d(\Q) - \lambda_d(A) < \xi $. 
    Observe that this cover can be represented as a union of equal cubes with disjoint interiors, by re-partitioning some of the dyadic cubes, so $ \Q \in \mathscr L_\e $ by Lemma~\ref{lem:cube_uniform}.
    The inclusion $ A\subset \Q $ implies, together with monotonicity \eqref{eq:monotonicity},
    \begin{equation}
        \label{eq:lower_bd}
        \liminf_{N\to\infty} \frac{\e(\omega_N^*(A))}{\t(N)} \geq \lim_{N\to\infty} \frac{\e(\omega_N^*(\Q))}{\t(N)} = \f(\lambda_d(\Q)) \geq \f(\lambda_d(A)) -\epsilon
    \end{equation}
    where the last inequality follows from the choice of $ \xi $.  

    To obtain the upper bound, we will use Lebesgue density theorem and Besicovitch covering theorem in order to construct a collection of disjoint cubes covering a large portion of $ A $. Let 
    \[
        A' := \left\{ x\in A : \lim_{a\downarrow 0} \frac{\lambda_d( A\cap(x+a\q_d) )}{\lambda_d(x+a\q_d)} = 1 \right\}  .
    \]
    By the assumption $ \lambda_d(A)>0 $, so for $ \lambda_d $-almost every point $ x\in A $ there holds $ x\in A' $.
    The collection of cubes $ Q = (x+a\q_d) $, $ x\in A' $, with sufficiently small side lengths given by
    \[
        \mathscr Q_\xi = \left\{ Q :  \frac{\lambda_d( A\cap Q )}{\lambda_d(Q)} \geq 1-\xi \right\}
    \]
    is a Vitali cover for $ A' $. By the Vitali covering theorem \cite[\S 2.8]{federerGeometric1996a}, there is a countable subcollection $ \Q \subset \mathscr Q_\xi $ of disjoint cubes covering $ \lambda_d $-almost all of $ A' $. Finally, we can extract a finite subcollection of $ M $ cubes $ \{ Q_m \}_{m=1}^M = \Q_M \subset \Q $,  such that 
    \begin{equation}
        \label{eq:cubes_approximate}
        \lambda_d(A\setminus \Q_M) < \xi;
    \end{equation}
    assume in addition that the cubes $ Q_m $ in $ \Q_M $ are positive distance apart, by shrinking them towards center, if necessary, without violating that $ Q_m\in\mathscr Q_\xi $ and \eqref{eq:cubes_approximate}. Recall that by monotonicity~\eqref{eq:monotonicity}, 
    \[
        \limsup_{N\to \infty} \frac{\e(\omega_N^*(A))}{\t(N)} 
        \leq
        \liminf_{N\to \infty} \frac{\e(\omega_N^*({A\cap\Q_M}))}{\t(N)}.
    \]
    Consider the set $ A\cap\Q_M $. By the short-range property \eqref{eq:short_range}, for every sequence of the form $ \omega_N = \bigcup_{m=1}^M \omega_{N_m}^*({A\cap Q_m}) $ where $ N_m/N \to \beta_m >0 $ when $ N\to \infty $, $ \sum_m N_m = N $, there holds
    \[
        \lim_{N\to\infty} \frac{\sum_m \e(\omega_N\cap Q_m, A\cap Q_m)}{\e(\omega_N, A\cap \Q_M)} = 1.
    \]
    It follows that for such a sequence $ \omega_N $,
    \begin{equation*}
        \begin{aligned}
        &
        \limsup_{N\to\infty} \frac{\e(\omega_N,A\cap \Q_M)}{\t(N)}
        \\
        &= \limsup_{N\to\infty} \frac{\sum_{m=1}^M\e(\omega_N\cap Q_m, A\cap Q_m)}{\t(N)} =\limsup_{N\to\infty} \frac{\sum_{m=1}^M\e(\omega_{N_m}^*({A\cap Q_m}))}{\t(N)} \\
        &\leq \sum_{m=1}^M \limsup_{N\to\infty} \frac{\e(\omega_{N_m}^*({A\cap Q_m}))}{\t(N)} 
    \end{aligned}
    \end{equation*}
    Since all the cubes $ Q_m \in \Q_M $, $ 1\leq m \leq M $, belong to $ \mathscr Q_\xi $, the continuity property \eqref{eq:cube_conts} can be applied to yield
    \[
         \limsup_{N\to\infty}  \frac{\e(\omega_{N_m}^*(A\cap Q_m))}{\t(N_m)} 
         \leq
         ( 1 - \epsilon)^{-1} \lim_{N\to\infty} \frac{\e(\omega_N^*(Q_m))}{\t(N)} 
         =
         ( 1 - \epsilon)^{-1}\f(a_m^d), \qquad 1\leq m \leq M.
    \]
    Combining the two previous displays gives us
    \[
        \begin{aligned}
        \limsup_{N\to\infty} \frac{\e(\omega_N^*(A\cap \Q_M))}{\t(N)}
        &\leq
        \limsup_{N\to\infty} \frac{\e(\omega_N,A\cap \Q_M)}{\t(N)}\\
        &\leq 
        \sum_{m=1}^M \limsup_{N\to\infty} \frac{\e(\omega_{N_m}^*({A\cap Q_m}))}{\t(N_m)}\cdot \frac{\t(N_m)}{\t(N)} \\
        &\leq 
        ( 1 - \epsilon)^{-1} \sum_{m=1}^M \w(\beta_m) \f(a_m^d),
    \end{aligned}
    \]
    where in the last inequality  it is used that $ N_m/N\to \beta_m $ and \eqref{eq:rate}. Note that here the values of $ \beta_m >0 $ can be chosen arbitrarily, as long as $ \sum_m \beta_m = 1 $, by construction. Since the minimum of $ \sum_m \w(\beta_m) \f(a_m^d) $ occurs for all $ \beta_m $ positive by convexity, the above inequality holds in fact for any collection $ \beta_m \geq 0 $ summing up to $ 1 $. On the other hand, applying Lemma~\ref{lem:cube_uniform} to the finite union of cubes with disjoint interiors $ \Q_M $ gives
    \[
        \lim_{N\to\infty} \frac{\e(\omega_N^*(\Q_M))}{\t(N)} = \sum_{m=1}^M \w(\beta^*_m)\f(a_m^d)
    \]
    for some  $ \beta_m^* \geq 0 $ satisfying $ \sum_m \beta_m^*=1 $, which results in
    \begin{equation}
        \label{eq:upper_bd}
        \begin{aligned}
        \limsup_{N\to\infty} \frac{\e(\omega_N^*(A\cap \Q_M))}{\t(N)} 
        &\leq
        (1-\epsilon)^{-1} \lim_{N\to\infty} \frac{\e(\omega_N^*(\Q_M))}{\t(N)} 
        = (1-\epsilon)^{-1} \f(\lambda_d(\Q_M)) \\
        &\leq (1-\epsilon)^{-1} (\f(\lambda_d(A))+\epsilon),
    \end{aligned}
    \end{equation}
    where we used that $ \lambda_d(\Q_M) \geq \lambda_d(A) - \xi $ by \eqref{eq:cubes_approximate} and the definition of $ \xi $. Since $ \epsilon $ can be taken arbitrarily small in \eqref{eq:lower_bd} and \eqref{eq:upper_bd}, these two equations complete the proof in the case $ \sigma > 0 $.

    Suppose now $ \sigma < -1 $. Using monotonicity \eqref{eq:monotonicity} in this case gives  
    \[
        \limsup_{N\to\infty} \frac{\e(\omega_N^*(A))}{\t(N)} \leq \lim_{N\to\infty} \frac{\e(\omega_N^*(\Q))}{\t(N)} = \f(\lambda_d(\Q)) \leq \f(\lambda_d(A)) +\epsilon
    \]
    in place of \eqref{eq:lower_bd}. It remains to prove the matching lower bound. It will follow from a lower bound on the asymptotics on $ A\cap \Q_M $, which by monotonicity \eqref{eq:monotonicity} satisfies
    \[
        \liminf_{N\to \infty} \frac{\e(\omega_N^*(A))}{\t(N)} 
        \geq
        \limsup_{N\to \infty} \frac{\e(\omega_N^*({A\cap\Q_M}))}{\t(N)}.
    \]
    Passing to a subsequence if necessary, we can guarantee that the $ \liminf \frac{\e(\omega_N^*({A\cap\Q_M}))}{\t(N)} $  is attained along a subsequence $ \omega_N = \omega_N^*(A\cap \Q_M)  $, $ N\in \mathscr N $, such that along this subsequence the counting measures are convergent: $ N_m^*/N := \#(\omega_N\cap Q_m)/N \to \beta_m^* $, $ N\to\infty $.
    Arguing as above in the case $ \sigma > 0 $ for this subsequence $ \mathscr N $, we obtain:
    \[
        \begin{aligned}
        &
        \liminf_{N\to\infty} \frac{\e(\omega_N,A\cap \Q_M)}{\t(N)}=  \lim_{\mathscr N \ni N\to\infty} \frac{\e(\omega_N,A\cap \Q_M)}{\t(N)}
        \\
        &
        =\lim_{\mathscr N \ni N\to\infty} \frac{\sum_{m=1}^M\e(\omega_{N}^*({A\cap \Q_M})\cap Q_m, A\cap Q_m)}{\t(N)} 
        \geq \lim_{\mathscr N \ni N\to\infty} \sum_{m=1}^M  \frac{\e(\omega_{N_m^*}^*({A\cap Q_m}))}{\t(N)}\\ 
        &
        \geq (1+\epsilon)^{-1} \sum_{m=1}^M \liminf_{N\to\infty} \frac{\e(\omega_{N_m^*}^*({Q_m}))}{\t(N)} = (1 + \epsilon)^{-1} \sum_{m=1}^M \w(\beta_m^*) \f(a_m^d).
        \end{aligned}
    \]
    From Lemma~\ref{lem:cube_uniform},
    \[
        \f(\lambda_d(\Q_M)) = \lim_{N\to\infty} \frac{\e(\omega_N^*(\Q_M))}{\t(N)} \leq \sum_{m=1}^M \w(\beta_m^*) \f(a_m^d),
    \]
    which gives finally
    \[
        \liminf_{N\to\infty} \frac{\e(\omega_N,A\cap \Q_M)}{\t(N)} \geq (1+\epsilon)^{-1} \f(\lambda_d(\Q_M)) \geq (1+\epsilon)^{-1} (\f(\lambda_d(A))-\epsilon),
    \]
    completing the proof of the lower bound for $ \sigma < -1 $.
\end{proof}

The above theorem can be summarized by saying that the value of $ \l_\e(A) $ on a compact set $ A $ depends only on $ \lambda_d(A) $.  

\begin{corollary}
    \label{cor:compare_minimizers}
    Theorem~\ref{thm:general_uniform} and equation \eqref{eq:rate} imply:
    for a sequence $ n = tN+o(N) $ with $ t > 0 $ and a compact set $ A $ with $ \lambda_d(A) >0 $ there holds
    \[
        \lim_{N\to \infty} \frac{\e(\omega_{n}^*(A))}{\e(\omega_N^*(A))} = \w(t).
    \] 
\end{corollary}
It will further be convenient to extend the short-range and continuity properties~\eqref{eq:short_range}--\eqref{eq:cube_conts} to general compact sets as follows.
\begin{corollary}
    \label{cor:short_range_general}
    If compact sets $ A_1,A_2\subset\mathbb R^d $ with $ \lambda_d(A_m) > 0 $, $ m=1,2 $, are positive distance apart, then
    \[
        \lim_{N\to\infty} \frac{\sum_m \e(\omega_N\cap A_m, A_m )}{\e(\omega_N, A_1\cup A_2)} = 1
    \]
    for sequences of the two forms:

    \noindent{\bf (i)}
    $ \omega_N  = \omega_N^*(A_1\cup A_2)  $;

    \noindent{\bf (ii)}
    $ \omega_N = \bigcup_m \omega_{N_m}^*(A_m) $, with $ \sum_m N_m = N $.
\end{corollary}
\begin{corollary}
    \label{cor:continuity_general}
    The continuity property from Definition~\ref{def:cube_conts} applies to general compact sets: for a fixed compact $ A\subset \Omega $ with $ \lambda_d(A) > 0 $ there holds
    \begin{equation}
        \label{eq:continuity_general}
        \lim_{N\to\infty}   \sgn \varsigma \cdot  \frac{\e(\omega_N^*(A))}{\t(N)} \geq  
         \limsup_{N\to\infty}\, ( \sgn \varsigma - \epsilon) \cdot  \frac{\e(\omega_N^*(D))}{\t(N)} 
    \end{equation}
    as soon as the compact set $ D\subset A $ satisfies  $\lambda_d(D) \geq (1 - \delta( \epsilon) )\, \lambda_d(A)$.
\end{corollary}

Using the asymptotics theorem above we can also obtain a result about the weak$ ^* $ limit $ \mu $ of the counting measures of minimizers of $ \e $, stated below as Theorem~\ref{thm:weak_star_limit}. To avoid repetition, the proof of this result is omitted here; instead, we give a proof of the more general Theorem~\ref{thm:weighted_dist} in the next section, concerning weighted short-range interaction functionals. To obtain Theorem~\ref{thm:weak_star_limit} from Theorem~\ref{thm:weighted_dist}, it suffices to set the weight function $ \h \equiv 1 $ on $ A \subset \Omega $.


\begin{theorem}
    \label{thm:weak_star_limit}
    Suppose that $ \e $ is an admissible interaction functional with $ \f(t) = \f(1) t^{-\sigma} $, $ \sigma > 0 $, and that the compact $ A $ satisfies $ \lambda_d(A) > 0 $.
    Then the weak$^* $ limit of the minimizers $ \omega_N^*(A) $ is the probability measure on $ A $, proportional to the restriction of Lebesgue measure $ \lambda_d $. 
\end{theorem}

\section{Short-range interactions depending on position}
\label{sec:translation_dependent}
\subsection{Functionals with continuous positive weight}
In this section we modify property \eqref{eq:cube_asymp} in Definition~\ref{def:cube_asymp} to include a continuous weight function on $ \Omega $. The short-range functionals we consider here will come associated with the familiar triple of functions 
$$ (\f, \t, \w):[0,\infty)\to[0,\infty], $$ 
with an addition of the continuous 
$$ \h:\Omega\to[h_0,\infty], $$ 
bounded below by an $ h_0 > 0 $.

Since the weighted short-range functionals that are useful in practice arise from the unweighted short-range functionals discussed in Section~\ref{sec:general}, we assume that $ \f $, the rate $ \t $, and by extension the fractional rate $ \w $, are the same as for the corresponding unweighted functional. That is,
\[
    \f(t) = \f(1)t^{-\sigma}, \qquad \w(t) = t^{1+\sigma}
\]
for a $ \sigma\in(-\infty,-1)\cup(0,+\infty) $. 

Recall that whenever we discuss minimizers on a compact set $ A $, it is assumed that $ A\subset \Omega $ for a fixed open $ \Omega $. In the following definition, we write,
\[
    \begin{aligned}
    \underline\l_{\e_\h}(x+a\,\q_d) 
    &= \liminf_{N\to\infty} \frac{{\e_\h}(\omega_N^*(x+a\,\q_d))}{\t(N)}\\
    \overline\l_{\e_\h}(x+a\,\q_d) 
    &=\limsup_{N\to\infty} \frac{{\e_\h}(\omega_N^*(x+a\,\q_d))}{\t(N)}.
    \end{aligned}
\]

\begin{definition}
    \label{def:weighted}
    Suppose that $ \f(t) = \f(1)t^{-\sigma}$, $ \t $ has fractional rate $ \w = t^{1+\sigma} $, and $ \h $ is continuous on $ \Omega $.
    A lower semicontinuous set-monotonic functional $ \e_\h $ that satisfies Definition~\ref{def:short_range} and Definition~\ref{def:cube_conts}, as well as the equation 
    \begin{equation}
        \label{eq:weighted}
        \min_{y\in(x+a\,\q_d)} \h(y) \frac{\f(1)}{(a^d)^\sigma} 
        \leq \underline\l_{\e_\h}(x+a\,\q_d) 
        \leq \overline\l_{\e_\h}(x+a\,\q_d) 
        \leq  \max_{y\in(x+a\,\q_d)} \h(y) \frac{\f(1)}{(a^d)^\sigma}
    \end{equation}
    for all  $a>0$ and $ x\in\mathbb R^d $, is called an {\it admissible interaction functional  with weight $ \h $}, or {\it weighted} admissible functional. 
\end{definition}
To motivate property \eqref{eq:weighted}, observe that it has the form of the mean value theorem for integrals with $ \h $ as the integrand and \f(t) playing the role of a measure.
The first important result of this section is as follows.

\begin{theorem}
    \label{thm:weighted}
    Let $ {\e_\h} $ be  an admissible interaction functional with the continuous weight $ \h \geq h_0 > 0$, cube limit function $ \f $, and rate $ \t $.  
    Then for any compact $ A \subset \Omega $ with $ \lambda_d(A) > 0 $,
    \[
        \lim_{N\to\infty} \frac{{\e_\h}(\omega_N^*(A))}{\t(N)} = \frac{\f(1)}{ \left(\int_A \h^{-1/\sigma} \d\lambda_d\right)^\sigma } 
    \]
In particular, the above limit exists. 
\end{theorem}
Note that the expression on the right is the power mean $ \|\eta \|_{-1/\sigma} $, so in particular the mean value theorem applies to it.
Before we proceed to the proof, it will be convenient to formulate one of its components as a separate lemma. 
\begin{lemma}
    \label{lem:simple_weight}
    Suppose $ \Q = Q_1 \cup\ldots \cup Q_M  $ is a collection of cubes pairwise positive distance apart and $ \h $ is constant on each of them, $ \h(x)\equiv \h_m $, $ x\in Q_m $. Then for a compact $ A\subset \Q $,
    \[
        \lim_{N\to\infty} \frac{\e_\h(\omega_N^*(A))}{\t(N)}
        = \min \left\{ \f(1)\sum_{m=1}^M  \h_m \left(\frac{\beta_m}{\lambda_d(A\cap Q_m)}\right)^{1+\sigma}\lambda_d(A\cap Q_m) : \sum_{m=1}^M \beta_m =1, \ \beta_m \geq 0 \right\}.
    \]
    The minimum is attained for $ \beta_m $ proportional to $ \h_m^{-1/\sigma} \lambda_d(A\cap Q_m) $.
\end{lemma}
\begin{proof}
    Estimate $ \e_\h(\omega_N^*(A)) $ by $ \omega, $ the union of configurations minimizing $ \e_{\h_m} $ on each $ Q_m $. Apply Lagrange multipliers to the energy of $ \omega $.
\end{proof}
It is not hard to see that we have already used this property implicitly  for $ \e_1 $ in the proofs dealing with the unweighted interactions. 

\begin{proof}[Proof of Theorem~\ref{thm:weighted}]
    We proceed by forming Riemann sums for the weight $ \h $. 
    Observe that when the weight $ \h $ is constant, the statement of the theorem follows immediately from Theorem~\ref{thm:general_uniform} by rescaling the cube limit function $ \f $. We will use this fact to approximate a general $ \h $ with simple functions.

    Fix an $ \epsilon > 0 $ and let $ \Q =  Q_1\cup\ldots \cup Q_M  $ be a collection of sufficiently small dyadic cubes covering $ A $, such that for $ A_m := A\cap Q_m $, $ m=1,\ldots,M $, there holds
    \[
        \max_{x\in A_m} \h(x) - \min_{x\in A_m} \h(x) < \epsilon, \qquad m=1,\ldots,M.
    \]
    Denote $ H_m := \max_{x\in A_m} \h(x) $ and $  h_m:=\min_{x\in A_m} \h(x) $, respectively. 
    By an abuse of notation, we will write $ H_m(x) $ and $ h_m(x) $ for the following simple functions: 
    \[
        H_m(x) = \begin{cases} H_m, & x\in A_m;\\ 0 &\text{otherwise},\end{cases} \qquad
        h_m(x) = \begin{cases} h_m, & x\in A_m;\\ 0 &\text{otherwise}.\end{cases}
    \] 
    As usual, the cubes in $ \Q $ are assumed to have pairwise disjoint interiors.

    Up to passing to a subsequence, assume that $ N_m^* := \#(\omega_N^* \cap Q_m) $ satisfies $ N_m^*/N \to \beta_m^* \geq 0 $.
    By the short-range property (i) from Definition~\ref{def:short_range}, one has 
    \[
        \begin{aligned}
            \liminf_{N\to\infty} \frac{\e_\h(\omega_N^*(A))}{\t(N)} 
        &\geq  \sum_{m=1}^M \liminf_{N\to\infty} \frac{\e_\h(\omega_N^*(A)\cap Q_m, A_m)}{\t(N)} \geq \sum_{m=1}^M \liminf_{N\to\infty} \frac{\e_\h(\omega_{N_m^*}^*(A_m))}{\t(N)} \\
        &\geq \sum_{m=1}^M \lim_{N\to\infty} \frac{\e_{h_m}(\omega_{N_m^*}^*(A_m))}{\t(N)} 
        = \sum_{m=1}^M h_m \w(\beta_m^*) \f(\lambda_d(A_m)) \\
        &\geq \frac{\f(1)}{\left(\sum_{m=1}^M h_m^{-1/\sigma} \lambda_d(A_m)\right)^\sigma },
        \end{aligned}
    \] 
    where the last inequality uses the minimal value of $ \sum_m h_m \w(\beta_m) \f(\lambda_d(Q_m)) $ over nonnegative $ \beta_m $ such that $ \sum_m \beta_m =1 $, computed in Lemma~\ref{lem:simple_weight}.

    It remains to obtain the matching upper bound. Let $ \sigma > 0 $ first.
     By the continuity property from Corollary~\ref{cor:continuity_general}, there exists a $ \gamma \in(0,1) $ such that for $ m=1,\ldots,M $, the rescaled cube $ \gamma Q_m $ with the same center as $ Q_m $ satisfies 
     \begin{equation}
         \label{eq:scaled_cubes}
         \lambda_d(A\cap Q_m) - \lambda_d(A\cap \gamma Q_m) < \delta(\epsilon) \lambda_d(A\cap Q_m), \qquad m=1,\ldots,M,
     \end{equation}
     with $ \delta(\epsilon) $ as in \eqref{eq:continuity_general}.
     Let 
     $$ A_m' := A\cap \gamma Q_m, \qquad  m=1,\ldots,M  $$
     and $ A': = \bigcup_m A_m' $, and observe that $ A_m' $ are positive distance apart, so Lemma~\ref{lem:simple_weight} applies to $ A' $.

     Denote
     $$ H(x):= \sum_{m=1}^M H_m(x) \qquad  h(x):= \sum_{m=1}^M h_m(x), $$
     then for $ x\in A' $ there holds $ h(x) \leq \h(x)\leq H(x)$.

     From the set-monotonicity \eqref{eq:monotonicity}, Lemma~\ref{lem:simple_weight}, and general continuity property \eqref{eq:continuity_general},  we have
    \[
        \begin{aligned}
        \limsup_{N\to\infty} \frac{{\e_\h}(\omega_N^*(A))}{\t(N)} 
        &\leq \liminf_{N\to\infty} \frac{{\e_\h}(\omega_N^*(A'))}{\t(N)} \leq \lim_{N\to\infty} \frac{{\e_H}(\omega_N^*(A'))}{\t(N)}\\
        &= \min \left\{ \f(1)\sum_{m=1}^M  H_m \left(\frac{\beta_m}{\lambda_d(A_m')}\right)^{1+\sigma}\lambda_d(A_m') : \sum_{m=1}^M \beta_m =1, \ \beta_m \geq 0 \right\}\\
        &\leq (1+\epsilon)\cdot \frac{\f(1)}{\left( \sum_{m=1}^M H_m^{-1/\sigma} \lambda_d(A_m) \right)^\sigma}.
    \end{aligned}
    \]
    Since $ \epsilon > 0 $ is arbitrary, both the upper and the lower bound obtained here converge to 
    \[
        \lim_{N\to\infty} \frac{{\e_\h}(\omega_N^*(A))}{\t(N)} = \frac{\f(1)}{ \left(\int_A \h^{-1/\sigma} \d\lambda_d\right)^\sigma }
    \]
    by the dominated convergence theorem.  We have verified the claim of the theorem in the case $ \sigma > 0 $. 

    For $ \sigma < -1 $, the upper bound is obtained by comparison with asymptotics of the configuration $ \omega_N = \bigcup_{m=1}^M \omega_{N_m^*}^*({A_m}) $, where $ N_m^* := \#(\omega_N^* \cap Q_m) $, as above in this proof. Applying the short-range property~\eqref{eq:short_range}, one has, up to passing to a subsequence,
    \[
        \begin{aligned}
            \limsup_{N\to\infty} \frac{\e_\h(\omega_N^*(A))}{\t(N)} 
        &\leq \limsup_{N\to\infty} \frac{\e_\h(\omega_N(A),A)}{\t(N)}\\
        &\leq\sum_{m=1}^M \limsup_{N\to\infty} \frac{\e_\h(\omega_N(A)\cap Q_m, A_m)}{\t(N)} 
        \leq \sum_{m=1}^M \lim_{N\to\infty} \frac{\e_{H_m}(\omega_{N_m^*}^*(A_m))}{\t(N)}
        \\
        & = \sum_{m=1}^M H_m \w(\beta_m^*) \f(\lambda_d(A_m)) 
        \leq \frac{\f(1)}{\left(\sum_{m=1}^M H_m^{-1/\sigma} \lambda_d(A_m)\right)^\sigma },
        \end{aligned}
    \] 
    where $ \beta_m^*=\lim_{N\to\infty} N_m^*/N $, as defined previously in this proof. Note that in the above display we are using the standard shorthand for minimizers
    \[
        \e_{H_m}(\omega_{n}^*(A_m)) = \min \{ \e_{H_m}(\omega_{n}^*(A_m)) : \omega_n \subset A_m \}.
    \]
    An application of the dominated convergence theorem completes the proof of the theorem for $ \sigma < -1 $.
    \end{proof}
    Theorem~\ref{thm:weighted} shows in particular that the class of sets $ \mathscr L_{\e_\h} $ on which the asymptotics exist includes all compact sets in $ \mathbb R^d $. As a result, Lemma~\ref{lem:simple_weight} can be formulated for a collection of arbitrary compact  sets that are pairwise positive distance apart. In Theorem~\ref{thm:weighted_dist} we will need this extended formulation of the lemma.

    Before we proceed to determine the weak$ ^* $ limit of the minimizing configurations $ \mu $, let us first show that of the measures $ \mu $ and $ \lambda_d $ ($ d $-dimensional Lebesgue measure in $ \mathbb R^d $), one is always absolutely continuous with respect to the other.
    \begin{lemma}
        \label{lem:abs_continuity}
        Let $ \e_\h $ be a short-range interaction functional with $ \f(t) = \f(1)t^{-\sigma} $, $ \sigma > 0 $, and a continuous weight $ \h \geq h_0 > 0 $.
        Then any weak$ ^* $ cluster point $ \mu $ of the counting measures of minimizers $ \omega_N^*(A) $ satisfies
        \[
            \mu \ll \lambda_d.
        \]
    \end{lemma}
    \begin{proof}
        By weak$ ^* $ compactness of Borel probability measures on a compact set, $ \mu $ is itself a Borel probability measure.  First, we will show that $ \sigma > 0 $ implies $ \mu $ is absolutely continuous with respect to  $ \lambda_d $. Indeed, if for a compact set $ B\subset A $, $ \mu(B) > 0 $ and $ \lambda_d(B)=0 $, then for any $ \delta >0 $ one can construct a finite collection of dyadic cubes $ \Q = \cup_m Q_m $  with 
        $$ \lambda_d \left(\Q \right) < \delta, $$
        such that 
        $$ \mu\left(B\cap \Q \right) > \mu(B)/2 $$
        and $ B\cap \Q $ is a positive distance from  the complement of $ \Q $. Then, by \eqref{eq:split_union},
        \begin{equation}
            \label{eq:split_continuity}
            \l_{\e_\h}(A\setminus \Q \,\cup\, B\cap \Q) = \w(\mu(A\setminus \Q)) \cdot \l_{\e_\h}(A\setminus \Q ) + \w(\mu(B\cap\Q)) \cdot \l_{\e_\h}(B\cap\Q).
        \end{equation}
        We will now need that $ \lim_{t\to0+} \f(t) = +\infty $. Since $ \w(t) = t^{1+\sigma} $, it suffices to choose $ \delta $ so small that 
        $$ \l_{\e_\h}(B\cap\Q) \geq  \f(1)(\delta)^{-\sigma} > \frac{\l_{\e_\h}(A) + 1}{(\mu(B)/2)^{1+\sigma}} > \frac{\l_{\e_\h}(A)+1}{ \mu\left(B\cap \Q \right)^{1+\sigma}} . $$
        On the other hand, by property \eqref{eq:cube_conts} for sufficiently small $ \delta $ there holds
        \[
            \l_{\e_\h}(A\setminus \Q \,\cup\, B\cap\Q) \leq \l_{\e_\h}(A\setminus\Q)  < \l_{\e_\h}(A) + 1.
        \]
        Combining the last three equations, we obtain a contradiction, proving that $ \mu(B) =0 $ implies $ \lambda_d(B) =0 $. Note that in the preceding argument we assumed that the set $ B $ is compact; for a general $ B' $ one can find a compact $ B\subset B' $ as above by the inner regularity of $ \mu $. Indeed, any Borel measure on a metric space is inner regular \cite[Thm. 1.1]{billingsley1999convergence}.
    \end{proof} 
    Observe that the roles of measures $ \lambda_d $ and $ \mu $ interchange for $ \sigma < -1 $, since in equation \eqref{eq:split_continuity}, the last term is then equal to
    \[
        \frac{\f(1)\lambda_d(B\cap\Q)^\sigma}{\mu(B\cap\Q)^{-\sigma-1}},
    \]
    which becomes large when $ \mu(B\cap\Q) $ is small.
    \begin{corollary}
        \label{cor:lambda_ll_mu}
        When the assumption $ \sigma < -1 $ is used in Lemma~\ref{lem:abs_continuity}, it follows
        \[
            \lambda_d \ll \mu.
        \]
    \end{corollary}
    The proof of Theorem~\ref{thm:weighted_dist} will proceed by using that the limits of counting measures of minimizers optimize the convex continuous functions of the form \eqref{eq:split_min}. Uniqueness of the minimizer of a strictly convex function will then imply that $ \mu $ is uniform with respect to $ \lambda_d $. 

\begin{theorem}
    \label{thm:weighted_dist}
    Suppose that $ \e_\h $ is a short-range interaction functional with $ \f(t) = \f(1)t^{-\sigma} $  and a continuous weight $ \h $ bounded below by a positive number, and that the compact $ A $ satisfies $ \lambda_d(A) > 0 $. 
    Then the weak$^* $ limit of the minimizers $ \omega_N^*(A) $, $ N\geq N_0(\e) $, exists and is the probability measure $ \mu $ such that $ d\mu/d\lambda_d $ is proportional to $ \h^{-1/\sigma} $.  
\end{theorem}
\begin{proof} 
    In this proof, for a cube $ Q_l = x_l+ a\q_d $, $ l=1,2 $, we write $ \gamma Q_l $ to denote the rescaled cube with the same center: $  x+ \gamma a\q_d $, and the same convention applies to $ \Q = Q_1\cup Q_2 $. The notations $ n $,  $n'$,  $ n_1,$  $ n_2  $ are used for members of sequences of natural numbers, depending on $ N $, $\gamma $, and $ a $:  $n=n(N,\gamma,a)$, etc. These sequences do not have to be defined for all $ N $, and their precise properties will transpire from the proof. Quantities $ \beta_l $ are nonnegative and depend on $ \gamma $ and $ a $.

    By Lemma~\ref{lem:abs_continuity} and Corollary~\ref{cor:lambda_ll_mu}, any weak$ ^* $ cluster point $ \mu $ of minimizers $ \omega_N^*(A) $  satisfies either $ \mu \ll \lambda_d $ or $ \lambda_d \ll \mu $. Hence the limit
    \[ 
        \frac{d\mu}{d\lambda_d}(x) = \lim_{a\downarrow 0}\frac{\mu(x+a\q_d)}{\lambda_d(x+a\q_d)};
    \]
    exists in $ [0,\infty] $ for $ \lambda_d $-a.e. $ x\in  A $ \cite[Thm. 2.12]{mattila1995geometry}. Furthermore, by the Lebesgue density theorem \cite[Cor. 2.14]{mattila1995geometry},
    \begin{equation}
        \label{eq:lebesgue_density}
        \lim_{a\downarrow 0}\frac{\lambda_d((x+a\q_d)\cap A)}{\lambda_d(x+a\q_d)} = 1, \qquad \text{for }  \lambda_d\text{-a.e. }  x\in  A .
    \end{equation}

    We have that in $ A\times A $ for $ \lambda_d\otimes\lambda_d $-almost every pair $ x_1, x_2\in  A $ both the above limits hold; assume that $ x_1,x_2 $ is such a pair. Let $ D_1 =  d\mu/d\lambda_d(x_1) $ and $ d\mu/d\lambda_d(x_2) = D_2 $; suppose also that $ D_1 $ is finite. If this is not the case for all $ x $ for which the two preceding limits exist, one has $ d\mu/d\lambda_d(x) = +\infty $ for $ \lambda_d $-a.e. $ x\in A $, so necessarily $ \sigma < -1 $. We will consider this case at the end of the proof.
    
    It follows that 
    \begin{equation}
        \label{eq:limiting_ratio}
        \lim_{a\downarrow 0}\frac{\mu(x_1+a\q_d)}{\mu(x_2+a\q_d)} = \frac{D_1}{D_2}
    \end{equation}
    is well-defined.
    Note also that
    $$ \mu(\partial(x_l+a\q_d)) = 0 $$
    holds except possibly for a countable collection of $ a $. We will assume in the rest of the proof that the relevant cubes are always chosen so that the above equality is satisfied. For a sufficiently small $ a_0 $, the cubes $ Q_l:= (x_l+a_l\q_d)$, $l=1,2$ with $ 0 <  a < a_0 $ are positive distance apart and do not cover $  A $. 
    Let
    \[
        \underline \h_l = \min_{Q_l} \eta(x)  \qquad \text{and} \qquad   \overline \h_l = \max_{Q_l} \eta(x), \quad l =1,2,
    \] 
    then $ \underline \h_l \to \h(x_l) $ and $ \overline \h_l \to \h(x_l) $ when $ a\downarrow0 $.

    For a sequence of minimizers $ \omega_N^* $, $ N\in \mathscr N $ whose counting measures weak$ ^* $ converge to $ \mu $, the short-range property \eqref{eq:short_range} together with Theorem~\ref{thm:weighted} gives for the union of cubes $ \Q := Q_1\cup Q_2 $:
    \begin{equation} 
        \label{eq:compare_1}
        \begin{aligned}
        \liminf_{\mathscr N \ni N\to\infty}
        &\frac{\e_\h(\omega_{N}^*\cap \Q,\,A\cap \Q)}{\t(N)} \geq \\
        & \geq 
        \f(1)\left(\underline \h_1 \mu(Q_1)^{1+\sigma} \cdot (\lambda_d[ Q_1\cap A])^{-\sigma}  
        + \underline \h_2 \mu(Q_2)^{1+\sigma} \cdot (\lambda_d[ Q_2\cap A])^{-\sigma} \right)
        \\
        & =: \underline L(a).
        \end{aligned}
    \end{equation}
    By \eqref{eq:limiting_ratio} and \eqref{eq:lebesgue_density}, 
    \begin{equation}
        \label{eq:energy_density_lower}
        \lim_{a\downarrow0} \frac{\underline L(a)}{a^d} = 
        \f(1)\left( \h(x_1) D_1^{1+\sigma}  
        +  \h(x_2) D_2^{1+\sigma} \right).
    \end{equation}

    To obtain a matching upper bound on the expression in \eqref{eq:energy_density_lower}, we next consider a configuration consisting of local minimizers on $ \Q $ and $ A\setminus \Q $. Fix $ \gamma_0\in(0,1) $ and let
    \[
        \gamma = \begin{cases}
            \gamma_0 \in (0,1), & \sigma > 0;\\
            1,& \sigma < -1.
        \end{cases}
    \]
    For $ \sigma>0  $ we will eventually set $ \gamma\uparrow1 $. Let $ \gamma\Q = (x_1+\gamma a\q_d) \cup (x_2+\gamma a\q_d) $, as introduced above.
    Apart from the set $ \gamma \Q \in \mathscr Q $, we construct a collection of cubes with disjoint interiors covering $ A\setminus \Q $; denote this collection by $ \Q' $. By taking the cubes in $ \Q' $ small enough, we ensure that $ \gamma\Q $ and $ \Q' $ are positive distance apart. Let $ A' = A\cap(\gamma\Q\cup \Q') $.

    Denote $ n = \#(\omega_N^*(A)\cap \overset{\circ}{\Q} ) $ and $ n'=\#(\omega_N^*(A)\cap \Q' ) $, so that $ n+n' = N $. 
    Let further $ n_1,$  $ n_2,$ be positive integers such that $ n_1+ n_2 = n $ and 
    \[
        \lim_{N\to \infty} \frac{n_l}{N}  = \beta_l, \qquad l =1,2,
    \]
    with $ \beta_l $ to be specified later. By construction, $ \beta_1+\beta_2 = \mu(\Q) $, and by definition of $ D_l $, there holds $ (\beta_1+\beta_2)/a^d \to (D_1+D_2) $, $ a\downarrow0 $. Without loss of generality, we assume that both limits
    \[
        \lim_{a\downarrow 0}\beta_l / a^d =: b_l, \qquad l = 1,2,
    \]
    exist.
    The upper bound on the right-hand side of \eqref{eq:energy_density_lower} is obtained by comparing the value of $ \e_\h $ on $ \omega_N^*(A) $ with that on $ \omega_N $, consisting of local minimizers on $ \gamma\Q $ and its complement to $ A' $:
    \[
        \omega_N = \omega_{n_1}^*(A\cap \gamma Q_1)  \cup \omega_{n_2}^*(A\cap \gamma Q_2) \cup \omega_{n'}^*(A\cap\Q').
    \]

    Let $ {\mathscr N}_{sup} \subset \mathbb N $ be the subsequence such that
    \[
        \lim_{{\mathscr N}_{sup} \ni N\to\infty} \frac{\e_\h(\omega_{N}^*\cap \Q,\,A\cap \Q)}{\t(N)}  = 
        \limsup_{\mathscr N \ni N\to\infty} \frac{\e_\h(\omega_{N}^*\cap \Q,\,A\cap \Q)}{\t(N)}
    \]
    Passing to the limit along $ {\mathscr N}_{sup} $, from property \eqref{eq:short_range} we deduce, using the sets $ \Q , \Q' \in \mathscr Q  $ for $ \omega_{N}^*(A) $ and $ \gamma\Q, \Q' \in \mathscr Q $ for $ \omega_N $ respectively: 
    \begin{equation*} 
        \begin{aligned}
        \lim_{N\to\infty}
        &\frac{\e_\h(\omega_{N}^*(A),\, A)}{\t(N)} \geq 
        \limsup_{\mathscr N \ni N\to\infty} \frac{\e_\h(\omega_{N}^*(A)\cap\Q,\, \Q)}{\t(N)} 
        +\liminf_{{\mathscr N}_{sup} \ni N\to\infty} \frac{\e_\h(\omega_{N}^*(A)\cap(\overline{A\setminus\Q}),\, (\overline{A\setminus\Q}))}{\t(N)}    \\
        \lim_{N\to\infty}
        &\frac{\e_\h(\omega_{N},\, A')}{\t(N)} = 
        \lim_{N\to\infty} \frac{\e_\h(\omega_N\cap\gamma\Q,\, \gamma\Q)}{\t(N)} 
        +\lim_{N\to\infty} \frac{\e_\h(\omega_{n'}^*(\overline{A\setminus\Q}))}{\t(N)},
        \end{aligned}
    \end{equation*}
    where the existence of the limit in the left-hand side of the second equation follows from existence of limits in the corresponding right-hand side, in particular from the existence of the limits of counting measures on $ \Q $ and $ \Q' $.
    Since $ A' =A $ for $ \sigma < -1 $, the set monotonicity~\eqref{eq:monotonicity} gives
    \[
        \lim_{N\to\infty} \frac{\e_\h(\omega_{N}^*(A),\, A)}{\t(N)} \leq \lim_{N\to\infty} \frac{\e_\h(\omega_{N},\, A)}{\t(N)} \leq 
        \lim_{N\to\infty} \frac{\e_\h(\omega_{N},\, A')}{\t(N)}.
    \]
    On the other hand, minimizers $\omega_{n'}^*(\overline{A\setminus\Q})$ are optimal on the set $ \overline{A\setminus\Q} $, implying
    \[
        \liminf_{{\mathscr N}_{sup} \ni N\to\infty} \frac{\e_\h(\omega_{N}^*(A)\cap(\overline{A\setminus\Q}),\, (\overline{A\setminus\Q}))}{\t(N)}  \geq \lim_{N\to\infty} \frac{\e_\h(\omega_{n'}^*(\overline{A\setminus\Q}))}{\t(N)},
    \]
    so that there must hold, by property \eqref{eq:short_range},
    \begin{equation}
        \label{eq:compare_2}
        \begin{aligned}
        \limsup_{\mathscr N \ni N\to\infty} \frac{\e_\h(\omega_{N}^*\cap \Q,\,A\cap \Q)}{\t(N)} 
        &\leq 
        \lim_{N\to\infty} \frac{\e_\h(\omega_N\cap\gamma\Q,\, \gamma\Q)}{\t(N)} \\
        &= 
        \lim_{N\to\infty} \frac{\e_\h(\omega_{n_1}^*(\gamma Q_1))}{\t(N)} + \lim_{N\to\infty} \frac{\e_\h(\omega_{n_2}^*(\gamma Q_2))}{\t(N)}\\
        &\leq 
        \f(1)\left(\overline \h_1 \beta_1^{1+\sigma} \cdot (\lambda_d[\gamma Q_1\cap A])^{-\sigma}  
        + \overline \h_2 \beta_2^{1+\sigma} \cdot (\lambda_d[\gamma Q_2\cap A])^{-\sigma} \right)\\
        & =:\overline L(a, \gamma,\beta_1,\beta_2).
        \end{aligned}
    \end{equation}
    As above for equation~\eqref{eq:energy_density_lower}, applying~\eqref{eq:limiting_ratio} and~\eqref{eq:lebesgue_density} gives 
    \begin{equation}
        \label{eq:energy_density_upper}
        \lim_{a\downarrow0} \frac{\overline L(a, \gamma,\beta_1,\beta_2)}{a^d} = 
        \gamma^{-\sigma}\f(1)\left( \h(x_1) b_1^{1+\sigma}  
        +  \h(x_2) b_2^{1+\sigma} \right).
    \end{equation}
    Using that for $ \sigma > 0 $, \eqref{eq:energy_density_upper} holds for any $ \gamma\in(0,1) $, combined with \eqref{eq:energy_density_lower} it gives
    \[
        \h(x_1) D_1^{1+\sigma}  +  \h(x_2) D_2^{1+\sigma}  \leq 
        \h(x_1) b_1^{1+\sigma}  
        +  \h(x_2) b_2^{1+\sigma} 
    \]
    for all $ \sigma \in (-\infty, -1)\cup(0,\infty) $ and all $ b_l \geq 0 $ with $ b_1+b_2 = D_1 + D_2 $. It implies
    \begin{equation}
        \label{eq:local_optimality}
        \h(x_1) D_1^{1+\sigma}  +  \h(x_2) D_2^{1+\sigma}  =  
        \min \left\{ \h(x_1) b_1^{1+\sigma}  +  \h(x_2) b_2^{1+\sigma} : b_1+b_2 = D_1 + D_2\right\}
    \end{equation}
    whence
    \[
        \frac{D_1}{D_2} = \frac{{\h(x_1)}^{-1/\sigma}}{{\h(x_2)}^{-1/\sigma}},
    \]
    assuming that both $ D_l $ are finite. If on the other hand $ D_2 = +\infty $, it must be $ \sigma < -1 $, and the minimum in~\eqref{eq:local_optimality} occurs for $ D_1 = D_2 = +\infty $. It follows that $ d\mu(x)/d\lambda_d(x) = +\infty $ for $ \lambda $-almost all $ x\in A $. Equivalently, $ d\lambda(x)/d\mu(x) = 0 $ for $ \lambda $-almost all $ x\in A $ This is a contradiction to $ \lambda_d \ll \mu $ and $ \lambda_d(A)>0 $.

    Note that the case when $ d\mu/d\lambda_d(x) = +\infty $ for $ \lambda_d $-a.e. $ x\in A $, which was previously postponed, has therefore been considered as well, and the proof is now complete.
\end{proof}
\begin{remark}
    \label{rem:nice_weighted_expression}
    Using the expression for the density of the weak$ ^* $ limit of the minimizers $ \mu $ from Theorem~\ref{thm:weighted_dist}, the value of the asymptotics from Theorem~\ref{thm:weighted} can be expressed as
    \[
        \lim_{N\to\infty} \frac{{\e_\h}(\omega_N^*(A))}{\t(N)} = \f(1) \int_A \h(x)\left(\frac{d\mu}{d\lambda}(x)\right)^{1+\sigma}\, d\lambda_d(x).
    \]
    Indeed, to verify this, substitute the expression for the density of $ \mu $ from Theorem~\ref{thm:weighted_dist} into the above display:
    \[
        \frac{d\mu}{d\lambda}(x) = \frac{\h^{-1/\sigma}(x)}{\int_A\h^{-1/\sigma}(x)\,d\lambda_d(x)}.
    \]
\end{remark}

\subsection{Combining weight with external field}
\label{sec:ext}
Theorems~\ref{thm:weighted} and \ref{thm:weighted_dist} can be naturally extended to a more general class of functionals $ \e_{\h,\xi} $, which are equipped with both weight and an external field (often called {\it confining potential})
\[
    \xi:\Omega\to[0,\infty].
\]
Such $ \e_{\h,\xi} $ must satisfy the short range and continuity properties~\eqref{eq:short_range}--\eqref{eq:cube_conts}, and in addition instead of the asymptotics on cubes must depend on $ {\h,\xi} $ as in \eqref{eq:ext_field} below. 
As before, we consider compact sets $ A\subset \Omega $ contained in the fixed open $ \Omega $, and use the standard notation for $ \liminf $ and $ \limsup $ of $ \e_{\h,\xi} $:
\[
    \begin{aligned}
    \underline\l_{\e_{\h,\xi}}(x+a\,\q_d) 
    &= \liminf_{N\to\infty} \frac{{\e_{\h,\xi}}(\omega_N^*(x+a\,\q_d))}{\t(N)}\\
    \overline\l_{\e_{\h,\xi}}(x+a\,\q_d) 
    &=\limsup_{N\to\infty} \frac{{\e_{\h,\xi}}(\omega_N^*(x+a\,\q_d))}{\t(N)}.
    \end{aligned}
\]
\begin{definition}
    \label{def:ext_field}
    Suppose that $ \f(t) = \f(1)t^{-\sigma} $, $ \t $ has the fractional rate $ \w = t^{1+\sigma} $, functions $ \h \geq \h_0>0$ and $ \xi\geq0 $  are continuous on $ \Omega $.
    A functional $ \e_{\h,\xi} $ that satisfies Definitions~\ref{def:short_range}--\ref{def:cube_conts}, and such that
    \begin{equation}
    \label{eq:ext_field}
    \begin{aligned}
        &\min_{y\in(x+a\,\q_d)} \h(y) \frac{\f(1)}{(a^d)^\sigma} 
        + \min_{y\in(x+a\,\q_d)}\xi(y)\\
        &\leq \underline \l_{\e_{\h,\xi}}(x+a\,\q_d) 
        \leq \overline \l_{\e_{\h,\xi}}(x+a\,\q_d)\\ 
        &\leq  \max_{y\in(x+a\,\q_d)} \h(y) \frac{\f(1)}{(a^d)^\sigma}
        + \max_{y\in(x+a\,\q_d)}\xi(y)
    \end{aligned}
    \end{equation}
    for all  $a>0$ and $ x\in\mathbb R^d $, is called a short-range interaction functional with weight $ \h $ {\it  and external field $ \xi $}
\end{definition}
The results about asymptotics of the minimal value of $ \e_{\h,\xi} $ on compact $ A $, similar to Theorems~\ref{thm:weighted} and \ref{thm:weighted_dist}, can easily be obtained along the same lines. Indeed, the case of a $ \xi $ identically constant on $ A $ is trivial. The main observation one has to make is that the dependence of the asymptotics on limits of counting measures is convex in the case of $ A $ being a union of compact cubes positive distance apart; namely we have the following generalization of Lemma~\ref{lem:simple_weight}:
\begin{lemma}
    Suppose $ \Q = \{ Q_1,\ldots,Q_M \}  $ is a collection of cubes pairwise positive distance apart. Then
    \[
        \begin{aligned}
        \lim_{N\to\infty}
        &\frac{\e_{\h,\xi}(\omega_N^*(A))}{\t(N)}\\
        &= \min \left\{ \f(1)\sum_{m=1}^M  \h(y_m) \left(\frac{\beta_m}{\lambda_d(Q_m)}\right)^{1+\sigma}\lambda_d(Q_m) +\beta_m \xi(z_m) : \sum_{m=1}^M \beta_m =1, \ \beta_m \geq 0 \right\}
        \end{aligned}
    \]
    for some $ y_m, z_m \in Q_m $. The minimum is attained for 
    \[
        \beta_m = \left(\frac{L_1-\xi(z_m)}{\f(1)(1+\sigma)\h(y_m)}\right)^{1/\sigma}_+ \lambda_d(Q_m),
    \]
    where $ L_1 $ is a normalizing constant, chosen so that $ \beta_m $ add up to 1, and $ (\cdot)_+ $ denotes the positive part.
\end{lemma}
The proof of the above lemma, just as that of Lemma~\ref{lem:simple_weight}, consists of applying Lagrange multipliers to the function in the minimum. One can repeat the proofs of Theorems~\ref{thm:weighted}~and~\ref{thm:weighted_dist} with the obvious adjustments, such as the right-hand side of estimates \eqref{eq:compare_1} and \eqref{eq:compare_2} now being in terms of the expression
\[
    \f(1)
    \left(
        \h(x_1) \frac{\mu(Q_1)^{1+\sigma}}{\lambda_d( Q_1\cap A)^{\sigma}}  +
        \h(x_2) \frac{\mu(Q_2)^{1+\sigma}}{\lambda_d( Q_2\cap A)^{\sigma}}  
    \right)
    + \mu(Q_1) \xi(x_1) + \mu(Q_2) \xi(x_2).
\]
The rest of the proofs for $ {\e_{\h,\xi}} $ follow Theorems~\ref{thm:weighted}~and~\ref{thm:weighted_dist}, and will be omitted here. We state the counterparts of the results for $ {\e_{\h}} $ as follows.
\begin{theorem}
    \label{thm:ext_dist}
    Let $ {\e_{\h,\xi}} $ be  a short-range interaction functional with cube limit function $ \f(1)t^\sigma $ and rate $ \t $, equipped with a continuous weight $ \h $ and an external field $ \xi $.  
    Then for any compact $ A \subset \Omega $,
    \[
        \lim_{N\to\infty} \frac{{\e_{\h,\xi}}(\omega_N^*(A))}{\t(N)} =
        \f(1) \int_A \h(x)\,\pphi(x)^{1+\sigma}\, d\lambda_d(x) + \int_A \xi(x) \, \pphi(x)\, d\lambda_d(x),
    \]
    where $ \pphi $ is the density of a probability measure $ \mu $ supported on $ A $, and is given by    
    \[
        \pphi(x) = \frac{d\mu}{d\lambda_d}(x) = \left(\frac{L_1-\xi(x)}{\f(1)(1+\sigma)\h(x)}\right)^{1/\sigma}_+
    \]
    for a normalizing constant $ L_1 $.
    Moreover, any sequence of minimizers of $ {\e_{\h,\xi}} $ on $ A $ converges weak$ ^* $ to $ \mu $.
\end{theorem}
The assumption of continuity of $ \xi $ can be removed. The more general case of lower semicontinuous $ \xi $ has been considered in \cite{hardinGenerating2017} for the short-range functional $ \e_{\h,\xi} $ given by the hypersingular Riesz energy:
\[
    \sum_{i\neq j} \|x_i - x_j\|^{-s} + N^{s/d} \sum_{i} \xi(x_i), \qquad s > d.
\]
In the present paper we only deal with continuous $ \xi $ so as to streamline the exposition; lower semicontinuous $ \xi $ can be handled similarly to \cite{hardinGenerating2017}.

A straightforward way to equip a (weighted) short-range functional with an external field is to add the sum $ \sum_i \xi(x_i) $ over the points in $ \omega_N $. Indeed, for a constant external field this fact is trivial; the general case is  a consequence of the previous theorem and the proof of Theorems~\ref{thm:weighted},~\ref{thm:weighted_dist}. As a result we have the following fact.
\begin{proposition}
    \label{prop:adding_ext_field}
    Suppose a short-range functional  $ {\e_{\h,\xi}} $ is given. Then for any short-range interaction functional $ {\e_{\h}} $ with the same weight $ \h $, cube limit function $ \f $, and rate $ \t $, but no external field, the asymptotics and the weak$ ^* $ limit of minimizers of $ {\e_{\h,\xi}} $ and $ \e_\h +  \t(N)/\xi $, defined by
    \[
        (\e_\h + \xi)(\omega_N,A):= \e_\h(\omega_N,A) + \frac{\t(N)}N \sum_{i=1}^N \xi(x_i),
    \]
    coincide.
\end{proposition}
\noindent We will see in Section~\ref{sec:examples} that adding a weight $ \h $ to a short-range functional $ \e $ can be done by multiplying an appropriate part of $ \e $ by a continuous function $ \h $.

Outside of this section, we use $ \e $ to denote functionals that can include weight or external field, whenever this cannot result in a confusion. That is, for a fixed continuous $ \h:\Omega \to [h_0,\infty] $ and $ \xi:\Omega \to [0,\infty] $, we write
\[
    \e = \e_{\h,\xi}.
\]

\section{Generalizations}
\subsection{Short-range functionals on embedded sets}
\label{sec:embedded}
In this section we consider the asymptotics of interaction functionals, defined on compact subsets of $ \mathbb R^{d'} $ of Hausdorff dimension $ d < d' $. To distinguish them from the interactions depending on the full-dimensional sets, we denote such functionals $ \ee(\omega_N, A) $. Note that in the prior discussion, the assumption $ \lambda_d(A) > 0 $ implied that the Hausdorff dimension of $  A $ satisfies $ \dim_H( A) = d $, the dimension of the ambient space. Now we deal with the case when the dimension of set $ A $ is strictly less than that of the ambient space. The overall idea is that major components of the argument in $ \mathbb R^d $ can be transferred to this new context.

The present section is an application of the standard measure theoretic argument, developed for the hypersingular Riesz energies \cite{hardinMinimal2005,borodachovAsymptotics2008,borodachovDiscrete2019} and inspired by the work of Federer~\cite{federerGeometric1996a}. We write $ \mathcal H_d $ for the $ d $-dimensional Hausdorff measure on $ \mathbb R^{d'} $. It is assumed to be normalized so that to coincide with $ \lambda_d $ on isometric embeddings of sets from $ \mathbb R^d $ into $ \mathbb R^{d'} $. We will need to consider maps from subsets of $ \mathbb R^{ d' } $ into $ \mathbb R $, that are almost isometries. 
By an almost isometry we understand a {\it bi-Lipschitz map} $ \psi: A \to \mathbb R^{d} $, which satisfies by definition
\[
    (1+\epsilon)^{-1} \| x-y\| \leq \| \psi(x)- \psi(y)\| \leq (1+\epsilon)\| x-y\| \qquad \text{ for any }x,y\in A \subset \mathbb R^{d'}.
\]
The quantity $ (1+\epsilon) $ is called the {\it Lipschitz constant} of the map $ \psi $. Note that the inverse map $ \psi^{-1} $ exists and is also bi-Lipschitz.
Finally recall the shorthand for the minimizing configurations $ \omega_N^*(A) $, which by definition satisfy
$$ \ee(\omega_N^*(A)) = \min \left\{ \ee(\omega_N,A) : \omega_N \subset A \right\} $$
for a compact $ A\subset \mathbb R^{d'} $.

As in Section~\ref{sec:short_range}, we fix a bounded open $ \Omega \subset \mathbb R^{d'} $ and assume that the weight and external field $ \h, \xi $ are defined on $ \Omega $ and continuous there. Let the lower semicontinuous functional $ \ee $ be defined on $ \Omega \subset \mathbb R^{d'} $; that is, 
\[
    \ee(\cdot,\ A): \omega_N \to [0,\infty], \qquad \omega_N\subset A\subset \Omega, \ N\geq N_0(\ee),
\]
is defined for all compact $ A\subset \Omega $ and is lower semicontinuous on $ { \Omega}^{N} $.
Similarly to Section~\ref{sec:short_range}, $ \ee(\omega_N, A) $ is additionally assumed to be set-monotonic in $ A \subset \mathbb R^{d'} $ for any fixed sequence $ \omega_N $, $ N\geq N_0(\ee) $, in the sense of \eqref{eq:monotonicity}.
Our goal is to reduce the asymptotics of $ \ee $ to those of a short-range functional defined on $ \mathbb R^d $, which will be denoted by $ \e $ as before. More precisely, we will need 
\begin{definition} 
    \label{def:stable_lip}
    A lower-semicontinuous functional $ \ee $ is said to be an {\it embedded translation-invariant d-dimensional short-range interaction}, if there exists an admissible  translation-invariant interaction functional $ \e $ acting on $ \mathbb R^d $, such that
    for a compact $ A\subset \mathbb R^{d'} $ and a configuration $ \omega_N\subset A $, for any bi-Lipschitz map $ \psi: A\to \subset \mathbb R^d $ with the constant $ 1+\epsilon $, one has
    \[
        (1-\gamma(\epsilon)) \cdot \e[\psi(\omega_N),\, \psi(A)] \leq  \ee(\omega_N,\, A) \leq (1+\gamma(\epsilon)) \cdot \e[\psi(\omega_N),\, \psi(A)]
    \]
    where the term $ \gamma(\epsilon) $ depends only on $ \epsilon $ and $ \gamma(\epsilon) \to 0 $ when $ \epsilon\downarrow0 $. The functional $ \e $ is referred to as {\it the flattened version of $ \ee $}.
\end{definition}
Note that an $ \ee $ satisfying the above definition is translation-invariant on sets that can be approximately embedded in $ \mathbb R^d $. As in Section~\ref{sec:short_range}, associated to the flattened interaction $ \e $ we have a triple 
$$ \left(\f(1)t^{-\sigma},\ \t,\ t^{1+\sigma}\right), $$ 
where $ \f(1) > 0 $,  $ \t $ is a rate \eqref{eq:rate}, and $ \sigma\in(-\infty,-1)\cup(0,\infty) $. 
Applying Definition~\ref{def:stable_lip} to $ \omega_N^*(A) $, we have
\begin{equation}
    \label{eq:stable_lip}
    (1-\gamma(\epsilon))\cdot \l_\e(\psi( A)) 
    \leq \liminf_{N\to\infty} \frac{\ee(\omega_N^*( A))}{\t(N)} 
    \leq \limsup_{N\to\infty} \frac{\ee(\omega_N^*( A))}{\t(N)}
    \leq (1+\gamma(\epsilon)) \cdot \l_\e(\psi( A)),
\end{equation}
where $ \psi: A \to \mathbb R^d $ is assumed to be bi-Lipschitz with constant $ 1+\epsilon $. Note that due to continuity of $ \psi $, $ \psi(A) $ is a compact subset of $ \mathbb R^d $, thus the asymptotics of $ \e $ on it are well-defined by Section~\ref{sec:short_range}; furthermore, if the embedding map can be constructed for any $ \epsilon > 0 $, the above asymptotics depend only on the value of $ \mathcal H_d(A) = $ (some smoothness assumptions on $ A $ are necessary as well, see below). It is essential here that the quantity $ \gamma(\epsilon) $ does not depend on a specific bi-Lipschitz embedding of the compact set $ A $ in $ \mathbb R^d $, but only on its constant.

As is clear from Definition~\ref{def:stable_lip}, we shall need the existence of maps arbitrarily close to an isometry, embedding compact sets $ A\subset \mathbb R^{d'} $ into $ \mathbb R^d $. Since not every compact $ A \subset \mathbb R^{d'} $ can be  embedded in a lower-dimensional $ \mathbb R^d $ with arbitrarily small distortions, we now proceed to describe the class of sets for which such maps do exist.
Recall that  a set $  A\subset \mathbb R^{d'} $ is said to be $ d ${\it-rectifiable} if there exist a bounded open $ A_0\subset \mathbb R^d $ and a Lipschitz map $ \psi: A_0 \to \mathbb R^{d'} $ such that $  A \subset \psi(A_0) $. It is a known result due to Federer that a compact $ d $-rectifiable $ A $ is up to a set of $ d $-Hausdorff measure 0 a union of bi-Lipschitz images of compact sets in $ \mathbb R^d $ with Lipschitz constants arbitrarily close to 1. More precisely, one has the following
\begin{lemma}[{\cite[Thm. 3.2.18]{federerGeometric1996a}}]
    \label{lem:federer}
    Suppose that $  A\subset \mathbb R^{d'} $ is $ d $-rectifiable. Then for any fixed $ \epsilon>0 $ there exist compact sets $ K_m $, $ m\geq 1 $ and bi-Lipschitz maps $ \psi_m:K_m\to \mathbb R^{d'} $, $ m\geq 1 $ such that $ \psi_m(K_m) \subset  A $ are disjoint and
    $$ \mathcal H_d\left( A\setminus \bigcup_{m=1}^\infty \psi_m(K_m)\right) =0.  $$
\end{lemma}
For our purposes it will also be enough to assume that $  A $ is such that $ \mathcal M_d(A) = \mathcal H_d(A) $ and $ A $ is $ (\mathcal H_d,d) $-rectifiable, that is, up to a set of $ \mathcal H_d $-measure 0 it is equal to a countable union of $ d $-rectifiable sets. It can be readily seen from the above lemma that changing $  A $ by a set of $ \mathcal H_d $-measure 0 does not change the bi-Lipschitz representation which is of interest to us.

Note further that both i) $ d $-rectifiability and ii) $ (\mathcal H_d,d) $-rectifiability together with $ \mathcal M_d(A) = \mathcal H_d(A) $ are preserved under taking compact subsets of $  A $. That is, any compact set $  D \subset  A $ has these smoothness properties as well \cite[Lem. 8.7.2]{borodachovDiscrete2019}.
It is useful to observe that for a $ d $-rectifiable $ A $, the $ d $-dimensional Minkowski content coincides with the $ d $-dimensional Hausdorff measure $ \mathcal H_d $ \cite[Thm. 3.2.39]{federerGeometric1996a}. When the smoothness assumptions on $  A $ must be relaxed to it being only $ (\mathcal H_d, d) $-rectifiable, one has to make the additional requirement $ \mathcal H_d( A) = \mathcal M_d( A) $.

In order to consider non-translation-invariant functionals on $ \mathbb R^{d'} $, we shall use a pair of continuous functions prescribing the weight and external field, respectively:
$$ \h:\Omega \to [\h_0,\infty], \qquad \xi:\Omega \to [0,\infty]. $$
We assume that there is a translation-invariant interaction $ \e $ acting on $ \mathbb R^d $, associated with $ \ee $.
Following Definition~\ref{def:ext_field}, we will say that $ \ee $ is {\it equipped with  weight $ \h $ and external field $ \xi $}, if
\begin{equation}
    \label{eq:associated}
    \begin{aligned}
        &(1-\gamma(\epsilon)) \cdot \l_\e(\psi( A)) \min_{y\in A} \h(y)  
        + \min_{y\in A}\xi(y)\\
        &\leq \liminf_{N\to\infty} \frac{\ee(\omega_N^*( A))}{\t(N)} 
    \leq \limsup_{N\to\infty} \frac{\ee(\omega_N^*( A))}{\t(N)}\\
        \leq &(1+\gamma(\epsilon)) \cdot \l_\e(\psi( A))  \max_{y\in A} \h(y) 
        + \max_{y\in A}\xi(y)
    \end{aligned}
\end{equation}
holds for all $ d $-rectifiable sets $ A $ [$ (\mathcal H_d,d)$-rectifiable A with $ \mathcal H_d(A) = \mathcal M_d(A) $].
Here both $ \h $ and $ \xi $ will be assumed continuous, but we expect the argument from \cite{hardinGenerating2017} dealing with lower semicontinuous $ \xi $ to be fully applicable.

In order that the asymptotics of a functional $ \ee $ on a $ d $-rectifiable set exist, it is necessary to make assumptions stronger than those in Definitions~\ref{def:cube_asymp}--\ref{def:cube_conts}. 
The analog of continuity with respect to the distance $ \lambda_d( A\triangle B) $ will need to be strengthened to more general sets and the distance $ \mathcal H_d( A\triangle B) $. We will give two versions of this property, that involve either $ d $-rectifiable or $ (\mathcal H_d, d) $-rectifiable set $ A $. 

\begin{definition}
    \label{def:strong_cube_conts}
    The functional $ \ee $ is said to have the {\it embedded continuity property}, if
    for every compact $ d $-rectifiable [$ (\mathcal H_d,d) $-rectifiable] $  A\subset  \Omega $ [with $ \mathcal H_d(A) = \mathcal M_d(A) $] and any $ \epsilon \in (0,1) $ there holds
    \begin{equation}
        \label{eq:strong_cube_conts}
        \liminf_{N\to\infty}   \sgn \sigma \cdot  \frac{\ee(\omega_N^*( A))}{\t(N)} \geq  
         \liminf_{N\to\infty}\, ( \sgn \sigma - \epsilon) \cdot  \frac{\ee(\omega_N^*( D))}{\t(N)} 
    \end{equation}
    whenever the compact $  D \subset  A $ satisfies $ \mathcal H_d( D) \geq (1 - \delta(A, \epsilon) )\,\mathcal H_d( A)$ 
    with $ \delta(A,\epsilon) $ depending only on $ A $ and $ \epsilon $.
\end{definition}
On the other hand, the short-range property from Definition~\ref{def:short_range} can be relaxed to sets positive distance apart as follows.
\begin{definition}
    \label{def:strong_short_range}
    A functional $ \ee $ is said to have the {\it embedded short-range property} if for any pair of compact $ d $-rectifiable [$ (\mathcal H_d,d) $-rectifiable] sets $ A_1, A_2 \subset  \Omega $ [with $ \mathcal H_d(A) = \mathcal M_d(A) $] positive distance apart, there holds
    \begin{equation}
        \lim_{N\to\infty} \frac{\ee(\omega_N\cap A_1, A_2)+\ee(\omega_N\cap A_2,A_2)}{\ee(\omega_N, (A_1\cup A_2))} = 1,
    \end{equation}
    where the sequence $ \omega_N $ has one of the two forms:

    \noindent{\bf (i)}
    $ \omega_N  = \omega_N^*(A)  $;

    \noindent{\bf (ii)}
    $ \omega_N = \bigcup_{m=1,2} \omega_{N_m}^*(A\cap Q_m) $, with $ N_1 + N_2 = N $, $ N_m\to \infty $.
\end{definition}
To summarize, a general short-range functional $ \ee $ for which we study the asymptotics of $ \ee(\omega_N^*(A)) $ when the compact $ A \subset \mathbb R^{d'} $ is $ d $-rectifiable [$ (\mathcal H_d,d) $-rectifiable with $ \mathcal H_d(A) = \mathcal M_d(A) $] must be as in the following
\begin{definition}
    \label{def:embedded}
    A lower-semicontinuous set monotonic functional $ \ee $ on $ \mathbb R^{d'} $ is called {\it an embedded admissible functional}, if it satisfies Definitions~\ref{def:strong_short_range} and~\ref{def:strong_cube_conts}, and there is an associated translation-invariant interaction functional $ \e $ acting on $ \mathbb R^d $, such that~\eqref{eq:associated} holds for every $ d $-rectifiable compact $ A\subset \Omega $ [$ (\mathcal H_d,d)$-rectifiable A with $ \mathcal H_d(A) = \mathcal M_d(A) $]. The triple $ (\f(1)t^{-\sigma}, \t, t^{1+\sigma}) $ of $ \e $ is called the {\it associated triple} of $ \ee $.
\end{definition}
Recall that set-monotonicity of $ \ee $ is the equation \eqref{eq:monotonicity} with $ \sigma $ in place of $ \varsigma $. For functionals satisfying Definition~\ref{def:embedded} we obtain an analog of Theorem~\ref{thm:ext_dist}.

\begin{theorem}
    \label{thm:embedded}
    Suppose $ A \subset \mathbb R^{d'} $ is a $ d $-rectifiable set, $ \mathcal H_d(A) > 0 $, {\rm [}$ (\mathcal H_d, d) $-rectifiable set with $ \mathcal M_d( A) = \mathcal H_d( A) ${\rm ]}, and $ \ee $ is an embedded admissible functional with associated triple $ (\f(1)t^{-\sigma}, \t, t^{1+\sigma}) $, continuous weight $ \h\geq h_0 > 0 $, and continuous external field $ \xi\geq0 $.
    Then 
    \[
        \lim_{N\to\infty} \frac{{\e_{\h,\xi}}(\omega_N^*( A))}{\t(N)} =
        \f(1) \int_{ A} \h(x)\,\pphi(x)^{1+\sigma}\, d\mathcal H_d(x) + \int_{ A} \xi(x) \, \pphi(x)\, d\mathcal H_d(x),
    \]
    where $ \pphi $ is the density of a probability measure $ \mu $ supported on $  A $, and is given by    
    \[
        \pphi(x) = \frac{d\mu}{d\mathcal H_d}(x) = \left(\frac{L_1-\xi(x)}{\f(1)(1+\sigma)\h(x)}\right)^{1/\sigma}_+
    \]
    for a normalizing constant $ L_1 $.
\end{theorem}
\begin{proof}
    This proof follows the standard approach using set monotonicity and short-range property to obtain two-sided estimates for $ \l_\ee(A) $; compare the proof of Poppy-seed bagel theorem \cite[\S 8.7]{borodachovDiscrete2019}. For simplicity, we set $ \xi = 0 $; the case of nontrivial external field can be handled similarly, see the discussion in Section~\ref{sec:ext}. The key element is the decomposition of set $  A $ from Lemma~\ref{lem:federer}. By this lemma, for any fixed $ \epsilon > 0 $ one can find a collection of compact subsets of $ \mathbb R^d $, $ \{ K_m \}_{m=1}^M $, such that 
    $$ \mathcal H_d\left( A\setminus \bigcup_{m=1}^M \psi_m(K_m)\right) <\delta(\epsilon)\mathcal H_d( A), $$
    with the quantity $ \delta(\epsilon) $ chosen as in Definition~\ref{def:strong_cube_conts}. Since $ \bigcup_m \psi(K_m) $ is a compact subset of $ d $-rectifiable set, it is itself $ d $-rectifiable, and $\mathcal H_d $ coincides with $ \mathcal M_d $ on it. Thus, property \eqref{eq:strong_cube_conts} applies to the compact set $ \bigcup_m \psi(K_m) \subset  A $.

    Without loss of generality, we can make the diameter of $ K_m $ sufficiently small, so that 
    $$ \max_{K_m} \h(x) - \min_{K_m} \h(x) =: \overline \h_m - \underline \h_m  \leq \epsilon. $$
    Since the compact sets $ \psi(K_m) $ are disjoint, they are positive distance apart, and short-range property applies. Arguing as in Lemma~\ref{lem:simple_weight},  we obtain using \eqref{eq:associated}
    \[
        \begin{aligned}
        \overline\l_\ee \left(\bigcup_{m=1}^M \psi_m(K_m)\right) 
        &\leq (1+\gamma(\epsilon)) 
        \min \left\{ \f(1)\sum_{m=1}^M \overline \h_m \left(\frac{\beta_m}{\lambda_d(K_m)}\right)^{1+\sigma}\lambda_d(K_m) : \sum_{m=1}^M \beta_m =1, \ \beta_m \geq 0 \right\}\\
        &\leq \frac{(1+\gamma(\epsilon))}{(1-\epsilon)^d} \cdot \frac{\f(1)}{\left(\sum_{m=1}^M \overline \h_m^{-1/\sigma} \mathcal H_d(\psi(K_m))\right)^\sigma },
        \end{aligned}
    \]
    as well as the converse inequality with $ \overline \l_\ee $ replaced by $ \underline\l_\ee $ and $ \overline \h_m $ with $ \underline \h_m $:
    \[
        \underline\l_\ee \left(\bigcup_{m=1}^M \psi_m(K_m)\right) 
        \geq \frac{(1-\gamma(\epsilon))}{(1+\epsilon)^d} \cdot \frac{\f(1)}{\left(\sum_{m=1}^M \underline \h_m^{-1/\sigma} \mathcal H_d(\psi(K_m))\right)^\sigma }.
    \]
    Note that for $ \epsilon\downarrow 0 $, due to continuity of $ \h $ on $  A $, the right-hand side in both estimates converges pointwise to 
    \begin{equation}
        \label{eq:goal}
        \frac{\f(1)}{\int_{ A} \h(x)\, d\mathcal H_d(x)}.
    \end{equation}
    This convergence is also bounded, since $ \h\geq \h_0 > 0 $.
    Having thus established two-sided estimates for the asymptotics on $ \bigcup_{m=1}^M \psi_m(K_m) $, we next have to relate them to the asymptotics on $ A  $. This relation follows from the monotonicity~\eqref{eq:monotonicity} and the embedded continuity property from Definition~\ref{def:strong_cube_conts}. For $ \sigma > 0 $ there holds by monotonicity
    \[
        \overline \l_\ee( A) \leq \overline\l_\ee \left(\bigcup_{m=1}^M \psi_m(K_m)\right) \leq \frac{(1+\gamma(\epsilon))}{(1-\epsilon)^d} \cdot \frac{\f(1)}{\left(\sum_{m=1}^M \overline \h_m^{-1/\sigma} \mathcal H_d(\psi(K_m))\right)^\sigma }.
    \]
    On the other hand, by the embedded continuity property~\eqref{eq:strong_cube_conts} and the choice of $ \{ K_m \} $,
    \[
        \underline \l_\ee( A) \geq (1-\epsilon) \underline\l_\ee \left(\bigcup_{m=1}^M \psi_m(K_m)\right)
        \geq (1-\epsilon) \frac{(1-\gamma(\epsilon))}{(1+\epsilon)^d}\cdot \frac{\f(1)}{\left(\sum_{m=1}^M \underline \h_m^{-1/\sigma} \mathcal H_d(\psi(K_m))\right)^\sigma }.
    \]
    For $ \sigma < -1 $ the two previous inequalities are reversed as follows:
    \[
        \begin{aligned}
        \overline \l_\ee( A)
        &\leq (1+\epsilon) \overline\l_\ee \left(\bigcup_{m=1}^M \psi_m(K_m)\right) \leq (1+\epsilon) \frac{(1+\gamma(\epsilon))}{(1-\epsilon)^d} \cdot \frac{\f(1)}{\left(\sum_{m=1}^M \overline \h_m^{-1/\sigma} \mathcal H_d(\psi(K_m))\right)^\sigma },
        \\
        \underline \l_\ee( A) 
        &\geq  \underline\l_\ee \left(\bigcup_{m=1}^M \psi_m(K_m)\right)
        \geq  \frac{(1-\gamma(\epsilon))}{(1+\epsilon)^d}\cdot \frac{\f(1)}{\left(\sum_{m=1}^M \underline \h_m^{-1/\sigma} \mathcal H_d(\psi(K_m))\right)^\sigma }.
        \end{aligned}
    \]
    For both signs of $ \sigma $, the right-hand side of the estimates converge to~\eqref{eq:goal}; by continuity of $ \h $ and Lebesgue dominated convergence theorem, our first claim follows.

    The second claim is obtained in the same way as the result of Theorem~\ref{thm:weighted_dist}, by using that compact subsets $ D $ of $  A $ are $ d $-rectifiable [$ (\mathcal H_d,d)$-rectifiable with $ \mathcal H_d( D) = \mathcal M_d( D) $]; see also the proof of the Poppy-seed bagel theorem \cite[Lem. 8.7.4]{borodachovDiscrete2019} in the case of uniform density.

    The addition of the external field is handled as in the full-dimensional case.
    As discussed in Section~\ref{sec:ext}, the assumption of continuity of the external field can be weakened to lower semicontinuity.
\end{proof}


\subsection{Concave fractional rate for \texorpdfstring{$ -1 <  \sigma < 0 $}{-1 < sigma < 0}}
\label{sec:generalizations} 
In Section~\ref{sec:general}, we mentioned briefly that the results concerning admissible functionals with strictly convex fractional rate $ \w $ can be generalized to maximizers of upper semicontinuous functionals with strictly concave fractional rate. We will now outline the necessary modifications Section~\ref{sec:general} to handle such functionals; in effect, one only needs to consider maximizers in place of minimizers.

We consider upper semicontinuous functionals $ \hat\e(\omega_N,A) $, and denote their maximizers by $ \omega^*_N(A) $, so that 
\[
    \E(\omega_N^*(A)) := \max \{ \E(\omega_N, A) : \omega_N \subset A \}.
\]
There is a triple associated to $ \E $,
$$ (\f, \t, \w):[0,\infty)\to[0,\infty]^3. $$
The rate $ \t $ is assumed to be measurable with the fractional rate given by a strictly concave $ \w $:
\begin{equation*}
    \lim_{u\to \infty}\frac{\t(t u)}{\t(u)} = \w(t),
\end{equation*}
and the set monotonicity is given by equation~\eqref{eq:monotonicity}, with $ \sgn \varsigma = -1 $. 
\begin{definition}
    An upper semicontinuous interaction functional with  rate $ \t $ is said to be {\it admissible}, if it is set monotonic with $ \sgn \varsigma = -1 $ and Definitions~\ref{def:cube_conts}, \ref{def:short_range}, and \ref{def:cube_conts} hold with $ \E $ in place of $ \e $ and $ \sgn \varsigma =-1 $, assuming that $ \omega_N^*(A) $ denotes the maximizers of this functional.
\end{definition}
All the results obtained in Sections~\ref{sec:short_range}--\ref{sec:translation_dependent} about asymptotics of the minimizers of $ \e $ with negative $ \sigma $ then apply to the maximizers of $ \E $. In particular, by the argument of Section~\ref{sec:regular_variation} about functions of regular variation, the above definition implies
\[
    \f(t) = \f(1) t^{-\sigma}, \qquad \w(t) = t^{1+\sigma}, \qquad -1 < \sigma < 0.
\]
An example of functional with such interaction is Riesz energy with negative exponent, truncated to $ k $ nearest neighbors:
\[
    E^k_s(\omega_N) = \sum_{i=1}^N \sum_{j\in I_{i,k} } \|{x}_i - {x}_j \|^{-s}, \qquad  -d < s < 0.
\]
Indeed, this expression is scale-invariant and short-range. It also satisfies Definition~\ref{def:local}: it is not hard to see that the maximizers have the optimal order of separation, and thus interactions between the different tiles are of a smaller order than $ E^k_s $, see Section~\ref{sec:nearest_neighbor}. Definition~\ref{def:cube_conts} can be verified as it was done for optimal quantizers, see Section~\ref{sec:CVT}; we omit further details here. 

In Section~\ref{sec:meshing}, we will explore in more detail another example of interaction with concave $ \w $: the Persson-Strang meshing algorithm, which has been successfully applied to generating point distributions, suitable for meshes \cite{Persson2004a}. It turns out, this functional also can be treated in the framework of Sections~\ref{sec:short_range}--\ref{sec:translation_dependent}.
It it an interesting question to study how suitable $ E^k_s $ with negative $ s $ is   for domain discretization, compared to the positive $ s $ case. 

\section{Verification of short-range interaction properties}
\label{sec:examples}

In this section we consider concrete examples of interaction functionals and verify that they are admissible; that is, lower semicontinuous, set-monotonic, and satisfying Definitions~\ref{def:cube_asymp}--\ref{def:cube_conts}. For the weighted and external field versions we check the Definitions~\ref{def:weighted}, and \ref{def:ext_field}, and the embedded continuity property as given for embedded sets in Definition~\ref{def:embedded}. The standard example of a translation-invariant short-range functional with $ \sigma > 0 $ will be the $ k $-truncated Riesz energy, discussed in Section~\ref{sec:nearest_neighbor}:
\begin{equation*}
    E^k_s(\omega_N) := \sum_{i=1}^N \sum_{j\in I_{i,k} } \|{x}_i - {x}_j \|^{-s}, \qquad  s > 0,
\end{equation*}
for which $ \sigma = s/d $. In Section~\ref{sec:1d} we show that on the one-dimensional flat torus $ \mathbb R/\mathbb Z $, minimizers of the $ k $-truncated Riesz energy are given by the equally spaced points, for any $ k\geq1 $.

The standard example of a short-range functional with $ \sigma < -1 $ will be the quantization error, discussed in Section~\ref{sec:CVT},
\begin{equation*}
    \cvt(\omega_N, A) = \int_A \min_i \|y-x_i\|^p \d\lambda_d(y),
\end{equation*}
for which $ \sigma = -1-p/d $.

A range of functionals can be dealt with using the methods of the present paper. In addition to the two types mentioned above, in Section~\ref{sec:meshing} we discuss some popular meshing algorithms and how they can be interpreted in our framework. 

\subsection{Asymptotics on the cube for \texorpdfstring{$ s $}{s}-scale invariant  functionals}
\label{sec:cube} 

Let us demonstrate how to establish the existence of asymptotics on cubes from Definition~\ref{def:cube_asymp} for some interaction functionals of interest; we assume that the considered functionals are lower semicontinuous, monotonic, and translation invariant. 

We say that a translation invariant interaction functional $ \e $ is {\it approximately s-scale invariant}, if there exists an $ s\in(-\infty,-d)\cup(0,\infty) $ such that 
\[
    \e(\omega_N^*(t\,A)) = \frac{\psi(t^d)}{t^{s}}\, \e(\omega_N^*(A)), \qquad N \geq N_0(\e),
\] 
where $ \psi(t) > 0  $ is continuous for $ t > 0 $ and satisfies 
\[
    \lim_{t\downarrow 0} \frac{\psi(c\,t)}{\psi(t)} = 1 \qquad \text{for any } c > 0.
\]
If $ \psi \equiv 1 $, the corresponding $ \e $ is said to be {\it s-scale invariant}. The monotonicity of the functional $ \e $  is assumed as in \eqref{eq:monotonicity}, with $ \sgn s = \sgn \varsigma $.
In this section we write as before $ \q_d = [-1/2,1/2]^d $.  
\begin{remark}
    As will become apparent from the proof of Theorem~\ref{thm:cube}, the equality in approximate $ s $-scale invariance can be replaced with an inequality:
    \[
        \e(\omega_N^*(t\,A)) = \frac{\psi(t^d)}{t^{s}}\, \e(\omega_N^*(A)), \qquad N \geq N_0(\e),
    \] 
    for a $ \psi(t) $ as above.
\end{remark}

Apart from the approximate scale invariance, we will need the additional assumption of local interactions dominating in $ \e $. We have previously assumed the short-range property \eqref{eq:short_range}; the property given in the following definition is closely related to the short-range property \eqref{eq:short_range}, yet does not follow from it.
\begin{definition}
    \label{def:local}
    We say that a translation invariant functional $ \e $ is {\it local}, if for any sequence of minimizing configurations $ \omega_n^* \subset \q_d $, $ n\geq N_0(\e) $, and $ \gamma \in (0,1] $, there holds 
    \begin{equation}
        \label{eq:locality}
        \limsup_{\gamma\uparrow 1}
        \limsup_{n\to \infty}
        \limsup_{L\to\infty} 
        \frac
        {\e(\omega, \q_d)}
        {L^d\,\e\left( \frac\gamma{L} \omega_n^* ,\frac\gamma{L}\q_d \right)}
        \leq 1,
    \end{equation}
    where 
    \[
        \omega = \bigcup_{
        \boldsymbol i\in\, L\mathbb Z^d }
        \left(\frac\gamma{L}\omega_n^* +  \frac{\boldsymbol i}L\right).
    \]
\end{definition} 
In the preceding definition, configuration $ \omega $ is obtained by tiling $ \q_d $ with copies of $ \omega_n $. Note that the denominator of \eqref{eq:locality} holds the interactions within each of the $ L^d $ cubes of side length $ \gamma/L $, which tile a portion of $ \q_d $, controlled by $ \gamma $. Setting $ \gamma =1 $, one obtains a precise  tessellation of $ \q_d $ by its copies $ \frac1L \q_d $.
\begin{remark}
    Any interaction functional $ \e $ that depends only on $ k $ nearest neighbor interactions for $ k $ fixed is local.
\end{remark} 
\begin{theorem}
    \label{thm:cube}
    Let $ \e $ be a local, translation invariant, and approximately $ s $-scale invariant interaction functional. Then the following limit exists:
    \[
        \lim_{N\to\infty} \frac{\e(\omega_N^*(\q_d))}{\psi(1/N)N^{1+s/d}}.
    \]
    In particular, $ \e $ satisfies equation \eqref{eq:cube_asymp} from Definition~\ref{def:cube_asymp} with
    \[
        \f(t) = t^{-s/d} \lim_{N\to\infty} \frac{\e(\omega_N^*(\q_d))}{\psi(1/N)N^{1+s/d}},\qquad t > 0.
    \]
\end{theorem}
Observe that this result determines the exponent $ \sigma $ in $ \f $ and $ \w $ for (approximately) $ s $-scale invariant functionals on $ \mathbb R^d $:
\[
    \sigma = s/d.
\]
\begin{proof}
    Suppose $ s > 0 $; the case of $ s < 0 $ can considered similarly, taking into account the different direction of monotonicity in \eqref{eq:monotonicity}.
    Let 
    \[
        \underline \l := \liminf_{N\to\infty} \frac{\e(\omega_N^*(\q_d))}{N^{1+s/d}},
    \] 
    and fix a sequence of minimizing configuration $ \underline \omega_n^* = \underline\omega_n^*(\q_d) $, $ n\in \uline{\mathscr N} $ for which $ \e(\underline \omega_n^*) < (\underline \l+\epsilon)n^{1+s/d} $, $ n\geq N_0(\e) $. In order to obtain an estimate on
    \[
        \overline \l := \limsup_{N\to\infty} \frac{\e(\omega_N^*(\q_d))}{N^{1+s/d}},
    \] 
    consider the sequence $ \oline{\mathscr N}\subset \mathbb N $ such that the minimizing configurations 
    $$ \{\overline \omega_N^* \}_{N\in \oline{\mathscr N}} = \{\omega_N^*(\q_d) \}_{N\in \oline{\mathscr N}}  $$
    achieve the limit of $ \overline \l $. For every $ N \in\oline{\mathscr N} $, find a pair of numbers of the form $ (L-1)^d\,n $ and $ L^d\,n $ such that 
    \begin{equation}
        \label{eq:squeeze}
        (L-1)^d\, n \leq N < L^d\,n.
    \end{equation}
    Following Definition~\ref{def:local}, we next estimate the value of $ \e(\overline\omega_N^*) $, $ N\in \oline{\mathscr N} $, from above by the value of $ \e $ on the configuration obtained by tiling $ \q_d $ with copies of $\underline \omega_n^*(\frac1L \q_d) $; more precisely, let
    \[
        \omega = \bigcup_{
        \boldsymbol i\in\, L\mathbb Z^d }
        \left(\omega_n^*\left(\frac\gamma{L}\q_d\right) +  \frac{\boldsymbol i}L\right), \qquad n \geq N_0(\e),
    \]
    where $ L\mathbb Z = \{ 0,1,\ldots, L-1 \} $, and  $\omega_n^*\left(\frac\gamma{L}\q_d\right)$ denotes, following the above notation, the configuration minimizing $ \e $ on the cube $ \frac\gamma L \q_d $. We assume $ \gamma \in (0,1) $, a constant to be specified later; we will eventually let $ \gamma\uparrow 1 $. 

    By the approximate scale invariance of $ \e $ there holds
    \[
        \e\left(\omega_n^*\left(\frac\gamma{L}\q_d\right)\right) = \psi\left( \frac{\gamma^d}{L^d} \right) \left(\frac L\gamma\right)^s \e(\underline\omega_n^*),
    \]
    and so from the minimality of the functional $ \e $ on $ \overline\omega_N^* $ and \eqref{eq:squeeze}, we have 
    \[
        \begin{aligned} 
             \frac{\e(\overline\omega_N^*)}{\psi(1/N)N^{1+s/d}}
            &\leq \frac{\e(\omega, \q)}{\psi(1/N)N^{1+s/d}} \\
            &
            \leq  \frac{\e(\omega,\q)}{L^d\,\e\left(\omega_n^*\left(\frac\gamma{L}\q_d\right)\right)}
            \cdot 
            \frac{L^d\,\psi\left(\frac {\gamma^d}{L^d}\right) \left(\frac L\gamma\right)^s \e(\underline\omega_n^*)}{\psi(1/N)(n(L-1)^d)^{1+s/d}} \\
            &\leq \frac{\e(\omega,\q)}{L^d\,\e\left(\omega_n^* \left(\frac\gamma{L}\q_d\right) \right)}
            \cdot 
            \frac{\psi\left(\frac {\gamma^d}{L^d}\right) }{\psi(1/N)}\cdot \gamma^{-s} \left(\frac{L}{L-1}\right)^{s+d}(\underline \l+\epsilon).
    \end{aligned}
    \]
    By the construction of $ \omega $, and locality of the functional $ \e $, the first ratio is eventually bounded by 1 when $ L \to \infty $ by Definition~\ref{def:local}. After taking $ \gamma $ close to 1, $ n $ sufficiently large, the first claim of the theorem follows.  


    The second claim is a consequence of 
    \[
        \w(t) = \lim_{N\to \infty} \frac{\psi(1/tN)(tN)^{1+s/d}}{\psi(1/N)(N)^{1+s/d}} = t^{1+s/d} = t^{1+\sigma},
    \]
    whence $ \f(t) = \f(1) t^{-s/d} $.
\end{proof}
The above theorem will allow us to obtain existence of asymptotics on cubes for the two classes of functionals of primary interest to us: Riesz $ k $ nearest neighbor functionals and optimal quantization error. It should be noted that the theorem does not guarantee that the asymptotics on cubes are finite; we shall verify this for both types of functionals separately.

\subsection{Riesz \texorpdfstring{$ k $}{k} nearest neighbor functionals}
\label{sec:nearest_neighbor}
This section deals with the modification of the classical Riesz energy that acts locally, by taking into account only the interactions with the $ k $ nearest neighbors of every point in $ \omega_N $. 
Namely, we define the value of the \textit{Riesz k nearest neighbor $ s $-energy} on a collection $ \omega_N = \{ x_1,\ldots,x_N \} \subset \Omega \subset \mathbb R^d $ as
\begin{equation*}
    E^k_s(\omega_N) := \sum_{i=1}^N \sum_{j\in I_{i,k} } \|{x}_i - {x}_j \|^{-s}, \qquad  s > 0,
\end{equation*}
As above, $ \omega_N \subset \Omega $ denotes $ N $-element subsets of $ \Omega $ and the indices of $ k $ nearest neighbors of $  x_i $ are denoted by $ I_{i,k} $. Thus $ I_{i,k} \subset \{ 1,\ldots,N \} $, $ \#  I_{i,k} = k $, and
\[
    \| x_i -  x_l\| \leq \| x_i -  x_j\| \quad \text{ for any } l\in I_{i,k} \text{ and } j\notin I_{i,k}.
\]
It will be also convenient to write $ x_{i,l} $ for the $ l $-th nearest neighbor of the point $ x_i $, $ 1\leq l \leq N-1 $. In particular, distances $ \|x_i - x_{i,l}\|  $ are increasing in $ l $.
In line with Definition~\ref{def:ext_field}, we will further augment the energy functional with a weight and an external field:
\begin{equation}\label{eq:k_energy}
    E^k_s(\omega_N; \kappa, \ext) := \sum_{i=1}^N \sum_{j\in I_{i,k} } \kappa( x_i,  x_j) \|{x}_i - {x}_j \|^{-s} +  N^{s/d} \sum_{i=1}^N \ext( x_i), \qquad  s > 0.
\end{equation}
Here the function $ \kappa \geq 0 $ is assumed to be lower semicontinuous on $ \Omega \times\Omega $; in addition, $ \kappa $ is bounded on $ \Omega\times \Omega $ and continuous and strictly positive on $ \diag(\Omega\times\Omega) $ as a subset of $ \Omega\times\Omega $. The external field $ \xi:\Omega \to [0,\infty] $ is assumed to be continuous as well (see the discussion after Theorem~\ref{thm:ext_dist} about lower semicontinuous $ \ext $). The positivity of $ \kappa $ and $ \ext $ can be relaxed to boundedness below; we assume positivity to avoid technicalities in the proofs.

Our eventual goal is to verify that Definitions~\ref{def:ext_field}, \ref{def:short_range}, \ref{def:cube_conts}. It is easy to see that the truncation of interactions to a fixed number of nearest neighbors guarantees that $ E^k $ is local, satisfying Definition~\ref{def:local}. Without such truncation, the locality does not hold when $ s < d $.

Modelling repulsive interactions that take into account only a certain number of nearest neighbors has numerous applications in physics and chemistry \cite{isobeHardsphere2015}. Truncation to a fixed number of nearest neighbors can be used to lower the computational complexity of node-distributing algorithms \cite{borodachovLow2014, vlasiukFast2018}; its advantage over radial truncation is that storing the coordinates of a configuration requires constant amount of memory.
 
\medskip

\noindent{\bf Existence of minimizers of $ E^k_s $ and their local properties. } In order to study the minimizing configurations, we first need to show that \eqref{eq:k_energy} attains its minimum on $ A^N $, that is, demonstrate lower semicontinuity of $ E^k $.
\begin{lemma}
    The functional $ E^k_s(\omega_N; \kappa,\ext) $ is lower semicontinuous on $ \Omega^N $ for $ N $ fixed.
\end{lemma}
\begin{proof}
    Fix a configuration $ \omega_N^0 = \{ x_1^0,\ldots,x_N^0 \} $; if $ x_i^0 = x_j^0 $ for some $ i\neq j $,  $ E^k_s(\omega_N^0; \kappa,\ext) = +\infty $. Due to the lower semicontinuity of $ \ker_s $, positivity of $ \kappa $ on the diagonal, and boundedness of $ \ext $ from below, there holds
    \[ 
        E^k_s(\omega_N; \kappa,\ext) \to +\infty, \qquad \hbox{ whenever }  \omega_N \to \omega_N^0 \hbox{ in }  \Omega^N.
    \]
    Let now $ \omega_N^0 $ consist of distinct points. Fix a sufficiently small $ \epsilon > 0 $ and suppose configuration $ \omega_N = \{ x_0,\ldots,x_N \} $ is within $ \epsilon $ in the $ l^\infty $ distance from $ \omega_N^0 $, so that 
    \begin{equation}
        \label{eq:linfty_convergence}
        \|{x}_i - {x}_i^0 \| < \epsilon, \qquad 1 \leq i \leq N.
    \end{equation}
    Note that if $ \epsilon < \Delta(\omega_N^0) /2 $, half of the minimal distance between the elements  of $ \omega_N^0 $ as given in \eqref{eq:separation}, then
    $ \|{x}_i - {x}_j \|^{-s} $ are continuous functions of $ \omega_N $ in the $ \epsilon $-neighborhood of $ \omega_N^0 $.  Fix an index $ i $ and consider distances from $ x_i $ to its nearest neighbors:
    \[
        d_l' := \| x_i - x_{i,l}\|, \qquad  1\leq l \leq N-1.
    \]
    By the definition of $ x_{i,l} $, $ d_{l+1}'\geq d_l' $. Let $ \{ d_m \} $ be the  increasing sequence of unique values among $ \{ d_l' \} $. Lastly, assume that $ \epsilon < \Delta(\omega_N^0) /2 $ and also
    \[
        \epsilon < \min \left\{ (d_{m+1}-d_m)/4 \right\}.
    \] 
    Define a partition of the points in $ \omega_N^0 $ and $ \omega_N $ indexed by $ d_m $, using the maps
    \[ 
        \begin{aligned}
          \mathfrak p^0\colon  x_j^0 & \mapsto \|x_i^0 - x_j^0\|,\\
          \mathfrak p  \colon  x_j   & \mapsto \|x_i^0 - x_j^0\|.
        \end{aligned}
    \]
    Observe that in view of \eqref{eq:linfty_convergence} and the choice of $ \epsilon $, for any $ x_m $, $ x_n\in\Omega_n $ inequality $ \mathfrak p(x_m) < \mathfrak p(x_n) $ implies 
    \[ 
        \|x_i - x_m \| < \|x_i - x_n \|.
    \]
    This means, nearest neighbors of $ x_i $ and $ x_i^0 $ of the same order in the two configurations must satisfy
    \[ 
        \mathfrak p^0 (x_{i,l}^0) = \mathfrak p (x_{i,l}),
    \]
    that is, the nearest neighbors of $ x_i $ in $ \omega_N $, ordered by increasing distance, have the same indices as those of $ x_i^0 $ in $ \omega_N^0 $, up to permutations inside the partitions introduced by the maps $ \mathfrak p^0 $ and $ \mathfrak p $.
    Let $ 1 \leq l \leq k $. The above construction can be summarized by observing that
    \[ 
        \|x_i^0 - x_{i,l}^0 \| - 2 \epsilon < \|x_i - x_{i,l} \| < \|x_i^0 - x_{i,l}^0 \| + 2 \epsilon,
    \]
    where $x_{i,l}^0$ and $x_{i,l}$ denote the $ l $-th nearest neighbor to $ x_i $ in the respective configuration. Since $ \|x_i^0 - x_{i,l}^0 \| > 0 $, we have shown that the Riesz term in \eqref{eq:k_energy} for $ \omega_N $, corresponding to the $ l $-th nearest neighbor of the $ i $-th point, is close to the same term in \eqref{eq:k_energy} for $ \omega_N^0 $. In combination with lower semicontinuity of $ \kappa $ and $ \ext $, this finishes the proof.
\end{proof}
Thus, the minimizers of $ E^k $ on any compact $ A\subset \Omega $ are well-defined.
By Proposition~\ref{prop:adding_ext_field}, adding the term 
\begin{equation}
    \label{eq:ext_field_summand}
    N^{s/d} \sum_{i=1}^N \ext( x_i)
\end{equation}
to a short-range interaction functional $ \e_\h $ with rate $ \t(t) = t^{1+s/d} $ gives a short-range functional with external field, $ \e_{\h,\xi} $. It therefore suffices to show that
\[
    E^k_s(\omega_N; \kappa) = E^k_s(\omega_N; \kappa, 0) = \sum_{i=1}^N \sum_{j\in I_{i,k} } \kappa( x_i,  x_j) \|{x}_i - {x}_j \|^{-s}
\]
is a short-range functional with the rate $ t^{1+s/d} $. In view of continuity of $ \kappa $ on $ \diag(\Omega\times\Omega) $, it is natural to expect that the weight $ \h $ of the above expression as a short-range functional $ \e_\h $ is $ \kappa(x,x)$,  $x \in \Omega $.

We will further establish some useful local properties of the $ k $-truncated energy: separation and covering.
Let
\begin{equation}
    \label{eq:separation}
    \Delta(\omega_N) := \min_{1\leq i \leq N} \|x_i - x_{i,1}\|
\end{equation}
denote the minimal distance to the nearest neighbor in a configuration $ \omega_N \subset \Omega $. We refer to the quantity $ \Delta(\omega_N) $ as the {\it separation} of $ \omega_N $. A simple volume argument implies $ \Delta(\omega_N) \leq C(A) N^{-1/d} $ for any $ \omega_N\subset A \subset \Omega $. To show that $ \Delta(\omega_N^*(A)) $ for minimizers has the optimal order $ N^{-1/d} $, we need the following geometric fact.
\begin{lemma}
    \label{lem:few_nns}
    Fix a configuration $ \omega_N\subset \mathbb R^d $ and $ r > 0 $. For any $ x_i\in \omega_N $, the number of points $ x_j $ in $ \omega_N\setminus B(x_i,r) $ such that $ x_i $ is one of $ k $ nearest neighbors of $ x_j $ is bounded by $ n(k,d) $, depending only on the number of neighbors $ k $ and the dimension $ d $. That is, 
    \[
        \#\{ j : i\in I_{j,k}, \ x_j\notin B(x_i,r) \} \leq n(k,d), \qquad 1\leq i \leq N.
    \]
\end{lemma}
    \begin{figure}[t]
        \centering
        \includegraphics[width=0.6\linewidth]{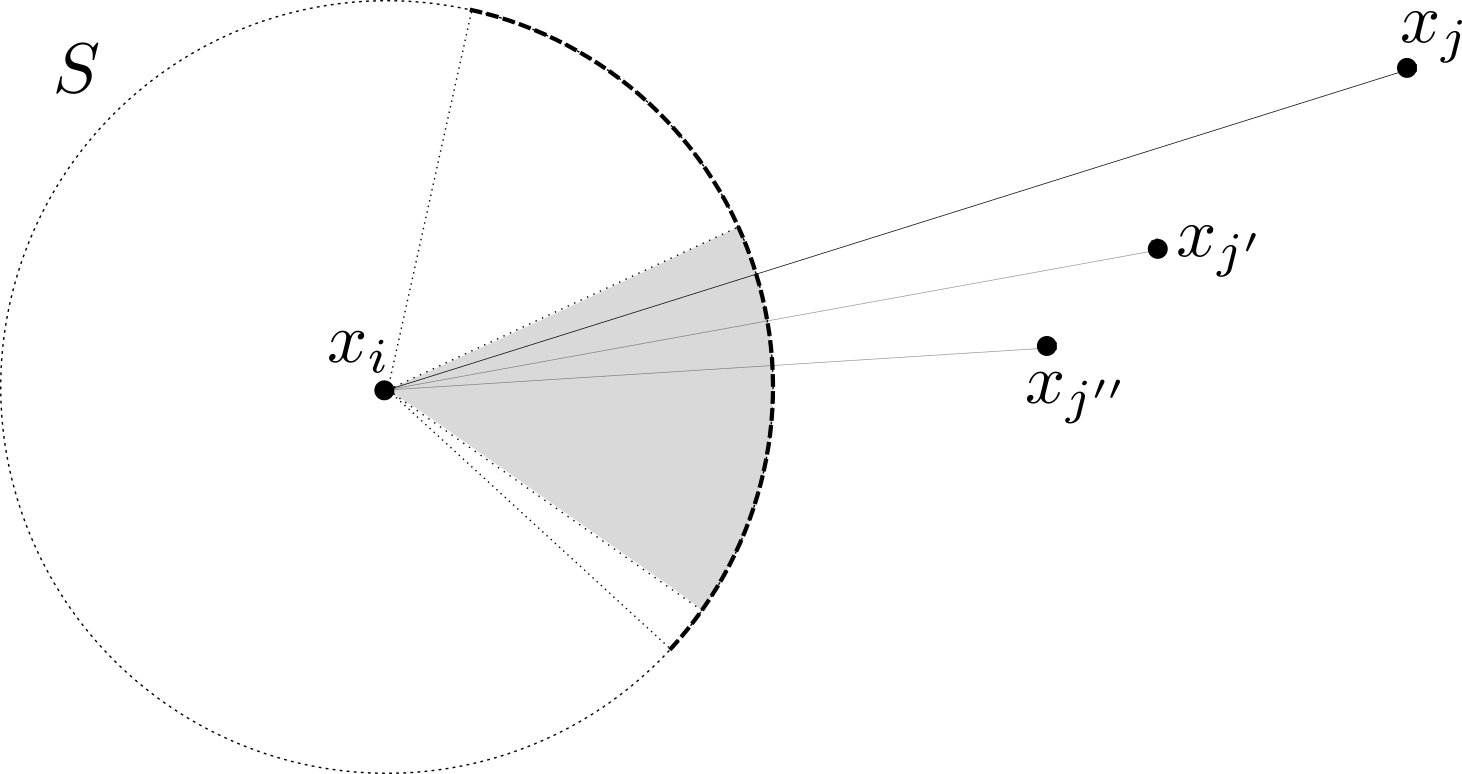}
        \caption{If the spherical cap of radius $ \pi/3 $ around $ x_j $ (dashed) contains the projection of $ x_{j'} $, it follows $ \|x_j-x_{j'}\| <\max\{ \|x_i-x_j\|, \|x_i-x_{j'}\|\} $. Hence, at most $ k $ points among $ \{ x_j \in B(x_i,r)^c: i\in I_{j,k} \} $ can be projected into any cap of angular radius $ \pi/6 $ (shaded).  }
        \label{fig:nns}
    \end{figure}
\begin{proof}
    Fix an $ x_i \in \omega_N $ and let $ \omega_{N,i} $ denote 
    \[
        \{ x_j : i\in I_{j,k} \}, \qquad 1\leq i \leq N.
    \]
    Let further $ \pi_S $ be the radial projection onto $ S := \partial B(x_i,r) $ and consider the image of points in $ \omega_{N,i}\setminus B(x_i,r) $ under this projection:  
    $$ \pi_S(\omega_{N,i}\setminus B(x_i,r)).$$
    Suppose that a geodesic ball on $ S $ of radius $ \pi/6 $, denoted $ B_S(y,\pi/6) $, $ y\in S $, contains more than $ k $ elements of this image. Let $ x_j $ be the furthest of the points in $ \omega_{N,i}\setminus B(x_i,r) $, such that $ \pi_S(x_j) \in B_S(y,\pi/6) $. Then $ B_S(y,\pi/6) \subset B_S(\pi_S(x_j),\pi/3) $, implying that $ B_S(\pi_S(x_j),\pi/3) $ contains $ k $ projections distinct from $ \pi_S(x_j) $. 

    On the other hand, for every point $ x_{j'} \in \omega_{N,i}\setminus B(x_i,r) $, from $ \pi_S(x_{j'}) \in B_S(\pi_S(x_j),\pi/3) $ it follows $ \angle x_jx_ix_{j'} < \pi/3 $, whence 
    $$ \|x_{j'}-x_j\| <\max\{ \|x_i-x_j\|, \|x_i-x_{j'}\|\} = \|x_i-x_j\|, $$
    since $ x_j $ was chosen the furthest from $ x_i $. Thus, every other point projected into $ B_S(x,\pi/6)$ is closer to $ x_j $ than $ x_i $, and it must be $ i\notin I_{j,k} $, a contradiction.
    By this argument, the constant $ n(k,d) $ chosen as
    \[
        n(k,d) : = \max \left\{ n : \exists \omega_n\subset \mathbb S^{d-1} \text{ such that } \#(B_S(y,\pi/6) \cap \omega_n)\leq k\ \forall y\in \mathbb S^{d-1}\right\}
    \]
    has the properties stated in the claim of the lemma.
\end{proof}
\noindent We will also need the following classical result due to Frostman.
\begin{proposition}[Frostman's lemma {\cite[p. 112]{mattila1995geometry}}, \cite{frostman1935potentiel}]
    \label{prop:frostman}
    For any compact $ A $ with $ \mathcal H_d(A) > 0 $ there is a finite nontrivial Borel measure $ \mu  $ on $ \mathbb R^{d'} $ with support inside $ A $ such that
\[ 
    \mu(B(x,r)) \leq r^d,\ {x}\in \mathbb R^{d'}.
\]
\end{proposition}

The following lemma will be necessary to establish asymptotic optimality of the truncated energies; it shows that for truncation to a nearest neighbor of any order, minimizers are spread over the set $ \Omega $ with the best possible order of separation.
We will discuss the weighted energy without external field, but the proof can be readily extended to the case of a continuous external field $ \xi $ produced by the term~\eqref{eq:ext_field_summand}. 
\begin{lemma}
    \label{lem:separation}
    Suppose $ s > 0 $, $ A \subset \Omega $ is compact, and a bounded lower semicontinuous $ \kappa(x,y):A\times A \to [0,\infty] $ is  continuous and positive at the points of $ \diag(A\times A) \subset A\times A $. Then the minimizers of $ E^k_s(\cdot\,; \kappa) $ have the optimal order of $ k $-separation:
    \[
        \Delta(\omega_N^*(A)) \geq C(A) N^{-1/d}.
    \]
\end{lemma}
\begin{proof}
    Suppose that, after a suitable renumbering of $ \omega_N^*(A) = \{ x_1^N,\ldots,x_N^N \} $, 
    \[
        \Delta(\omega_N^*(A)) = \| x_1^N - x_{1,1}^N\| = c_N N^{-1/d}.
    \]
    We need to show that $ c_N \geq C(A) > 0 $.
    By the discussion after  Proposition~\ref{prop:frostman}, there exists a point $ y^N \in A $, such that
    \[
        \| y^N - x_i^N\| \geq c(A) N^{-1/d}, \qquad i = 1,\ldots,N,
    \]
    where $ c(A) =  (\mu(A)/2)^{1/d}  $
    Since $ \omega_N^*(A) $ is a minimizer, replacing  $ x_1^N $ with $ y^N $ results in a larger value of $ E^k_s(\cdot\,; \kappa) $ than the one on $ \omega_N^*(A) $. Let 
    $ \omega_N' =  \{ y^N \} \cup \omega_N^*(A)\setminus \{ x_1^N \}  $, the configuration obtained by such replacement.
    Denote by $ y_l $, $ 1\leq l \leq k $ the $ k $ nearest neighbors to $ y $ in $ \omega_N^*(A) $, $ N\geq N_0(\e) $. 

    By compactness of $ A $ and continuity of $ \kappa $ on the diagonal, there exists $ \delta > 0 $ such that
    \[
        0 < m_\delta : =  \min \{ \kappa(x,y) : \|x-y\|\leq \delta \}. 
    \]
    Let further 
    \begin{equation}
        \label{eq:offdiagonal}
        M := M(\kappa,N) = \sup\{ \kappa(x,y) : \|x-y\|\geq c(A) N^{-1/d} \},
    \end{equation}
    where $ M $ is independent of $ N $ in view of boundedness of $ \kappa $. If
    \[
        \liminf_{N\to \infty} c_N N ^{-1/\d} \geq \delta,
    \]
    the proof is complete. Otherwise, suppose $ \| x_1^N - x_{1,1}^N\| = c_N N ^{-1/\d} < \delta $ for some $ N $. For such an $ N $, we have
    \begin{equation*}
        \begin{aligned}
            0  &\leq E^k_s(\omega_N'; \kappa) - E^k_s(\omega_N^*(A); \kappa) \leq\\
            &\leq \frac{-\kappa(x_1^N , x_{1,1}^N)}{\| x_1^N - x_{1,1}^N\|^{s}} 
            +
            \sum_{l=1}^k \frac{\kappa(y^N , y_{l}^N) }{\| y^N - y_{l}^N\|^{s}} +
            \sum_{i:y^N\in I_{i,k}} \frac{\kappa(x_i,y^N)}{ \| y^N - x_i\|^{s}}\\
            &\leq  \left(-m_\delta c_N^{-s} + k M c(A)^{-s} + n(d,k) M c(A)^{-s}\right) N^{s/d},
        \end{aligned}
    \end{equation*}
    where it is used that all the points $ x_i $ in the second summation are distance at least $ c(A) N^{-1/d} $  from $ y^N $. This implies 
    \begin{equation}
        \label{eq:replace_point}
        c_N\geq \left(\frac{m_\delta}{(k  + n(d,k)) M }\right)^{1/s} c(A),
    \end{equation}
    as desired.
\end{proof}
\begin{corollary}
    Minimizing the $ k $-truncated energy for any $ k \geq 1 $ yields a separated configuration.
\end{corollary}
It can be seen from \eqref{eq:replace_point} that the weight $ \kappa $ needs to be strictly positive on the diagonal to guarantee separation.


Observe that the proof of Lemma~\ref{lem:separation}, in particular inequality~\eqref{eq:replace_point}, gives for any $ y\in A $, 
\[
    B(y, C(A)\Delta(\omega_N^*) ) \cap \omega_N^*(A) \neq \emptyset.
\]
That is, the covering radius of the collection $ \omega_N^*(A) \subset A $ is at most $ C(A)\Delta(\omega_N^*) $. Using this inequality, we will establish that minimizers of $ E^k $ have covering radius of the optimal order.
\begin{lemma}
    \label{lem:covering}
    For a compact set $ A $ and energy $ E^k_s(\cdot\,; \kappa) $ as in the previous lemma, there holds
    \[
        \min_i \left\{\|y - x_i^N\| : x_i^N\in \omega_N^*(A)  \right\} \leq C(A)N^{-1/d}
    \]
    for any $ y \in A $.
\end{lemma}
\begin{proof}
    In view of the upper bound of $ C(A)\Delta(\omega_N^*) $ on the covering radius, it suffices to verify $ \Delta(\omega_N) \leq C(A) N^{-1/d} $ for any sequence of discrete sets $ \omega_N,$  $ N\geq k+1 $.
    Since the set $ A $ is compact, it is contained in some cube $ [-a,a]^d $. Suppose by contradiction, for a sequence $ c_N\to +\infty $ with $ c_N N^{-1/d}\to 0 $, and a collection of discrete sets $ \omega_N\subset A $, $ N\geq 2 $, there holds
    \[
        \Delta(\omega_N) > c_NN^{-1/d}.
    \]
    Then 
    \[
        B\left(x_i^N, \frac{c_N}2 N^{-1/d} \right) \cap B\left(x_j^N, \frac{c_N}2 N^{-1/d} \right) = \emptyset, \qquad i\neq j,
    \]
    but since $ \lambda_d \left( B\left(x_i^N, \frac{c_N}2 N^{-1/d} \right) \right) = c(d)c_N^{d} /N $ and all such balls are contained in 
    $$ \left[-a-c_N N^{-1/d},\ a+c_N N^{-1/d}\right]^d ,$$
    we must have
    \[
        2^d\left(a+c_N N^{-1/d}\right)^d \geq c(d)c_N^{d},
    \]
    which is a contradiction to $ c_N\to +\infty $ and $ c_N N^{-1/d}\to 0 $. We have therefore established the lemma.
\end{proof}
To see that this is indeed the optimal order for covering, let the measure $ \mu $ be as in Frostman's lemma, then for $ r_0=  c(A) N^{-1/d} := (\mu(A)/2)^{1/d} N^{-1/d}  $ and the set
$$ D =A\setminus \bigcup_{i} B({x}_i, r_0) ,$$
there holds $ \mu (D) \geq \mu(A)/2 $, so in particular $ D \neq \emptyset $. It follows that the covering radius of $ A $ is at least $ c(A) N^{-1/d} $.

\medskip
\noindent{\bf Verification of \eqref{eq:weighted}.}
Having established the separation and covering of the minimizers of $ E^k $, we return to verification of Definition~\ref{def:weighted}. We aim to show that $ E^k $ is an admissible functional with asymptotics on cubes $ \f(t) = \f(1) t^{s/d} $, rate $ \t(t) = t^{1+s/d} $, and weight $ \h(x) = \kappa(x,x) $. Observe that the monotonicity property~\eqref{eq:monotonicity} is an immediate consequence of $ E_s^k $ being dependent only on the collection $ \omega_N $, and not the set $ A $. Thus, minimizers satisfy $ \e(\omega_N^*(A)) \geq \e(\omega_N^*(B)) $ if $ A\subset B $.

We shall start with the unweighted energy $ E^k_s(\cdot\,;) $. Note that it is $ s $-scale invariant and local in the sense of Definition~\ref{def:local}. To verify Definition~\ref{def:cube_asymp}, it remains to show that the asymptotics on cubes are finite and strictly positive. It is easy to see that for any $ \omega_N $ chosen on a rectangular lattice,
\[
    E^k_s(\omega_N) \leq CN^{1+s/d},
\]
since distance between nearest neighbors for such $ \omega_N $ is on the order of $ N^{-1/d} $. To verify positivity, let $ \omega_N^*, N\geq k+1 $ denote a sequence of minimizers of $ E^k $ on $ \q_d $. There holds
\[
    E^k_s(\omega_N^*) \geq \sum_{i=1}^N \|x_i - x_{i,1}\|^{-s}  = \sum_{i=1}^N (\|x_i - x_{i,1}\|^d)^{-s/d} \geq N^{1+s/d} \left(\sum_{i=1}^N \|x_i - x_{i,1}\|^d\right)^{-s/d},
\]
where the last inequality follows from Jensen's inequality applied to the function $ t^{-s/d} $. It remains to observe that balls of the form $ B(x_i, \|x_i - x_{i,1}\|/2) $ are pairwise disjoint, so $ \lambda_d(B(0,1))\cdot2^{-d}\sum_{i=1}^N \|x_i - x_{i,1}\|^d \leq 3^d = \lambda_d(3\q_d) $. To summarize, 
\[
    E^k_s(\omega_N^*) \geq cN^{1+s/d}.
\]

Thus, Theorem~\ref{thm:cube} applies, and the asymptotics on cubes for $ E^k $ exist and are finite; Definition~\ref{def:cube_asymp} holds for the unweighted case. This implies also that equation \eqref{eq:weighted} from  Definition~\ref{def:weighted} is satisfied for any constant weight $ \kappa(x,y) $. For a non-constant weight $ \kappa $, due to the continuity of $ \kappa $ on $ \diag \Omega\times \diag \Omega $, for any $ \epsilon > 0 $ there exists a $ \delta > 0 $ for which  $ 1-\epsilon \leq \kappa(x,y)/\kappa(x,x) \leq 1+\epsilon $ whenever $ \|x-y\|< \delta $.
We can then choose $ N $ sufficiently large, so that $ \|x_i - x_{i,k}\| \leq \delta $ for $ 1\leq i \leq N $ and any minimizer $ \omega_N^* $, by the covering lemma. Since $ \epsilon $ is arbitrary, it follows
\begin{equation*}
    \min_{y\in(x+a\,\q_d)} \kappa(y,y) \f(a^d) 
    \leq \underline\l_{\e_\h}(x+a\,\q_d) 
    \leq \overline\l_{\e_\h}(x+a\,\q_d) 
    \leq  \max_{y\in(x+a\,\q_d)} \kappa(y,y) \f(a^d),
\end{equation*}
which is equation \eqref{eq:weighted}.

\medskip
\noindent{\bf Verification of \eqref{eq:short_range}.}
The short-range property \eqref{eq:short_range} for truncated Riesz energies holds trivially in the case of cubes $ Q_m $ positive distance apart. Indeed, the terms involving interactions between different cubes combined are at most  $ Nk r = o (N^{1+s/d}) $, where $ r $ is the separation distance between cubes. To verify this property for the sequence $ \omega_N^*(A) $ and cubes with disjoint interiors (but not necessarily disjoint), observe by Lemma~\ref{lem:separation}, for any $ \epsilon >0 $, for a sufficiently small $ \delta $, the $ \delta $-neighborhood of $ \bigcup_m \partial(Q_m) $ contains at most the $ \epsilon $ fraction of all points in $ \omega_N^*(A) $. The interactions between such points are therefore bounded by $ \epsilon N^{1+s/d} $; interactions between different cubes $ Q_m $ for elements of $ \omega_N^*(A) $ outside such $ \delta $-neighborhood are $ o (N^{1+s/d}) $ as above. Because $ \epsilon $ is arbitrary, \eqref{eq:short_range} follows.

\medskip
\noindent{\bf Verification of \eqref{eq:cube_conts}.}
It remains to focus on Definition~\ref{def:cube_conts}, the regularity property for $ E^k $. Recall the statement of this property: for any cube $ A =  x+a\,\q_d $ there must hold
    \begin{equation*}
        \lim_{N\to\infty}   \frac{\e(\omega_N^*(Q))}{\t(N)} \geq  
         \limsup_{N\to\infty} ( 1 - \epsilon) \cdot  \frac{\e(\omega_N^*(D))}{\t(N)} 
    \end{equation*}
    whenever the compact $ D \subset Q $ satisfies $ \lambda_d(D) \geq (1 - \delta( \epsilon) )\,\lambda_d(Q)$.
\medskip

The next lemma is the primary component in \eqref{eq:cube_conts} for the truncated Riesz. It uses the $ d $-dimensional Minkowski content $ \m_d $ in $ \mathbb R^{d'} $ in place of the Lebesgue measure $ \lambda_d $ on $ \mathbb R^d $. In the case $ d = d' $ these two coincide. We will later use also the statement in terms of Minkowski content to establish the asymptotics on embedded sets, see Section~\ref{sec:embedded}.

\begin{lemma}
    \label{lem:upper_reg}
    For a fixed compact set $ A \subset \mathbb R^{d'} $ with $ \mathcal H_d(A) > 0 $ and any $ \epsilon > 0 $, there exists a $ \delta(A,\epsilon) $ such that 
    \[
        \liminf_{N\to \infty}\frac{E^k_s(\omega_N^*(A); \kappa)}{N^{1+s/d}} > (1-\epsilon) \liminf_{N\to \infty} \frac {E^k_s(\omega_{N}^*(D); \kappa)}{N^{1+s/d}}
    \] 
    holds whenever a compact set $  D \supset A $ satisfies $ \m_d (D) > \m_d (A) - \delta $. 
\end{lemma}
\begin{proof}
    Let $ D \subset A $ be a compact set. According to Lemma~\ref{lem:separation}, the separation of a minimizer configuration on $ A $ satisfies $ \Delta(\omega_N^*(A)) \geq C(A) N^{-1/d} $. By rescaling $ A $ if necessary, it can be assumed that $ C = 1 $ so that
    \[
        \Delta(\omega_N^*(A)) \geq N^{-1/d}.
    \]

    The proof will consist in demonstrating a way to retract configurations from $ A $ to $ D $ without increasing the value of $ E^k $ on them too much. Let us write $ S(h) $ for the $ h $-neighborhood of a set $ S $.
    Fix an $ N \geq k+1 $ and consider $ \omega_N^*(A) $. We will first show that most of $ x_i\in\omega_N^*(A) $ have a point from $ D $ close to them. Observe that $ D(h) \subset A(h) $ for any $ h > 0 $. In view of $ \m_d (A) < \m_d (D) + \delta $, we can find small enough $ c>0 $ such that for all $ N $ there holds
    \[
        \lambda_p\left[A(c N^{-1/d})\right] - \lambda_p\left[D(c N^{-1/d})\right] < v_{{d'}-d} c^{{d'}-d} N^{-({d'}-d)/d} \delta,
    \]
    where $ \lambda_p $ denotes the $ {d'} $-dimensional Lebesgue measure in $ \mathbb R^{d'} $ and $ v_{d'-d} $ is the volume of the unit ball in $ \mathbb R^{{d'}-d} $. Thus the number of disjoint balls of radius $ \gamma N^{-1/d} $, for a $ \gamma>0 $, that can be contained in $ A(c N^{-1/d})\setminus D(c N^{-1/d}) $ is at most 
    \[
        \frac{v_{{d'}-d}\, c^{{d'}-d} \delta N^{-({d'}-d)/d} }{v_p \gamma^{d'} N^{-{d'}/d}} = \delta K({d'},d)\gamma^{-{d'}} c^{{d'}-d} N.
    \]
    It follows that for at least $ N ( 1- \delta K({d'},d)\gamma^{-{d'}} c^{{d'}-d} )  $ points in $ \omega_N^*(A) $, the closest point in $ D $ is at most distance
    \[
        (\gamma + c)N^{-1/d}
    \]
    away. Consider the subset $ \{ x_i' \} \subset \omega_N^*(A) $ for which this is the case, and for each $ x_i' $ find the closest point in $ D $. Denote the resulting set by $ \omega $. By the preceding discussion, 
    \[
        N_\omega := \#\omega\geq N \left( 1- \delta K({d'},d)\gamma^{-{d'}} c^{{d'}-d} \right). 
    \]
    Consider a pair $ x_i' $, $ x_j' $; let the nearest points in $ \omega $ be $ y_i $ and $y_j$ respectively. Since the separation between elements of $ \omega_N^* $ is at least $ N^{-1/d} $, there holds
    \begin{equation}
        \label{eq:nottooclose}
        \|y_i - y_j\| \geq (1-2(\gamma+c)) \| x_i' - x_j' \|.
    \end{equation}
    In estimating $ E^k_s(\omega_N^*(A);\kappa) $ in the next step, we will allow $ \kappa $ to be infinite on the diagonal. We would like to control the ratio of terms of the form
    \[
        \kappa(y_i , y_j)\|y_i - y_j\|^{-s} \big/ \kappa(x_i' , x_j') \| x_i' - x_j' \|^{s},
    \]
    and it will be enough to assume that the weight $ \kappa $ does not grow too fast towards the diagonal: at most as a polynomial of degree $ r $, so that when $ \|z_1-z_2\| \geq t \|z_1'-z_2'\| $, then
    \[
        \kappa(z_1 , z_2)\big/ \kappa(z_1' , z_2') \leq Ct^{-r}.
    \]
    This is of course the case with $ r=0 $ when $ \kappa $ is bounded and strictly positive on the diagonal $ \diag(D\times D) $.

    By the scaling properties of the $ \ker_s $ kernel and the above discussion, \eqref{eq:nottooclose} implies
    \[
        E^k_s(\omega_N^*(A);\kappa) \geq (1-2(\gamma+c))^{-s-r} E^k_s(\omega;\kappa).
    \]
    By the above estimate on the cardinality $ N_\omega $, we have finally
    \begin{equation}
        \label{eq:projected_omega_n}
        \frac{E^k_s(\omega_N^*(A);\kappa)}{N^{1+s/d}} \geq \frac{\left( 1- \delta K({d'},d)\gamma^{-{d'}} c^{{d'}-d} \right)^{1+s/d}}{(1-2(\gamma+c))^{s+r}}\cdot \frac{E^k_s(\omega_{N_\omega}^*(D);\kappa)}{N_\omega^{1+s/d}}.
    \end{equation}
    Taking $ \gamma = \delta^{1/{d'}} $ and letting $ \delta \to 0 $, $ c\to 0 $, we complete the proof of the first claim. 
\end{proof}
The proof of the above lemma shows that $ \kappa $ can be continuous on the diagonal in the extended sense, by growing to $ +\infty $ with subpolynomial rate.
\begin{lemma}
    For a compact set $ A \subset \mathbb R^{d'} $ for which $ \mathcal H_d(A) > 0 $ and $ \l_{E^k}(A) $ exists, and any $ \epsilon > 0 $, there is a $ \delta(A,\epsilon) $ such that 
    \[
        \l_{E^k} (A) > (1-\epsilon) \limsup_{N\to \infty} \frac {E^k_s(\omega_{N}^*(D); \kappa)}{N^{1+s/d}}
    \] 
    holds whenever a compact set $  D \subset A $ satisfies $ \m_d (D) > \m_d (A) - \delta $. 
\end{lemma}
\begin{proof}
    Equation \eqref{eq:projected_omega_n} with $ \gamma = \delta^{1/{d'}} $ implies
    \[
        \frac{E^k_s(\omega_N^*(A);\kappa)}{N^{1+s/d}} \geq \frac{\left( 1- K({d'},d) c^{{d'}-d} \right)^{1+s/d}}{(1-2(\delta^{1/{d'}}+c))^{s}}\cdot \frac{E^k_s(\omega_{N_\omega}^*(D);\kappa)}{N_\omega^{1+s/d}}
    \]
    holds for any sufficiently small $ c > 0 $ and $ \delta > 0 $ with $ N_\omega \geq ( 1- K({d'},d) c^{{d'}-d} )N $. For a fixed $ \epsilon > 0 $ pick $ \delta $ and $ c $ so that
    \[
        \frac{\left( 1- K({d'},d) c^{{d'}-d} \right)^{1+s/d}}{(1-2(\delta^{1/{d'}}+c))^{s}} > 1-\epsilon,
    \]
    and let a compact $ D $ satisfy $ \m(D) > \m(A) - \delta $. Let further $ \mathscr N $ be such that 
    \[
        \lim_{\mathscr N \ni n\to \infty} \frac {E^k(\omega_{n}^*(D))}{n^{1+s/d}} = \limsup_{N\to \infty} \frac {E^k(\omega_{N}^*(D))}{N^{1+s/d}},
    \]
    and for each $ n $ take $ N = \lceil n / ( 1- K({d'},d) c^{{d'}-d} ) \rceil $. Then for $ N_\omega(n,c) $ corresponding to such $ N(n,c) $ there holds $ N_\omega \geq n $ with $ N_\omega - n \leq 1 $. By \eqref{eq:projected_omega_n} we then have
    \[
        \begin{aligned}
        \frac{E^k_s(\omega_N^*(A);\kappa)}{N^{1+s/d}}
        &\geq \frac{\left( 1- K({d'},d) c^{{d'}-d} \right)^{1+s/d}}{(1-2(\delta^{1/{d'}}+c))^{s}}\cdot \frac{E^k_s(\omega_{N_\omega}^*(D);\kappa)}{N_\omega^{1+s/d}}\\
        &\geq \frac{\left( 1- K({d'},d) c^{{d'}-d} \right)^{1+s/d}}{(1-2(\delta^{1/{d'}}+c))^{s}}\cdot \left(\frac{n}{N_\omega}\right)^{1+s/d}  \frac {E^k(\omega_{n}^*(D))}{n^{1+s/d}}.
    \end{aligned}
    \]
    Since $ n/N_\omega \to 1 $, $ n\to \infty $, the desired statement follows from the choice of $ c,\delta $, and $ D $.
\end{proof}
Our goal now is to obtain \eqref{eq:cube_conts} in the weighted case from the above lemmas.
For a constant weight it follows from the scale-invariance of $ E^k_s(\cdot\,) $. Namely, by the previous lemma and the existence of cube asymptotics for the unweighted energy $ E^k_s(\cdot\,) $ which was established above, we have for the unit cube $ \q_d $ and any $ \epsilon > 0 $
\[
    \l_{E^k} (\q_d) = \f(1) \geq (1-\epsilon) \limsup_{N\to \infty} \frac {E^k(\omega_{N}^*(D))}{N^{1+s/d}},
\] 
whenever there holds $ D\subset \q_d $ and $ \lambda_d(D) > 1- \delta_0 $, with $ \delta_0 (\epsilon)$ depending only on $ \epsilon $. 
For an arbitrary cube $ z+a\q_d\subset \Omega $ and a compact subset $ D\subset \q_d $ satisfying $  \lambda_d(D) > a^d(1-\delta_0) $, the rescaled and translated set $ -z+\frac1a D $ is such that $  \lambda_d\left(\frac1a D\right) > 1-\delta_0 $, so one has by $ s $-scale invariance
\[
    \omega_{N}^*\left(-z+\frac1a D\right) = a^s \omega_{N}^*\left(D\right),
\]
and
\[ 
    \l_{E^k} (z+a\q_d) =  a^{-s} \f(1) \geq a^{-s}(1-\epsilon) \limsup_{N\to \infty} \frac {E^k\left(\omega_{N}^*\left(-z+\frac1a D\right)\right)}{N^{1+s/d}} = (1-\epsilon) \limsup_{N\to \infty} \frac {E^k\left(\omega_{N}^*\left(D\right)\right)}{N^{1+s/d}}.
\]
This implies that for any cube $ z+a\q_d $ of Lebesgue measure $ a^d > 0 $ and a compact set $ A\subset (z+a\q_d) $ with $ \lambda_d(A) > a^d - \delta $, where $ \delta = \delta_0(\epsilon)a^d $, there holds
\[
    \liminf_{N\to\infty} \frac{\e(\omega_N^*(z+a\q_d))}{\t(N)} \geq  \limsup_{N\to\infty} (1 - \epsilon) \frac{\e(\omega_N^*(A))}{\t(N)}. 
\]
Finally, an inspection of the proof of Theorem~\ref{thm:weighted} shows that for a weighted energy $ E^k_s(\cdot\,;\kappa) $, it suffices to verify \eqref{eq:cube_conts} for cubes on which $ \kappa(x,y) $ is constant, which is immediate from the above. Once Theorem~\ref{thm:weighted} is obtained, a general continuity result follows.
This completes the proof of the continuity on cubes \eqref{eq:cube_conts} for weighted $ E^k_s(\cdot\,;\kappa) $.
The embedded continuity property for Definition~\ref{def:strong_cube_conts} follows from directly from Lemma~\ref{lem:upper_reg}.
We have now established the following 
\begin{proposition}
    For a positive bounded, continuous on the diagonal $ \kappa $, the $ k $-truncated Riesz energy
    \[
        E^k_s(\omega_N;\kappa) = \sum_{i=1}^N \sum_{j\in I_{i,k} } \kappa(x_i,x_j) \|{x}_i - {x}_j \|^{-s}
    \]
    is an admissible functional with weight $ \kappa(x,x) $ and rate $ \t = t^{1+s/d} $.
\end{proposition}

The final result in this section rephrases the above proposition, and is a counterpart to the prior results about Riesz energies \cite{hardinMinimal2005}, in particular those equipped with a weight \cite{borodachovAsymptotics2008}, or an external field \cite{hardinGenerating2017}. 

\begin{proof}[Proof of Theorem~\ref{thm:k_asympt}.]
    As discussed above, the $ k $-truncated Riesz energy
    \[
        E^k_s(\omega_N; \kappa, \ext) := \sum_{i=1}^N \sum_{j\in I_{i,k} } \kappa( x_i,  x_j) \|{x}_i - {x}_j \|^{-s} +  N^{s/d} \sum_{i=1}^N \ext( x_i), \qquad  s > 0,
    \]
    satisfies Definition~\ref{def:ext_field}, Definitions~\ref{def:short_range} and~\ref{def:cube_conts}. It follows that Theorem~\ref{thm:ext_dist} applies. This immediately gives the desired statement.
\end{proof}
\subsection{Minimizers on the circle}
\label{sec:1d}
When $ d = \dim_H\Omega = 1 $, we can compute explicitly the values of $ C_{s,1}^k $ for any $ s > 0 $ and $ k\geq 1 $. In fact, we will show that the minimizers of the energy $ E^k(\cdot\,; \ker) $ on the periodized interval $ [0,1] $ are equally spaced, for any convex decreasing kernel $ \ker $. Equivalently, minimizers of such energies on $ \mathbb S^1 $ with embedded distance are equally spaced points for any convex decreasing kernel.

\begin{theorem}
    \label{thm:1d} 
    Let $ \Omega = \mathbb S^1  $  with the metric $ \rho = s / 2\pi $ for the arc length $ s $, and assume that $ \ker(x,y) = \pphi(\rho(x , y)) $ for a convex decreasing $ \pphi:[0,\infty)\to [0,\infty) $. For any $ N\geq 2$ and $ k\geq 1 $, the energy
    \[
        E^k(\omega_N; \pphi) = \sum_{i=1}^N \sum_{j\in I_{i,k}} \pphi\left(\rho({x}_i , {x}_j) \right).
    \]
    is minimized by every set $ \omega_N^* $ that consists of equally spaced points.
\end{theorem}
\begin{proof}
    Consider an arbitrary set $ \omega_N $ of $ N $ distinct points in $ \Omega $. We will show that its energy is at least the one of $ \omega_N^* $, as defined above.
    We will assume that the elements of $ \omega_N = \{ x_1,\ldots,x_N \} $ are enumerated in the increasing order, so that for example $ x_1 $ and $ x_3 $ are adjacent to the point $ x_2 $, etc. We will also use indices of $ x_i $ modulo $ N $, so for any $ x_i $ the two adjacent points in the above ordering are given by $ x_{i-1} $ and $ x_{i+1} $.

    Consider the following sets of indices 
    \[
        \left\{ i-\left\lfloor \frac k2\right\rfloor, i-\left\lfloor \frac k2\right\rfloor +1,\ldots, i+\left\lceil \frac k2\right\rceil-1, i+\left\lceil \frac k2\right\rceil \right\}.  
    \] 
    Let $ I'_{i,k} $ stand for the above set, ordered by the nondecreasing distance  from $ j' \in I'_{i,k} $ to $ x_i $. Then, if $ j \in I_{i,k} $ and $ j' \in I'_{i,k} $ have the same position in the respective ordered set, there holds
    \[
        \rho(x_i - x_j) \leq  \rho(x_i - x_{j'}),
    \]
    since by definition given previously, $ I_{i,k} $ is the set of indices of the $ k $ nearest neighbors of $ x_i $, ordered by the nondecreasing distance to $ x_i $. By the monotonicity of $ \pphi $, it follows 
    \begin{equation}
        \label{eq:iprime}
            \sum_{j\in I_{i,k}} \pphi\left(\rho({x}_i , {x}_j) \right) \geq \sum_{j'\in I'_{i,k}} \pphi\left( \rho ({x}_i , {x}_{j'}) \right), \qquad 1\leq i \leq N.  
    \end{equation}

    Now observe that for any set of $ N $ distinct points $ \omega_N\subset \Omega $,
    \[
        \sum_{i=1}^N \rho(x_i , x_{i+1}) = 1.
    \]
    Indeed, the above sum contains the distances between adjacent points, which add up to the length of $ \Omega $. Similarly, one has
    \begin{equation}
        \label{eq:distance_sum}
        \sum_{i=1}^N \rho(x_i , x_{i+l})  = \sum_{i=1}^N \sum_{j=1}^l  \rho( x_{i+j-1} , x_{i+j}) = l,
    \end{equation}
    whenever $ 2k \leq N $.

    In view of \eqref{eq:iprime} and convexity of $ \pphi $, we obtain
    \[
        \begin{aligned}
        E^k(\omega_N; \phi)
        &= \sum_{i=1}^N \sum_{j\in I_{i,k}} \pphi\left( \rho({x}_i , {x}_j) \right) \geq \sum_{i=1}^N\sum_{j'\in I'_{i,k}} \pphi\left( \rho({x}_i , {x}_{j'}) \right) \\
        &= \sum_{\substack{l=-\lfloor k/2\rfloor\\ l\neq 0}}^{\lceil k/2 \rceil} \sum_{i=1}^N \pphi\left(\rho ({x}_i , x_{i+l} )\right) \geq \sum_{\substack{l=-\lfloor k/2\rfloor\\ l\neq 0}}^{\lceil k/2 \rceil} N \, \pphi \left(\frac1N \sum_{i=1}^N \left(\rho({x}_i , x_{i+l}) \right)\right) \\
        &= N  \sum_{\substack{l=-\lfloor k/2\rfloor\\ l\neq 0}}^{\lceil k/2 \rceil} \pphi \left(\frac lN \right) = E^k(\omega^*_N; \phi).
    \end{aligned}
    \]
    In the second line of the above equation we used Jensen inequality. Here, as defined in the statement of the theorem, $ \omega_N^* $ consists of $ N $ equally spaced points in $ \Omega $.
\end{proof}

\begin{corollary} 
    The value of the constant $ C_{s,1}^k $ is given by 
    \[
        C_{s,1}^k  = \sum_{\substack{l=-\lfloor k/2\rfloor\\ l\neq 0}}^{\lceil k/2 \rceil} \frac1{|l|^s}.
    \] 
\end{corollary}
\begin{proof}
    The unit circle $ \mathbb S^1 $ with the metric $ \rho $ above can be identified with the unit interval $ [0,1) $ equipped with the natural distance with periodization. Due to the properties of short-range energies, the asymptotics of the minimal energy for this set $ A' = [0,1) $ with this distance coincide with the asymptotics for $ A = [0,1] \subset \mathbb R $ with the Euclidean distance. 
\end{proof}

\subsection{Relation of \texorpdfstring{$ k $}{k}-truncated energies for \texorpdfstring{$ s>d $}{s>d} to hypersingular Riesz energies} 
\label{subsec:background} 
In this section we will clarify the relation between the $ k $-truncated Riesz energy discussed above and the non-truncated hypersingular Riesz energy $ E_s(\omega_N; \kappa, \ext) $ on $ \mathbb R^d $, given by  
\begin{equation}
    E_s(\omega_N; \kappa, \ext) := \sum_{i \neq j } \kappa( x_i,  x_j) \|{x}_i - {x}_j \|^{-s} +  N^{s/d} \sum_{i=1}^N \ext( x_i), \qquad  s > d.
\end{equation}
We show that for $ k\to \infty $, the asymptotics of $ E_s^k $ and $ E_s $ coincide.

Recall the following statement, characterizing the minimal asymptotics of this energy for a constant weight, as well as the behavior of its minimizers in the weak$ ^* $ topology. 

\begin{theorem}\cite{hardinGenerating2017} 
    \label{thm:lim_external}
Assume $s > d$. Let $A\subset \R^{d'} $ be a $ d $-rectifiable compact set, $ \mathcal{H}_d(A)>0 $. Let further $\ext$ be a lower semicontinuous function on $ A $, finite on a set of positive $ \Hd $-measure, and $ \kappa \equiv 1 $. 
Then every sequence of minimizing configurations $ \omega_N^*, N\geq 2, $ satisfies
\begin{equation}\label{eq:S_lim}
    \lim\limits_{N\to\infty}  \frac{ E_s(\omega_N^*; 1, \ext) }{N^{1+s/d}}=  \int\limits \frac{L_1+sq( x)/d}{1+s/d}  \, \d\mu( x) =: \mathfrak S(s,\ext),
\end{equation} 
where $ \mu $ is the probability measure given by
\begin{equation}\label{eq:L}
\frac{\d\mu }{ \d\Hd}( x) = \left(\frac{L_1-\ext( x)}{\c(1+s/d)}\right)^{d/s}_+.
\end{equation}
Conversely, if $ \{ {\omega}_N\}_{N\geq2} $ satisfies 
\begin{equation}\label{asymp_en_min}
	\lim_{N\to\infty} \frac{ E_s(\omega_N; 1, \ext) }{N^{1+s/d}} = \mathfrak S(s,\ext), 
	\end{equation}
then
\begin{equation*}
\frac{1}{N}\sum_{ x\in {\omega}_N} \delta_{  x}\weakto  \d\mu \quad \text{ as } N\to\infty.
\end{equation*}
\end{theorem}
An analogous theorem has been established by Borodachov and the first two authors for a non-constant weight and absence of external field \cite{borodachovAsymptotics2008}.
As has been demonstrated in \cite{hardinMinimal2005}, the hypersingular Riesz kernel gives rise to a short-range interaction. By the general approach given in Section~\ref{sec:short_range}, it follows that the results of \cite{borodachovAsymptotics2008} and \cite{hardinGenerating2017} can be combined to obtain a result identical to Theorem~\ref{thm:k_asympt}, with non-truncated sum of the hypersingular interaction with $ s>d $. Suppose the weight $ \kappa $ is continuous on the diagonal $ \diag (A\times A) $, positive, bounded, and the external field $ \xi $ is continuous on $ A $, then there holds
\begin{corollary}
    Suppose $ A \subset \mathbb R^{d'} $ is a compact $ d $-rectifiable set.
    For an $ s > d $, any sequence of minimizers of the hypersingular Riesz $ s $-energy with kernel $ \kappa $ and external field $ \ext $ has the energy asymptotics given by
    \[
        \begin{aligned}
            \lim_{N\to\infty} \min_{\omega_N\subset A} \frac{E_s(\omega_N; \kappa, \ext)}{N^{1+s/d}}  
            &= \c\int_A \kappa(x,x) \pphi( x)^{1+s/d}\d\Hd(x) + \int_A \ext( x) \pphi( x) \d\Hd(x)\\
            &=: \mathfrak S(s,\kappa,\ext),
    \end{aligned}
    \]
    where
    \begin{equation*}
        \pphi( x) = \left(\frac{L_1-\ext( x)}{\c(1+s/d)\kappa( x,  x)}\right)^{d/s}_+,
    \end{equation*} 
    and the constant $ L_1 $ is chosen so that $ \pphi \d\Hd $ is a probability measure.
    Furthermore, for any sequence of configurations $ \{ \omega_N \} $ achieving the energy asymptotics of  $ \mathfrak S(s,\kappa,\ext) $, the weak$ ^* $-limit of the corresponding counting measures is $ \pphi \d\Hd $.
\end{corollary}
As before, the continuity of the external field can be relaxed to lower semicontinuity (see the end of Section~\ref{sec:translation_dependent}), and the weight only needs to be continuous on the diagonal in the extended sense, assuming it grows slower than a polynomial (see Section~\ref{sec:cube} and the proof of Lemma~\ref{lem:upper_reg}).
The value of $ C_{d,d} $ here is given by
\begin{equation}\label{eq:c_dd}
 C_{d,d}:= {\mathcal{H}_d(\mathbb{B}^d)}= \frac{\pi^{d/2} }{\Gamma\left(d/2+1 \right) } , \quad d\geq1, 
\end{equation}
where $ \Gamma $ is the standard gamma function, \cite{hardinMinimal2005}.
Furthermore, for $ d=1,\ s> 1 $ there holds
\begin{equation}\label{eq:cs1}
 \begin{aligned}
 C_{s,1}= 2\zeta(s),&\quad s>1,
 \end{aligned} 
\end{equation}
where $ \zeta $ is the Riemann zeta function, see e.g. \cite{FiMaRaSa2004}. Lastly, the {\it universal optimality} of $ E_8 $ and the Leech lattice means that they minimize all energies with completely monotonic kernels as functions of the distance over discrete sets with fixed density. The Riesz kernel $ 1/r^s $ is completely monotonic (that is, its derivatives have alternating signs), and as a result, $ \c $, $ d=8,24 $, is related to the respective lattice as
\begin{equation}\label{eq:conj_val}
 \c=|\Lambda_d|^{s/d}\zeta_{\Lambda_d}(s),\qquad s>d, \quad d = 8,24.
\end{equation}
Here $ \Lambda_d $  denotes the lattice, either $ E_8 $ or the Leech lattice; $ |\Lambda_d| $ stands for the volume of the fundamental cell of $ \Lambda_d $, and $ \zeta_{\Lambda_d} $ is the corresponding Epstein  zeta-function. The universal optimality of these lattices was shown by Cohn, Kumar, Miller, Radchenko, and Viazovska \cite{cohnUniversal2019}, following the methods of Viazovska \cite{viazovskaSphere2017a}. 
The exact value of $ \c $ is unknown for all the other pairs $ s,d $. In dimensions $ d= 2,\ 4$, the conjectured value is also given by the expression \eqref{eq:conj_val} with $ \Lambda_d $, respectively, the hexagonal and $ D_4 $ lattices \cite[Conj. 2]{brauchartNextorder2012}. It is easy to show \cite[Prop. 1]{brauchartNextorder2012} that the  conjectured values \eqref{eq:conj_val} are upper bounds for their respective $ \c $.


In the remainder of this section, we will show that in the hypersingular case $ s>d $, the asymptotics of the full energy are obtained from those of the truncated $ k $-energy for $ k\to \infty $.  We will need the following lemma, which has been established in a slightly different form in \cite[Lem. 5.2]{borodachovLow2014}. 
\begin{lemma}
    \label{lem:offdiagonal}
    Let $ \omega_N \subset \Omega \subset \mathbb R^d $  be a sequence of discrete sets satisfying $ \Delta(\omega_N) > C N^{-1/d} $, $ N\geq 2 $. Then there holds
    \[
        \limsup_{N\to \infty}\frac1{N^{1+s/d}}\sum_{i=1}^N  \sum_{j\notin I_{i,k} } \|{x}_i - {x}_j \|^{-s} = o(1), \qquad k \to \infty.
    \]
\end{lemma}
\begin{proof}
    Fix an index $ i $. Denote $ CN^{-1/d} $ by $ 2h_N $ for brevity. For an $ m\geq1 $, let $$ L_m =  \{ j : mh_N < \|x_j - x_i\| \leq (m+1)h_N \}.  $$ 
    There holds $ |B(x_i,r)| =  c_p' r^d $ for any $ r>0 $, whence $ |B(x_i,r+t) \setminus B(x_i,r) | \leq c_p'' t(r+t)^{d-1} $, $ t\geq 0 $ for some positive constants $ c_p' $, $ c_p'' $. Since the distance between any two points in $ \omega_N $ is at least $ 2h_N $, open balls $ B(x_j, h_N) $ for $ 1\leq j \leq N $ must be pairwise disjoint. This allows to estimate $ \#L_m $ by volume considerations:
    \[
        \bigcup_{j\in L_m} B(x_j, h_N) \subset \left[B\left(x_i,(m-1)h\right) \, \big\backslash \, B\left(x_i,(m+2)h\right)\right], 
    \]
    where the union in the left-hand side is pairwise disjoint. Thus $ c_p' h_N^d \cdot \#L_m \leq  3c_p'' h_N((m+2)h_N)^{d-1}$, which gives
    \[ 
        L_m \leq  c_p m^{d-1}.
    \]
    Summing up the pairwise energies over spherical layers around $ x_i $, one obtains further
    \[ 
        \begin{aligned}
            \sum_{j=1}^N \|{x}_i - {x}_j \|^{-s} 
            &=  \sum_{m=1}^\infty \sum_{j\in L_m} \|{x}_i - {x}_j \|^{-s} \leq \sum_{m=1}^\infty \frac{c_p m^{d-1}}{(m h_N)^s} = \frac{c_p N^{s/d}}{C} \sum_{m=1}^\infty \frac1{m^{s-d+1}}.
        \end{aligned}
    \]
    Observe that when $k \geq \sum_{1}^{M-1} \# L_m$ and $ m <M $, $ L_m \subset I_{i,k} $. In view of the upper bound on $ \#L_m $ and the above equation,
    \[
        \frac1{N^{1+s/d}}\sum_{i=1}^N  \sum_{j\notin I_{i,k} } \|{x}_i - {x}_j \|^{-s} \leq c_pC^{-1} \sum_{m=M}^\infty \frac1{m^{s-d+1}} \qquad \hbox{ whenever }\quad k \geq \sum_{m=1}^{M-1} c_p m^{d-1},
    \]
    which gives the desired statement.  
\end{proof}


\begin{theorem}
    Suppose a sequence $ k(N) $, $ N\geq 2 $ satisfies $ k(N)\to \infty $  $ N\to \infty $. Then 
    \[
        \min_{\omega_N\subset\Omega} E^{k(N)}_s(\omega_N; \kappa,\ext) \Big/\min_{\omega_N\subset\Omega} E_s(\omega_N; \kappa, \ext)  \longrightarrow 1, \qquad N\to\infty,
    \]
    if this family of energies are considered on $ \mathbb R^d $.
\end{theorem}
\begin{proof}
    First, assume $ \ext \equiv 0 $ on $ \Omega $.
    Let $ \omega_N^* = \{ x_1^*,\ldots,x_N^* \} $ be a sequence of minimizing configurations for the $ k(N) $-truncating kernels, so that
    \[
        E^{k(N)}_s(\omega_N^*; \kappa, 0) = \min_{\omega_N\subset\Omega} E^{k(N)}_s(\omega_N; \kappa, 0),
    \]
    with $ k(N)\to \infty $, $ N\to \infty $. Similarly, let $ \omega_N' = \{ x_1',\ldots,x_N' \} $ be the sequence minimizing the non-truncated Riesz energy:
    \[
        E_s(\omega_N'; \kappa, 0) = \min_{\omega_N\subset\Omega}
        E_s(\omega_N; \kappa,  0). 
    \]
    By the construction of $ \omega_N^* $ and $ \omega_N' $, for every $ N $ there holds, 
    \[ 
        E^{k(N)}_s(\omega_N^*; \kappa, 0) + 
        \sum_{i=1}^N  \sum_{j\notin I_{i,k} } \kappa(x_i,x_j) \|{x}_i' - {x}_j' \|^{-s} \leq E_s(\omega_N'; \kappa, 0).
    \]
    On the other hand, the value of $ E_s(\omega_N^*; , \kappa,0) $ can be written as
    \[
        \begin{aligned}
            E_s(\omega_N^*; \kappa, 0) 
            &=  \Biggl(\sum_{j\in I_{i,k} } +\sum_{j\notin I_{i,k} }\Biggr)\sum_{i=1}^N \kappa(x_i, x_j) \|{x}_i - {x}_j \|^{-s}\\
            &= E^{k(N)}_s(\omega_N^*; \kappa, 0) + \sum_{j\notin I_{i,k} } \kappa(x_i,x_j) \|x_i^* - x_j^* \|^{-s}.
    \end{aligned}
    \]
    Since the multiplicative weight satisfies $ \kappa(x_1, x_2) \leq \kappa_1 $ for some constant $ \kappa_1 $, combining the last two equations gives
    \[ 
        E_s(\omega_N^*; \kappa, 0) \leq E_s(\omega_N'; \kappa, 0) + \sum_{j\notin I_{i,k} } \kappa(x_i,x_j) \|x_i^* - x_j^* \|^{-s} - \sum_{j\notin I_{i,k} } \kappa(x_i,x_j) \|x_i' - x_j' \|^{-s}.
    \]
    Both $ \omega_N^* $ and $ \omega_N' $ are optimally separated, so Lemma~\ref{lem:offdiagonal} implies
    \[
        E_s(\omega_N^*; \kappa, 0) \leq E_s(\omega_N'; \kappa, 0) + o(N^{1+s/d}),
    \]
    which completes the proof in the case $ \ext \equiv 0 $. 
    The general case follows from weak$ ^* $-convergence for a continuous function $ \ext $.
\end{proof}

\subsection{Optimal quantizers}
\label{sec:CVT}
In this section we discuss the asymptotic behavior of optimal quantizers. Our goal is to verify that $ \cvt $ is an admissible functional in the sense of Section~\ref{sec:short_range}, so that Theorem~\ref{thm:general_uniform} applies. We will also verify that $ \cvt $ can be augmented with a continuous weight and external field, and the weighted version of $ \cvt $ is an admissible functional with weight as defined in Section~\ref{sec:translation_dependent}. Throughout the present section we have a fixed strictly convex norm $ \|\cdot \|: \mathbb R^d \to \mathbb R_+ $, such as one of the $ l_r $ norms $ \|\cdot\|_r $ for $ r > 1 $.  The {\it quantization error} of a configuration $ \omega_N\subset A $ for a compact $ A $ is defined as
\begin{equation}
    \label{eq:cvt}
    \cvt(\omega_N, A) = \sum_{i=1}^N \int_{V_i} \|y-x_i\|^p \d\lambda_d(y) = \int_A \min_i \|y-x_i\|^p \d\lambda_d(y),
\end{equation}
where $ p > 0 $ and $ V_i $ is the Voronoi cell of $ x_i $ on $ A $ with respect to the norm $ \|\cdot \| $:
\[
    V_i = \{ y\in A : \|y-x_i\| \leq \|y-x_j\|, \ j\neq i \}.
\]
The {\it Voronoi tessellation } of $ A $ corresponding to the configuration $ \omega_N $ is defined as the collection of Voronoi cells $ V_i $, $ 1\leq i \leq N $.
For a strictly convex norm, $ V_i $ are disjoint $ \lambda_d $-almost everywhere \cite[Thm. 1.5]{grafFoundations2000}. 

Clearly, $ \cvt $ is a translation invariant functional. In addition, there holds 
\[
    \cvt(T(\omega_N), T(A)) = c^{p+d}\cdot\cvt(\omega_N, A),
\]
for a map $ T:\mathbb R \to \mathbb R $ which is a similarity with scaling $ c >0 $.
It follows that the functional $ \cvt $ is $ -(p+d) $-scale invariant in the sense of Section~\ref{sec:cube}, which by the results of that section implies that the value of $ \sigma $ in the exponent of the fractional rate function $ \w $ must be
\[
    \sigma = - 1 - p/d,
\]
so that $ \w(t) = t^{-p/d} $.

As before, we write $ \omega_N^*(A) $ for the configuration attaining
\[
    \cvt(\omega_N^*(A)) := \min\{ \cvt(\omega_N, A)  : \omega_N \subset A \}.
\]
Note that in the terminology of Graf-Luschgy \cite{grafFoundations2000}, quantity~\eqref{eq:cvt} is the error of quantizing the measure $ \mathbb 1_A \lambda_d $ with the configuration $ \omega_N $. The difference between our approach and  \cite{grafFoundations2000} is that in the latter, optimal quantizers are the solutions of 
\[
    \min\{ \cvt(\omega_N, A)  : \omega_N \subset \mathbb R^d \},
\]
whereas we constrain optimization to $ A $. 
In the case of convex $ A $, $ \omega_N^*(A) $ belongs to $ A $, and the two definitions agree. Furthermore, $ \omega_N^*(A) $ converges to $ A $ asymptotically in the unconstrained case. Some additional discussion of how our approach applies to the unconstrained optimization can be found in Section~\ref{sec:further}.

To obtain quantizations of a nonuniform measure, absolutely continuous with respect to $ \lambda_d $, we modify $ \cvt $ with a continuous weight, as we did for the $ k $-truncated Riesz energy:
\begin{equation*}
    \cvt(\omega_N, A; \h) = \sum_{i=1}^N \int_{V_i} \h(y)\, \|y-x_i\|^p \d\lambda_d(y),
\end{equation*}
which gives a functional without translational invariance. Similarly to how it was done for the $ k $-truncated energies, $ \cvt $ can further be equipped with an external field term. 

As in the argument for the truncated Riesz functional, we start with the unweighted case,
$ \rho \equiv 1 $. To apply the framework of Section~\ref{sec:general} to the quantization error function $ \cvt $, we need to verify Definitions~\ref{def:cube_asymp},~\ref{def:short_range}, and~\ref{def:cube_conts}.
First, let us verify existence of asymptotics on cubes in Definition~\ref{def:cube_asymp}.

\medskip
\noindent{\bf Verification of \eqref{eq:cube_asymp}.}
By the results of Section~\ref{sec:cube} and scale-invariance of $ \cvt $, it suffices to show that 
\begin{equation}
    \label{eq:cvt_order}
    0 < \liminf_{N\to\infty} N^{p/d} {\cvt(\omega_N^*(\q_d))} \leq \limsup_{N\to\infty} N^{p/d} {\cvt(\omega_N^*(\q_d))} < +\infty,
\end{equation}
and verify Definition~\ref{def:local}. 

There exists an optimal covering configuration $ \omega_N^c $ on $ \q_d $, for which 
\[
    \min_i \left\{\|y - x_i^N\| : x_i^N\in \omega_N^c  \right\} \leq C(A)N^{-1/d}
\]
holds for every $ y\in \q_d $. Evaluating $ \cvt $ on such $ \omega_N^c $ gives
\[
    \cvt(\omega_N^c, \q_d) =  \int_{\q_d} \min_i \|y-x_i\|^p \d\lambda_d(y) \leq C(A)^pN^{-p/d},
\]
which yields the upper bound in~\eqref{eq:cvt_order}. The lower bound follows from an  isoperimetric inequality for optimal quantizers:
\begin{theorem}[{\cite[Thm. 4.16]{grafFoundations2000}}]
    A collection of equal balls $ \{ B_i \}_1^N \subset \mathbb R^d $ with disjoint interiors attains the minimum of $ \cvt(\omega_N^*(A)) $ over compact sets $ A $ of the same $ \lambda_d $-measure.
\end{theorem}
\noindent From this theorem,
\[
    \inf\left\{ \cvt(\omega_N^*(A)) : \lambda_d(A) = 1, \ A\text{ compact} \right\} =  N^{-p/d}\int_{B(0,1)} \|y\|^p \, d\lambda_d(y),
\]
implying the lower bound for $ \q_d $.

To check that Definition~\ref{def:local} applies, it suffices to observe that for $ \cvt $ the expression in the limit~\eqref{eq:locality} is continuous in $ \gamma $, and one may simply substitute $ \gamma = 1 $. Then, using translational invariance and the following property of $ \cvt $ \cite[Lem. 4.14b]{grafFoundations2000}, which holds for an arbitrary configuration $ \omega \subset \bigcup_m A_m $:
\begin{equation}
    \label{eq:concavity}
    \cvt\left(\omega,\, \bigcup_m A_m\right) \leq \sum_m \cvt\left(\omega\cap A_m,\,  A_m\right),
\end{equation}
we have the desired inequality~\eqref{eq:locality}. This completes the verification of existence of asymptotics on cubes~\eqref{eq:cube_asymp} for constant weights.   

For a continuous weight $ \h $, we obtain the equation \eqref{eq:weighted} from the mean value theorem for integrals, from continuity of $ \h $. Indeed, for positive weights $ \h_1 \leq \h_2 $ there holds
\[
    \cvt(\omega_N^*(A), A; \h_1) \leq \cvt(\omega_N^*(A), A; \h_2).
\]
Our claim follows for a general continuous, strictly positive $ \h $ by comparison to $ \min_{y\in A} \h(y) $ and $ \max_{y\in A} \h(y) $.

\medskip
\noindent{\bf Verification of \eqref{eq:short_range}.}
To verify Definition~\ref{def:short_range}, we need a result of Gruber \cite{gruberOptimum2004} about the covering of optimal quantizers. In order to formulate it, recall that a {\it d-regular set} $ A $ by definition satisfies $ cr^d \leq  \mathcal H_d (A\cap B(x,r)) \leq Cr^d $ for the $ d $-dimensional Hausdorff measure $ \mathcal H_d $, every $ x\in A $, and $ 0\leq r \leq \diam A $.
\begin{proposition}[{\cite[Stmt. (2.19)]{gruberOptimum2004}}]
    \label{prop:covering_cvt}
    If the compact set $ A $ is $ d $-regular, any sequence of optimal quantizers $ \omega_N^*(A) $ has the optimal order of covering: that is, there exists a constant $ C $ such that 
    \[
        \min_i \left\{\|y - x_i^N\| : x_i^N\in \omega_N^*(A)  \right\} \leq C(A))N^{-1/d}, \qquad N \geq 1,
    \]
    for any $ y \in A $.
\end{proposition}
\noindent Observe that by~\eqref{eq:concavity}, the value of the limit in \eqref{eq:short_range} is at least 1. On the other hand, the difference between the numerator and denominator in \eqref{eq:short_range} is due to the points $ y\in A\cap Q_m $, $ m=1,\ldots,M $, such that the nearest element to $ y $ in $ \omega_N\cap \Q $ belongs to a different cube $ Q_{l} $, $ l\neq m $. For brevity, we will give the argument in the case $ M = 2 $; the result for a general $ M $ follows similarly. Using the notation
\[
    Z_m = \{ y\in A\cap Q_m : \dist(y, \omega_N\cap Q_m) > \dist(y, \omega_N\cap Q_{3-m}) \}, \qquad m= 1,2,
\]
one has 
\[
    \begin{aligned}
        \e(\omega_N\cap Q_1, A\cap Q_1)
        &+\e(\omega_N\cap Q_2,A\cap Q_2) - \e\Big(\omega_N\cap [Q_1\cup Q_2],\ A\cap [Q_1\cup Q_2]\Big) \\
        &\leq \sum_{m=1,2} \int_{Z_m} \h(y) \dist\left(y,\omega_N\cap Q_m\right)^p\,d\lambda_d(y).
    \end{aligned}
\]
On the other hand, any line segment connecting $ y\in A\cap Q_m $ to the complement $ Q_m^c $ must intersect $ \partial Q_m $, whence by Proposition~\ref{prop:covering_cvt},
\[
    \begin{aligned}
    Z_m 
    &\subset \{ y\in A\cap Q_m : \dist(y, \omega_N\cap Q_m) > \dist(y, \partial Q_m) \} \\
    &\subset \{ y\in A\cap Q_m : \dist(y, \partial Q_m) \leq C(A) N^{-1/d} \},
    \end{aligned}
\]
and as a result, the measure of $ Z_m $ can be bounded as $ \lambda_d(Z_m) \leq C(A) N^{-1/d} \lambda_{d-1} (\partial Q_m) $. Combining this with the lower bound from~\eqref{eq:cvt_order}, we have
\[
    \limsup_{N\to\infty} \frac{ \sum_{m=1,2} \int_{Z_m} \h(y) \dist\left(y,\omega_N\cap Q_m\right)^p\,d\lambda_d(y) }{ \e(\omega_N(A) \cap [Q_1\cup Q_2]) } 
    \leq \limsup_{N\to\infty} \frac{\|\h\|_\infty \left(C(A) N^{-1/d}\right)^{1+p}\, 2 \lambda_{d-1} (\partial Q_m) }{ C(A)N^{-p/d} } = 0,
\]
which completes the proof of equality~\eqref{eq:short_range} when $ \omega_N $ is a minimizer on $ A $. The case of piecewise minimizer is handled in the same fashion. 

To verify the short-range property from Definition~\ref{def:strong_short_range}, when optimal quantizers are considered in the embedded case, suppose the compact sets $ A_1$,  $A_2 $ satisfy the necessary smoothness assumptions and are distance $ R>0 $ apart. It suffices to use $ N $ so large that for the covering radius estimate there holds $ C(A) N^{-1/d} < R $, since then the sets $ \tilde Z_m $ are empty:
\[
    \tilde Z_m = \{ y\in A_m : \dist(y, \omega_N\cap A_m) > \dist(y, \omega_N\cap A_{3-m}) \}, \qquad m= 1,2.
\]
Arguing as above for the full-dimensional case, we have the embedded short-range property.

\medskip
\noindent{\bf Verification of \eqref{eq:cube_conts}.}
Finally, we discuss Definition~\ref{def:cube_conts}. We will verify Definition~\ref{def:strong_cube_conts} directly; the full-dimensional property of Definition~\ref{def:cube_conts} then follows by scale-invariance of $ \cvt $. We assume that $ A $ is a $ d $-rectifiable and $ d $-regular compact subset of $ \mathbb R^{d'} $, $ d'\geq d $.
In view of $ \sgn \sigma = -1 $, \eqref{eq:strong_cube_conts} takes the form
\begin{equation}
    \label{eq:newform_iii}
    \limsup_{N\to\infty} N^{p/d} {\cvt(\omega_N^*(A))} \leq  
    \limsup_{N\to\infty} (1+ \epsilon) N^{p/d}{\cvt(\omega_N^*(D))},
\end{equation}
which must hold whenever the compact $ D \subset A $ satisfies $ \mathcal H_d(D) > \mathcal H_d (A)(1 - \delta(A, \epsilon) )$. We further suppose that $ A $ is $ d $-regular.

In view of the $ d $-regularity of $ A $, it is possible to insert at most $ c\delta N $ disjoint balls with radius $ CN^{-1/d} $ into the relative interior of $ A\setminus D $. This implies, for a maximal collection of such disjoint balls, its covering radius is again $ CN^{-1/d} $ (with a twice as big constant).
Denote the centers of a maximal collection $ \omega_n $, so that $ n =\# \omega_n \leq c\delta(A,\epsilon) N $, and let $ \omega = \omega_n \cup \omega_{N-n}^*(D) $. The value of the limit in the left-hand side of \eqref{eq:newform_iii} now can be estimated using \eqref{eq:concavity} and $ \omega $:
\[
    \begin{aligned}
        \limsup_{N\to\infty} &\, N^{p/d} {\cvt(\omega_{N}^*(A))}\\
        &\leq   \limsup_{N\to\infty} N^{p/d} {\cvt(\omega, A)} \\
        &\leq\limsup_{N\to\infty} N^{p/d} [ E(\omega_n, A\setminus D) + E(\omega_N^*(D)) ]  \\
        &\leq   c C  \delta + \limsup_{N\to\infty}\frac{N^{p/d}}{(N-c\delta N)^{p/d}} \cdot (N-n)^{p/d}{\cvt(\omega_{N-n}^*(D))},\\
        &\leq    \limsup_{N\to\infty} (1+\epsilon) \cdot N^{p/d}{\cvt(\omega_N^*(D))},
    \end{aligned}
\]
for sufficiently small $ \delta(A,\epsilon) $, where for the last inequality we used the $ d $-regularity of $ A $. Indeed, due to the $ d $-regularity, there holds $ \limsup N^{p/d}{\cvt(\omega_N^*(D))} \geq C > 0 $ \cite[(2.19)]{gruberOptimum2004}

The above gives~\eqref{eq:newform_iii}. Continuity of $ \ee $ with a constant weight suffices for our purposes, see the proof of Theorem~\ref{thm:embedded}. This concludes the proof of \eqref{eq:strong_cube_conts}, and the proof of admissibility of $ \cvt $, both in the full-dimensional and embedded case.

Due to the admissibility of the optimal quantization error as an interaction functional, we obtain the applications of Theorems~\ref{thm:weighted_dist} and~\ref{thm:embedded} to this context. We formulate the result in the embedded case as an example. Observe that it strengthens the results of Gruber~\cite{gruberOptimum2004} to $ A $ being $ d $-rectifiable instead of a smooth manifold, and in addition includes the external field.
\begin{theorem}
    \label{thm:quantizers}
    Suppose $ p >0 $, $ A \subset \mathbb R^{d'} $ is a $ d $-rectifiable $ d $-regular set and $ \cvt $ is the quantization error with continuous weight $ \h\geq h_0 > 0 $, and continuous external field $ \xi\geq0 $:
    \[
        \cvt(\omega_N; \h,\xi) = \int_{A} \h(y) \min_{i} \|y-x_i\|^p \,d\mathcal H(y) + N^{p/d-1} \sum_{i=1}^N \xi(x_i),\qquad A\subset \mathbb R^{d'},\ d'\geq d.
    \]
    Then 
    \[
        \lim_{N\to\infty} {N^{p/d}}{{\e_{\h,\xi}}(\omega_N^*( A))} =
        c_{p,d} \int_{A} \h(x)\,\pphi(x)^{-p/d}\, d\mathcal H_d(x) + \int_{ A} \xi(x) \, \pphi(x)\, d\mathcal H_d(x),
    \]
    where $ \pphi $ is the density of a probability measure $ \mu $ supported on $  A $, and is given by    
    \[
        \pphi(x) = \frac{d\mu}{d\mathcal H_d}(x) = \left(\frac pd \cdot\frac{c_{p,d}\h(x)}{\xi(x)-L_1}\right)^{\frac{d}{p+d}}_+
    \]
    for a normalizing constant $ L_1 $. If $ \mathcal H_d(A) > 0 $, the weak$ ^* $-limit of the counting measures for any sequence of configurations attaining the optimal asymptotics is $ \pphi \d\Hd $

    Moreover, when $ \xi \equiv 0 $,  formula for the density simplifies to
    \[
        \pphi(x) = \frac {\h(x)^{\frac{d}{p+d}}}{\int_A \h(x)^{\frac{d}{p+d}} \, d\mathcal H_d(x)}.
    \]
\end{theorem}
Note that the converse result about the weak$ ^* $-limit of an asymptotically optimal sequence follows from the convexity of the $ \Gamma $-limit of the energy functional, see Remark~\ref{rem:uniqueness}.
\noindent The constant $ c_{p,d} $ in the theorem is given by
\[
    c_{p,d} = \lim_{N\to\infty} N^{p/d}\, \e(\omega_N^*(\q_d)),
\]
which is the value of $ \f(1) $ in this context, see Definition~\ref{def:cube_asymp}.

\subsection{Meshing algorithms and short-range interactions}
\label{sec:meshing}
In this section we consider the full-dimensional problem only; all the sets below are in $ \mathbb R^d $ and have positive $ d $-dimensional Lebesgue measure.
We will discuss an algorithm for generating point distributions, due to Persson and Strang \cite{Persson2004a}, and how it fits into the framework of Sections~\ref{sec:short_range}--\ref{sec:translation_dependent}. Given a compact set $ A $, the algorithm produces a configuration in it, stable under the discrete dynamics determined by the Euler's method for an ODE. Persson and Strang are interested in the mesh given by the Delaunay triangulation of this configuration; the triangulation also determines the dynamics. Let $ T_i $ be the indices of vertices connected to $ x_i \in \omega_N $ in the Delaunay triangulation of $ \omega_N $. Vectors of the form $ x_j - x_i, \ j\in T_i $, will be called {Delaunay edges}. 

Persson and Strang introduce \cite[Fig. 3.1]{Persson2004a} a discrete update of $ \omega_N $ of the form
\begin{equation}
    \label{eq:dynamics}
    \omega_N^{l+1} = \omega_N^{l} + \Delta t \cdot F(\omega_N^l),
\end{equation}
where the superscript $ l $ denotes the number of the iteration, and the addition is performed in the vector space $ \mathbb (R^d)^N $. The force $ F $ is given by a modified Hooke's law, which is limited to only being repulsive; more precisely, it is the sum of forces directed along the edges of the Delaunay graph, with magnitudes given by
\[
    f_{i,j} = \left(  (1+P) \cdot \sqrt {m_2(\omega_N)} - \|x_j-x_i\|^2 \right)_+,
\] 
where
\[
    m_2(\omega_N) = \frac{\sum_{i=1}^N\sum_{j\in T_i} \|x_j-x_i\|^2}{2\sum_{i=1}^N \#T_i} 
\]
is the average square of edge length, and $ P $ a fixed strictly positive constant (a useful choice is $ P=1.2 $ for $ d=2 $ \cite{Persson2004a}).
Hence, $ F $ can be viewed as a gradient of the sum of energies of all edges, given by the Hooke's law:
\begin{equation}
    \label{eq:persson_strang}
    {\hat\e}(\omega_N) = \sum_{i=1}^N \sum_{j\in T_i} \frac12 \left(  (1+P) \cdot m_2(\omega_N) - \|x_j-x_i\|^2 \right)_+.
\end{equation}

Intuitively, $ {\hat\e} $ is the sum of energies from the compressions of springs which correspond to the Delaunay edges. The uncompressed lengths of these springs are prescribed by $ m_2(\omega) $ -- the quadratic average of the lengths of Delaunay edges. In \eqref{eq:persson_strang}, all the springs in the uncompressed state are of the same length, but this method also extends to non-uniform distributions. The approach of Persson and Strang is to introduce the weight into the first term of the inner sum, but the method of Section~\ref{sec:translation_dependent}, consisting in adding a multiplicative weight to the $ i $-th term of the outer sum, is applicable too.

Notice that, due to $ {\hat\e} $ being $ -2 $-scale invariant, the results of Section~\ref{sec:cube} imply
\[
    \sigma = -2/d,
\]
so for this functional the rate $ \t(N) = N^{1-2/d} $, a concave function. Thus, as explained in Section~\ref{sec:generalizations}, we need to consider the problem of maximizing $ {\hat\e} $:
\[
    \max \{ {\hat\e}(\omega_N) : \omega_N \subset A \}.
\]
In accordance with the notation of Section~\ref{sec:generalizations}, we write $ \E $ for the meshing functional (so that the notation of $ \e $ is restricted to functionals that are being minimized).
It is easy to see that the minimal value, $ {\hat\e}(\omega_N^0) = 0 $ is attained for the configuration consisting of $ N $ identical points: $ \omega_N^0 = \{ x_0, x_0,\ldots x_0 \} $. This poses a curious contradiction to the intuitive idea of the dynamics \eqref{eq:dynamics} as minimizing the energy of the system of springs, suggested by Persson and Strang. 

We will additionally assume that 
\[
    T_i \subset I_{i,k}
\]
for a large fixed $ k $; as before, $ I_{i,k} $ are the indices of the $ k $ nearest neighbors to $ x_i $. For a random configuration, the expected maximal cardinality of $ T_i $ grows with $ N $ as $ \log N / \log\log N $ \cite{bernExpected1991}, but the expected number of edges is only $ k(d)n $ \cite{dwyerExpected1993}. While it will be clear from the following discussion that there exists a sequence of minimizers in which $T_i \subset I_{i,k}$ holds, whether every sequence of minimizers satisfies this assumption remains open.

Let us verify that the functional $ {\hat\e} $ above is admissible — that is, has the properties from Definitions~\ref{def:cube_asymp}, \ref{def:short_range}, and \ref{def:cube_conts}, is upper semicontinuous, satisfies the monotonicity equation \eqref{eq:monotonicity} — and has concave rate function, which is what causes one to consider its maximizers. 
Note that we need to verify the monotonicity property~\eqref{eq:monotonicity} and Definitions~\ref{def:cube_asymp}, \ref{def:short_range}, \ref{def:cube_conts} with the maximizers of $ {\hat\e} $ instead of minimizers. Notice first that $ \E $ is not in fact upper semicontinuous. To circumvent this issue, we will assume that the corresponding maximizers are obtained for the upper semicontinuous envelope of $ \E $.  In practice, it means that if a subcollection of points in $ \omega_N $ lies on a sphere, the Delaunay triangulation of $ \omega_N $ is chosen so as to maximize the value of $ \E $ on $ \omega_N $. 
By an abuse of notation, we will still write $ \E $ to denote the envelope.

\medskip
\noindent{\bf Verification of monotonicity \eqref{eq:monotonicity}.} Writing $ \omega_N^*(A) $ for maximizers of $ {\hat\e} $ on the compact set $ A $, we have to check
\begin{equation*}
        \limsup_{N\to\infty}  \frac{{\hat\e}(\omega_N^*(A))}{\t(N)}   \leq
         \liminf_{N\to\infty}  \frac{{\hat\e}(\omega_N^*(B))}{\t(N)}      \qquad\text{for}\quad  A\subset B.
\end{equation*}
However, this is obvious: maximization in the right-hand side occurs on the larger set $ B $.

\medskip
\noindent{\bf Verification of \eqref{eq:cube_asymp}.}
Note that $ m_2(\omega_N) \leq C_d N^{-2/d} $ for any configuration $ \omega_N \subset \q_d $, by a volume argument.
It follows that $ {\hat\e}(\omega_N^*(\q_d)) \leq CN^{1-2/d} $. The converse inequality with a different constant follows trivially by taking a cubic lattice  in $ \q_d $.

By Section~\ref{sec:cube}, we have to verify that $ {\hat\e} $ satisfies Definition~\ref{def:local}; in the case of maximization of $ {\hat\e} $, the inequality necessary for locality takes the form
\begin{equation}
    \label{eq:local_maximizing}
    \limsup_{n\to \infty}
    \limsup_{L\to\infty} 
    \frac
    {{\hat\e}(\omega)}
    {L^d\,{\hat\e}\left( \frac1{L} \omega_n^* \right)}
    \geq 1,
\end{equation}
with 
\[
    \omega = \bigcup_{
    \boldsymbol i\in\, L\mathbb Z^d }
    \left(\frac1{L}\omega_n^* +  \frac{\boldsymbol i}L\right).
\]
Indeed, Delaunay edges appearing in the expression $ {\hat\e}(\omega) $ include the Delaunay edges in the same tile $ \frac1{L}\omega_n^* +  \frac{\boldsymbol i}L $, and the edges with endpoints in different tiles. Rescaling a configuration results in a rescaling of the distribution of lengths of Delaunay edges, so the average of the edges in $ \omega $ differs from that in $ \omega_n^* $ only due to the edges with endpoints in different tiles. In view of $ m_2(\omega_n^*) \leq Cn^{-2/d} $, by using $ \gamma\omega_n^* $ in place of $ \omega_n^* $ with $ \gamma = 1 -  Cn^{-2/d} $, it is possible to guarantee that the lengths of edges connecting different tiles are at least $ m_2(\omega_n^*)/L $, and therefore do not decrease the average. The contribution to $ {\hat\e}(\omega) $ made by edges with endpoints in the same tile is diminished by at most a factor of $ (1 -  Cn^{-2/d})^2 $; since the edges connecting different tiles can only increase $ {\hat\e}(\omega) $, equation~\eqref{eq:local_maximizing} follows by taking $ n\to \infty $.

\medskip

The following two properties will be verified in smaller generality than described in Section~\ref{sec:short_range}: we will discuss only the case of $ A $ with $ \lambda_d(A) =0 $.

\medskip

\noindent{\bf Verification of \eqref{eq:short_range}.} 
We will check the short-range property for a collection of cubes with disjoint interiors, but with shared faces. By scale-invariance, it suffices to consider a connected collection of cubes from $ \mathbb Z^d $. Indeed, if $ \Q_1 $, $\Q_2 $ are two such collections positive distance apart, edges longer than $ \sqrt{m_2(\omega_N)} \leq CN^{-1/d} $ do not contribute to the value of $ {\hat\e}(\omega_N) $, $ \omega \subset (\Q_1\cup \Q_2) $, which gives~\eqref{eq:short_range}.

Consider a pair of cubes $ Q_1 $, $ Q_2 $ and a minimizing configuration $ \omega_N^*(Q_1\cup Q_2) $. By the Remark~\ref{rem:one_sided_short_range}, it suffices to show that for suitably chosen $ N_1+N_2 = N $,
\begin{equation*}
    \limsup_{N\to\infty} \frac{{\hat\e}(\omega^*_{N_1}(Q_1)) + {\hat\e}(\omega_{N_2}^*(Q_2))}{{\hat\e}(\omega_N^* (Q_1\cup Q_2))} \geq 1,
\end{equation*}
where we have reversed the inequalities compared to those in the remark, to reflect the use of maximizers of $ {\hat\e} $ in this section. Suppose the above inequality does not hold, so that it is possible to construct a sequence of minimizers $ \omega_n(Q_1\cup Q_2) $,  $ n\in \mathcal N \subset \mathbb N $ for which
\begin{equation*}
    \lim_{n\to\infty} \frac{{\hat\e}(\omega^*_{n_1}(Q_1)) + {\hat\e}(\omega_{n_2}^*(Q_2))}{{\hat\e}(\omega_n^* (Q_1\cup Q_2))}  = \alpha < 1, \qquad n_1 + n_2 =n.
\end{equation*}
This implies in turn, that using the tiling of $ \q_d $ from~\eqref{eq:local_maximizing}, one obtains a configuration $ \omega \subset \q_d $, for which
\[
    \limsup_{n\to \infty}
    \limsup_{L\to\infty} 
    \frac
    {{\hat\e}(\omega)}
    {L^d\,{\hat\e}\left( \frac1{L} \omega_n^* \right)}
    \geq \alpha,
\]
a contradiction to \eqref{eq:cube_asymp}, which has been verified for $ {\hat\e} $ above. This contradiction gives the desired short-range property.

\medskip
\noindent{\bf Verification of \eqref{eq:cube_conts}.} 
Applying the two previous properties to sets that consist of unions of cubes in a lattice gives the formula
\[
    \lim_{N\to\infty} \frac{{\hat\e}(\omega_N^*(A))}{N^{1-2/d}} = C_{\hat\e} (\lambda(A))^{2/d}
\]
for such sets $ A $. The general continuity property of \eqref{eq:cube_uniform} for sets with boundary of $ \lambda_d $-measure zero then follows by the monotonicity and the fact that such sets are Jordan-measurable.

\medskip

The above discussion can be summarized in the following result.  
\begin{theorem}
    Suppose $ d\geq 3 $ and the compact set $ A\subset \mathbb R^d $ satisfies $ \lambda_d(\partial A) =0 $. If $ T_i \subset I_{i,k} $ for $1\leq i \leq N$ and a fixed $ k=k(d) $, there holds
    \[
        \lim_{N\to \infty} \frac{\sup_{\omega_N\subset A} \hat\e(\omega_N)}{N^{1-2/d}} = C_\e(P)\, \lambda_d(A)^{2/d}.
    \]
    If additionally $ \lambda_d(A)>0 $, the maximizers of the upper semicontinuous envelope of $ \hat\e $ converge to the uniform measure on $ A $.
\end{theorem}
The results concerning the non-translation-invariant modifications of the above interaction can be obtained similarly, see Section~\ref{sec:translation_dependent}. As already noted, the variable density version of the Persson-Strang algorithm \cite{Persson2004a} involves modifying the term $ m_2(\omega_N) $ in 
\[
    {\hat\e}(\omega_N) = \sum_{i=1}^N\sum_{j\in T_i} \frac12 \left( (1+P) \cdot m_2(\omega_N) - \|x_j-x_i\|^2 \right)_+,
\]
depending on the location of $ x_i $ in $ A $. It appears that the same principle as that of the proofs in Section~\ref{sec:translation_dependent} (partitioning $ A $ and approximating the interaction with its Riemann sums) allows to obtain an expression for the limiting density in this approach.

\section{\texorpdfstring{$ \Gamma $-convergence of the short-range energies}{Gamma-convergence of the short-range energies}}\label{sec:gamma}
\subsection{\texorpdfstring{Definition and essential properties of $ \Gamma $-convergence}{Definition and essential properties of Gamma-convergence}}
For a compact $ A $ we denoted by  $ \mathcal P(A) $ the space of probability measures supported on $ A $. It is a  compact metrizable space. We shall discuss the properties of short-range interactions on discrete configurations $ \omega_N $, by viewing them as acting on the normalized counting measures of $ \omega_N $. First, recall the notion of $ \Gamma $-convergence.
\begin{definition}[\cite{degiorgiSu1975}]
    \label{def:gcon}
    Let $ X $ be a metric space. Suppose that functionals $ F,\, F_n:X\to \mathbb R,\, n \geq 1, $ satisfy
    \begin{gammalist}
        \item\label{it:gamma_lower} for every sequence $ \{  x_n\}\subset X$ such that $ x_n \to x,\, n\to \infty $, there holds $ \liminf_{n\to\infty} F_n(x_n) \geq F(x) $;

        \item\label{it:recovery_seq} for every $ x\in X $ there exists a sequence $ \{x_n\} \subset X $ converging to it and such that $ \lim_{n\to\infty} F_n(x_n) = F(x) $.  
    \end{gammalist}
    We shall then say that the sequence $ \{F_n\} $ is $ \Gamma $-\textit{converging} to the functional $ F $ on $ X $ with the metric topology; in symbols, $ \glim_{n\to\infty} F_n = F .$ \end{definition}
 The sequence in \ref{it:recovery_seq} is called \textit{recovery sequence}. Usefulness of $ \Gamma $-convergence for energy minimization consists in that together with compactness of $ X $, it guarantees that minimizers of $ F_n $ converge to those of $ F $. Moreover, $ F_n $ needs not to attain its minimizer, but this is the case for $ F $ on compact sets, due to lower semicontinuity.  Namely, the following properties hold.
 \begin{proposition}[\cite{braidesLocal2014}, \cite{dalmasoIntroduction1993}]
     \label{thm:min_to_min}
    If a sequence of functionals $ \{F_n\} $ on a compact metric space $ X $  $ \Gamma $-converges to $ F $, then
    \begin{enumerate}
        \item $ F $ is lower semicontinuous and $ \min F = \lim_{n\to \infty}\inf F_n $
        \item if $ \{x_n\} $  is a sequence of (global) minimizers of $ F_n $, converging to an $ x\in X $, then $ x $ is a (global) minimizer for $ F $.
    \end{enumerate}
    If $ F_n $ is a constant sequence, $ \glim F $ is the lower semicontinuous envelope of $ F $; i.e., the supremum of lower semicontinuous functions bounded by $ F $ above.
\end{proposition}
In our setting, the underlying metric space $ X $ is $ \mathcal P(A) $, the space of probability measures on $ A $; functionals $ \e(\mu, A) $ are defined as in Section~\ref{sec:short_range} when $ \mu $ is a counting measure of some $ \omega_N $,  $ \mu = \frac1n \sum_i \delta_{x_i} $, and equal to $ +\infty $ otherwise, see Theorem~\ref{thm:gamma}.

\begin{remark} 
    The short-range property in Definition~\ref{def:short_range} can be viewed as an instance of $ \Gamma $-convergence as well. Indeed, on the one hand, an analog of \ref{it:gamma_lower} is immediate:
    \[ 
        \e(\omega_N^*(A)) \leq {\e\left(\bigcup_m \omega_{N_m}^*(A\cap Q_m),\ A\cap \Q\right)},
    \]
    where $ \sum_m N_m = N  $. As usual, we omit the reference to the underlying set for minimizers, see~\eqref{eq:omitting_set}.
    On the other, properties (i)--(ii) in Definition~\ref{def:short_range} give
    \[
        \begin{aligned}
            \lim_{N\to\infty} \frac{\e(\omega_N^*(A)\cap \Q, A\cap \Q)}{\t(N)} 
        &= \lim_{N\to\infty} \frac{\sum_{m=1}^M \e(\omega_N^*(A)\cap Q_m, A\cap Q_m)}{\t(N)} \\
        &\geq\lim_{N\to\infty} \frac{\sum_{m=1}^M \e(\omega_{N_m}^*(A\cap Q_m))}{\t(N)}\\
        &= \lim_{N\to\infty} \frac{\e\left(\bigcup_m \omega_{N_m}^*(A\cap Q_m),\ A\cap \Q\right)}{\t(N)}, 
        \end{aligned}
    \]
    assuming the cardinalities are chosen as $ N_m = \#(\omega_N^*(A)\cap Q_m) $.
    This shows the existence of recovery sequence, consisting of piecewise minimizers, similar to \ref{it:recovery_seq}.
\end{remark}

\subsection{\texorpdfstring{$ \Gamma $}{Gamma}-convergence in the short-range case}
\label{subsec:gamma_hypersingular}
Let us derive from variational principles that the limiting measure from Theorem~\ref{thm:ext_dist} is a minimizer of a functional acting on measures of the form $ \pphi\lambda_d $, given by
\[
    \s(\pphi\lambda_d; \h, \ext) := \f(1) \int_A \h(x)\,\pphi(x)^{1+\sigma}\, d\lambda_d(x) + \int_A \xi(x) \, \pphi(x)\, d\lambda_d(x).
\]
This functional coincides with the expression for asymptotics in Theorem~\ref{thm:ext_dist} when $ \pphi=d\mu/d\lambda_d(x) $. It will be shown in the proof of Theorem~\ref{thm:gamma} that it is the $ \Gamma $-limit of discrete functionals 
\begin{equation}
    \label{eq:extend_e}
    \e_N(\mu, A) = 
    \begin{cases}
        \e(\omega_N, A), & \mu \in \mathcal P_N(A),\ \supp(\mu) =\omega_N;\\
        +\infty,                                   & \text{otherwise,}
    \end{cases}
\end{equation}
where we write $ \e $ for an interaction possibly equipped with a weight and an external field.  Throughout the section, we let $ \h $, $ \xi $ be continuous.

The density $  d\mu/d\lambda_d( x) $ is given by
\begin{equation}
    \label{eq:phi_intro}
        \pphi_0(x) = \frac{d\mu}{d\lambda_d}(x) = \left(\frac{L_1-\xi(x)}{\f(1)(1+\sigma)\h(x)}\right)^{1/\sigma}_+
\end{equation}
where the constant $ L_1 $ is chosen so that $ \pphi_0\, d\lambda_d $ is a probability measure.
Let us now apply a variational argument to show that the density $ \pphi_0 $ minimizes $ \s(\pphi d\lambda_d;\h, \ext) $; for brevity, we shall write simply $ \s(\pphi;\h, \ext) $. First, observe that $ \s(\pphi;\h, \ext) $ is convex in $ \pphi $, and must therefore attain a minimum on the set $ \int_{A} \pphi d\lambda_d =1 $. Suppose $ \pphi^* $ is a minimizer, and let $ \delta\pphi $ satisfy $ \int_{A} \delta\pphi d\lambda_d  =0 $. Taking G\^ateaux derivative at $  \pphi^* $ in the direction $ \delta \pphi $ gives
\[
    \begin{aligned}
        d[\s(\pphi^*;\h, \ext); \delta \pphi]
        &= \lim_{t\to 0} \frac1t\left[S\left((\pphi^* + \delta\pphi);\h, \ext\right) -S\left(\pphi^*;\h, \ext\right) \right]  \\  
        &= \lim_{t\to 0} \frac1t
        \left[
            \int_{A} \delta\pphi( x)t\left(\f(1) (1+\sigma)\h(x)\pphi^*( x)^{\sigma} +  \ext( x)\right) \,d\lambda_d( x) + o(t)
        \right]\\
        &= \int_{A} \delta\pphi( x)
        \left[\f(1) (1+\sigma)\h(x)\pphi^*( x)^{\sigma} +  \ext( x)\right] \,d\lambda_d( x).
    \end{aligned}
\]
Since $  \pphi^*  $ is a minimizer, equality $ d[\s(\pphi^*;\h, \ext); \delta \pphi] = 0 $ must hold for every $ \delta\pphi $ such that $ \int_{A} \delta\pphi\, \,d\lambda_d  = 0 $.
This implies that the factor
\[
    f( x) := \f(1) (1+\sigma)\h(x)\pphi^*( x)^{\sigma} +  \ext( x)
\]
is a $ \lambda_d $-a.e.\ constant. Indeed, let $ \delta\pphi = f - \int_{A} f \,d\lambda_d/\lambda_d({A}) $; there holds $ \int_{A} \delta\pphi\,\,d\lambda_d = 0 $, therefore
\[
    \begin{aligned}
        0 &= \int_{A} \delta\pphi\, f \,d\lambda_d\\  
          &= \int_{A} \delta\pphi\, f \, \,d\lambda_d - \int_{A} \frac{\delta\pphi}{\lambda_d({A})} \left(\int_{A} f \,d\lambda_d \right) \,d\lambda_d\\
          &= \int_{A} \left(f - \frac1{\lambda_d({A})}\int_{A} f \,d\lambda_d \right)^2\,\,d\lambda_d,
    \end{aligned}
\]
as desired. Let $ L_1 = \int_{A} f \,d\lambda_d $ so that  $ f = L_1 $ $ \lambda_d $-a.e.; then the definition of $ f $ and \eqref{eq:phi_intro} imply $ \pphi^* = \pphi_0 $ $ \lambda_d $-a.e.

In the following proof we obtain a stronger statement than that in the formulation of Theorem~\ref{thm:gamma}. Namely, we consider an admissible interaction functional with weight and external field, and show that $ \e_N $, obtained by the extension \eqref{eq:extend_e}, $ \Gamma $-converge to $ \s(\pphi\lambda_d;\h,\xi) $.

\begin{proof}[Proof of Theorem~\ref{thm:gamma}.]
    To verify the property \ref{it:gamma_lower} of the definition of $ \Gamma $-convergence, suppose a sequence $ \{\mu_N\}\subset \mathcal P({A}) $ weak$ ^* $ converges to $ \mu\in\mathcal P({A}) $; observe that if     \begin{equation*}
        \liminf_{N\to\infty} \frac1{\t(N)} \e_N(\mu_N; \h, \ext) = +\infty \geq \s(\mu; \h, \ext),
    \end{equation*}
    the inequality in \ref{it:gamma_lower} holds trivially. It therefore suffices to assume that the limit in the last equation is finite. In particular, $ \{\mu_N\} $ must contain a subsequence comprising only elements from $ \mathcal P_N({A}) $, so passing to a subsequence if necessary we now suppose that the sequence of discrete measures $ \mu_N,\, N\geq 2, $ is such that 
    \[
        \mu_N \weakto \mu, \qquad \mu_N\in \mathcal P_N({A}).
    \]
    Since $ \sigma > 0 $, this implies that $ \mu $ must be absolutely continuous with respect to $ \lambda_d $ for the above limit to be finite.

    We first assume that the the density $ d\mu/d\lambda_d $ is continuous. 
    By Theorem~\ref{thm:ext_dist}, there exists a function $ \ext^* $ such that the minimizers of $ \e_N(\cdot\,; \h, \xi^*) $ also converge to $ \mu $; denote them by $ \mu^*_N,\, N\geq 2 $. 
    It also follows from Theorem~\ref{thm:ext_dist} that for the sequence $ \{\mu_N^*\} $, there holds 
    \[ 
        \frac1{\t(N)} \e_N(\mu_N^*; \h, \ext^*) \to \s(\mu\,; \h, \ext^*), \qquad N\to \infty.  
    \]
    Note that both $ \ext $ and $ \ext^* $ are continuous; the latter follows from the continuity of the density of $ \mu $. Using this together with the weak$ ^* $ convergence of the two sequences $ \mu_N,\, \mu_N^* $ to $ \mu $, we can conclude
    \[
        \begin{aligned}
            \liminf_{N\to\infty} \frac1{\t(N)} \e_N(\mu_N; \h, \ext)
            &= \liminf_{N\to\infty} \left( \frac1{\t(N)} \e_N(\mu_N; \h, \ext^*) + \int_{A} (\ext - \ext^*)\d\mu_N \right) \\
            &\geq \lim_{N\to\infty} \frac1{\t(N)} \e_N(\mu^*_N; \h, \ext^*) + \int_{A} (\ext - \ext^*)\d\mu \\
            &=  \s(\mu\,; \h, \ext^*) + \int_{A} (\ext - \ext^*)\d\mu = \s(\mu\,; \h, \ext).
        \end{aligned}
    \]
    This gives property \ref{it:gamma_lower} in the definition of $ \Gamma $-convergence.
    
    To establish property $ \ref{it:recovery_seq} $, first note that taking a sequence of minimizers such as $ \mu_N^* $ to be the recovery sequence immediately implies $ \ref{it:recovery_seq} $  for any $ \mu $ such that $ \s(\mu\,; \h, \ext) < +\infty $. On the other hand, if $ \s(\mu\,; \h, \ext) = +\infty $, then by lower semicontinuity
    \[
        \lim_{N\to\infty} \frac1{\t(N)} \e_N(\mu_N; \h, \ext) = +\infty = \s(\mu\,; \h, \ext)
    \]
    for every sequence $ \{ \mu_N \} $ of probability measures weak$ ^* $ converging to $ \mu $.

    The case of a general $ d\mu/d\lambda_d $ without the continuity assumption follows from the measures with continuous Radon-Nikodym derivatives being dense in $ \mathcal P(A) $, and an approximation argument.

    The above proof applies to the case $ \sigma < 0 $ as well. It should be noted that for such $ \sigma $, $ \s $ is also defined for measures, singular with respect to $ \lambda_d $. Since $ \h $, $ \xi $ are assumed to be continuous, constructing a recovery for such measures involves an approximation argument as well.
\end{proof}
\begin{remark}
    \label{rem:uniqueness}
    The above proof gives in particular that any sequence $ \{\mu_N\} \subset \mathcal P({A}) $ with the optimal asymptotics
    \[ 
        \lim_{N\to\infty}\frac1{\t(N)} \e_N(\mu_N; \h, \ext) = \min_\mu \s(\mu; \h, \ext) 
    \]
    must converge weak$ ^* $ to the measure $ \mu $ from Theorem~\ref{thm:ext_dist}. This follows from the convexity of the functional $ \s(\cdot\,; \h, \ext) $, by which it has a unique minimizer, and property \ref{it:gamma_lower} of the $ \Gamma $-convergence.
\end{remark}

\section{Further directions}
\label{sec:further}
In this section we outline some possible further directions for exploring short-range interactions. Broadly, one may ask: what are some short-range interactions commonly encountered in physics and engineering applications, and what can be proved for them? Are there long-range interactions that can be reduced to the short-range methods?
\medskip

\noindent{\bf Different homogeneous dependence in the short-range equation.}
The short-range property in Definition~\ref{def:short_range} was given using a sum in the numerator of equation~\eqref{eq:short_range}. It is interesting to consider other dependencies of $ \e(\omega_N^*(A)) $ on the parts of the set $ A $. For example, in a joint work with Petrache, the first two authors studied the problem of optimal unconstrained polarization, that is, finding a configuration $ \omega_N\subset \mathbb R^d $ for which
\[
    \mathcal P_s(\omega_N,A) =  \min_{y\in A} \sum_{i=1}^N \|y-x_i\|^{-s}, \qquad s > d,
\]
is maximal. It turns out, for this functional, an analog short-range equation holds with minimum in place of summation in the numerator of~\eqref{eq:short_range}. One may ask further if there are functionals of interest, in which the dependence is given by another 1-homogeneous function of the values of $ \e $ on subsets of $ A $.

\medskip 

\noindent {\bf Other short-range interactions with scale-invariance.}
One can study functionals of the form
\[
    \e(\omega_N, A) = \sum_{i=1}^N \e_0(\omega_{N,i}, V_i),
\]
where $ \omega_{N,i}= \{ x_j : j\in I_{i,k} \} $ denotes the subset of the $ k $ nearest neighbors to $ x_i\in\omega_N $, and $ V_i $ is the Voronoi cell on $ A $, corresponding to $ x_i $. The functional $ \e_0 $ can be assumed to be $ s $-scale invariant (see Section~\ref{sec:cube} for the definition) and does not depend on the underlying set. Such $ \e $ will then be scale-invariant and have the short-range property. The set-continuity~\eqref{eq:cube_conts} will be satisfied under some continuity assumptions on $ \e_0 $.

\medskip 

\noindent {\bf Short-range interactions without scale-invariance.}
Short-range functionals without scale invariance are extensively used for meshing algorithms, and in physics. As an example, one can consider the meshing algorithm due to Shimada and Gossard~\cite{Shimada1995}, in which interaction between nearby elements of $ \omega_n $ is given by a spline. It is natural to expect that the short-range property should suffice for the uniform distribution of the optimizers of such interactions; the study of asymptotics of the optimal values appears to be more challenging.

Lastly, it would be of interest to study the asymptotics of mixtures of short-range interactions. The most famous example of such is perhaps the Lennard-Jones potential, for which even defining the asymptotic behavior is nontrivial.

\bibliographystyle{acm}
\bibliography{refs}

\vskip\baselineskip

{\footnotesize
    {\sc Center for Constructive Approximation\\ \indent Department of Mathematics, Vanderbilt University, Nashville, TN, 37240}

    {\it Email address:} {\tt doug.hardin@vanderbilt.edu}

    {\it Email address:} {\tt edward.b.saff@vanderbilt.edu}

    \medskip
    {\sc Department of Mathematics, Florida State University, Tallahassee, FL 32306}

    {\it Email address:} {\tt ovlasiuk@fsu.edu}

}


\end{document}